\documentclass[10pt,english]{amsart}
\usepackage{amsmath,amssymb,amsthm,amsfonts,graphicx,color}
\usepackage{amssymb,geometry}

\usepackage{color, graphics}

\usepackage{graphicx}
\usepackage[T1]{fontenc}
\usepackage[latin1]{inputenc}
\usepackage[english]{babel}
\usepackage{geometry}
\usepackage{indentfirst} % RIENTRO PRIMO PARAGRAFO

\usepackage{cases}

\usepackage[colorlinks=true]{hyperref}

\makeatletter
\def\@currentlabel{2.1}\label{e:dispaa}
\def\@currentlabel{2.21}\label{e:dispau}
\def\@currentlabel{2.22}\label{e:dispav}
\def\@currentlabel{2.23}\label{e:dispaw}
\def\@currentlabel{2.24}\label{e:dispax}
\def\theequation{\thesection.\@arabic\c@equation}
\makeatother

\usepackage[colorlinks=true]{hyperref}
\hypersetup{linkcolor=black,citecolor=black,filecolor=black,urlcolor=black} % black links, for printed output

\newcommand{\N}{\mathbb{N}}

\newcommand{\Q}{\mathbb{Q}}
\newcommand{\R}{\mathbb{R}}

\newcommand{\D}{\mathbb{D}}

\newcommand{\diver}{\text{div}}
\newcommand{\tr}{\text{tr}}

\newcommand{\eps}{\varepsilon}

\newcommand{\ve}{\varepsilon}

\renewcommand{\theequation}{\thesection.\arabic{equation}}
 
 \newtheorem{lemma}{Lemma}[section]
\newtheorem{definition}{Definition}[section]
\newtheorem{theorem}{Theorem}[section]
\newtheorem{proposition}{Proposition}[section]

\newtheorem{remark}{Remark}[section]
\newcommand{\bremark}{\begin{remark} \em}
\newcommand{\eremark}{\end{remark} }

\begin{document}

\title[Delaunay ends solutions of the Cahn-Hilliard equation]{Multiple Delaunay ends solutions of the Cahn-Hilliard equation}
%\date{\today}

\author{Micha{\l } Kowalczyk}
\address{Departamento de Ingenier\'{\i}a Matem\'atica and Centro
de Modelamiento Matem\'atico (UMI 2807 CNRS), Universidad de Chile, Casilla
170 Correo 3, Santiago, Chile.}
\email {kowalczy@dim.uchile.cl}
\author{Matteo Rizzi}
\address{Centro
de Modelamiento Matem\'atico (UMI 2807 CNRS), Universidad de Chile, Casilla
170 Correo 3, Santiago, Chile.}
\email{mrizzi1988@gmail.com}
\thanks{M. Kowalczyk was partially supported by Chilean research grants Fondecyt 1130126 and 1170164 and Fondo Basal AFB170001 CMM-Chile.  M. Rizzi was partially supported by Fondecyt postdoctoral research grant 3170111 and Fondo Basal AFB170001 CMM-Chile.}

\begin{abstract}
Let $\Sigma$ be a surface of constant mean curvature in $\R^3$ with multiple Delaunay ends.
Assuming that $\Sigma$ is non degenerate in this paper we construct new solutions to the Cahn-Hilliard equation $\ve\Delta u+\ve^{-1}u(1-u^2)=\ell_\ve$ in $\R^3$  such that as $\ve\to 0$ the zero level set of $u_\ve$ approaches $\Sigma$. Moreover, on compacts of the connected components of $\R^3\setminus \Sigma$ we have $1-|u_\ve|\to 0$ uniformly. 
\end{abstract}

\maketitle

\tableofcontents

\section{Introduction}

\setcounter{equation}{0}

In this paper we consider the Cahn-Hilliard equation
\begin{eqnarray}
\eps\Delta u+\frac{1}{\eps}(u-u^3)=\ell_\eps\label{cahn_hilliard}
\end{eqnarray}
in $\R^3$, which arises from phase transitions theory. If, for instance, two liquids are mixed in a bounded container $\Omega$ and $u(x)$ is the density of one of the two at a point $x\in\Omega$, we expect the optimal configuration to minimise an energy, which, at first glance, can be taken to be
$$\int_\Omega W(u) dx,\qquad W(u)=\frac{(1-u^2)^2}{4}.$$
However, this naive model reveals to be unsatisfactory, since any piecewise constant function taking only the values $\pm 1$ would be a minimiser. Therefore it looks convenient to add a gradient term to the energy, in order to penalise the transition between the two phases represented by $-1$ and $1$. More precisely, we consider the energy
\begin{equation}
\label{allen-cahn_energy}
E_\eps(u,\Omega):=\int_\Omega \bigg(\frac{\eps}{2}|\nabla u|^2+\frac{1}{\eps}W(u) \bigg)  dx.
\end{equation}
It turns out that equation (\ref{cahn_hilliard}) on a domain $\Omega$ is the Euler equation of $E_\eps(\cdotp,\Omega)$ under the mass constraint
\begin{equation}
\label{mass_constraint}
\frac{1}{|\Omega|}\int_\Omega u(x)dx=m, \qquad m\in(-1,1).
\end{equation}
Modica proved that, if $\eps_k$ is a sequence of positive numbers tending to $0$ and $u_{\eps_k}$ is a sequence of minimisers of $E_{\eps_k}(\cdotp,\Omega)$ under the aforementioned constraint (\ref{mass_constraint}) such that $u_{\eps_k}\to u_0$ in $L^1(\Omega)$, then $u_0(x)\in\{\pm 1\}$ for almost every $x\in\Omega$, and the boundary in $\Omega$ of the set $E:=\{x\in\Omega:\, u_0(x)=1\}$ has minimal perimeter among all subsets $F\subset \Omega$ such that $|F|=|E|$, where $|\cdotp|$ denotes the volume (see \cite{M}, Theorem $1$).  For further $\Gamma$-convergence results relating the Ginzburg-Landau energy $E_\eps$ to the perimeter, we refer to \cite{MM}. In view of this results, we are lead to think that, for $\eps$ small, the interface of the minimisers $u_\eps$ resemble minimal surfaces. Conversely, there are several results in the literature where the authors construct families $\{u_\eps\}_{\eps\in(0,\eps_0)}$ of solutions to the Allen-Cahn equation
\begin{eqnarray}
\eps\Delta u_\eps+\frac{1}{\eps}(u_\eps-u_\eps^3)=0,\label{allen_cahn}
\end{eqnarray}
not necessarily minimisers, whose nodal set approaches, for $\eps$ small, some given minimal surface. Some relevant results in this direction are obtained, for instance, in \cite{DKPW,DKW,KLPW,PR}.\\

In our case, the presence of a nontrivial Lagrange multiplier suggests that, in order to construct solutions to (\ref{cahn_hilliard}), it is convenient to start from a constant mean curvature surface instead of a minimal surface. It is known that the only compact embedded constant mean curvature surface is the round sphere, while, as regards the non compact case, the simplest example is the cylinder. Moreover, a relevant family of rotationally symmetric embedded constant mean curvature surfaces was discovered by Delaunay, whose construction relies on rotating some given periodic graph around a fixed axis in $\R^3$ (see, for instance, \cite{MP, JP}). This family is parametrised by a real number $\tau\in(0,1]$, and usually denoted by $\{D_\tau\}_{\tau\in(0,1]}$. In the degenerate case $\tau=1$, $D_\tau$ reduces to a cylinder, while, for $\tau\in(0,1)$, there is a nontrivial curvature in the direction of the axis too. A relevant existence result was obtained by Hernandez and Kowalczyk, who proved the following.
\begin{theorem}\cite{HK1}
Let $\tau\in(0,1)$. Then there exists $\eps_\tau>0$ such that, for any $0<\eps<\eps_\tau$, there exists a solution $u_{\tau,\eps}$ to (\ref{cahn_hilliard}) such that 
\begin{equation}
u_{\tau,\eps}\to\pm 1 \qquad\text{as $\eps\to 0$, uniformly in $\Omega^\pm_\tau$,}
\end{equation}
where $\Omega^\pm_\tau$ denote the interior and the exterior of $D_\tau$ respectively.
\label{th_KH}
\end{theorem}
Here we are interested in a similar construction, based on another kind of embedded constant mean curvature surfaces, that is we start from an arbitrary element $\Sigma$ in the set $\mathcal{M}_{k,g}$ of complete Alexandrov embedded constant mean curvature surfaces of genus $g$ with $k$-ends. This means that our surface is given, outside a large ball $B_R$, by the disjoint union of $k$ connected components $E_j$ called \textit{the ends}, that is
$$\Sigma\cap(\R^3\backslash B_R)=\cup_{j=1}^k E_j.$$
There are several results in the literature about theses surfaces, for instance it is known that $\mathcal{M}_{0,g}$ only consists of the round sphere (since, by definition, any surface in $\mathcal{M}_{0,g}$ has to be compact), $\mathcal{M}_{1,g}$ is empty and $\mathcal{M}_{2,g}$ only consists of Delaunay surfaces, while the situation is highly nontrivial if we have $k\ge 3$ ends (see for instance \cite{K,KMP,JP}). A very important general fact is that every end is exponentially close to a translated and rotated copy of some Delaunay surface (see \cite{KMP}), whose parameter will be denoted by $\tau_j$ and whose axis is spanned by a unit vector, which will be denoted by $\mathbf{c}_j$. Note that, in general, we have two possible choices of $\mathbf{c}_j$: in the paper, we will always assume that the orientation of $\mathbf{c}_j$ agrees with the one of the end, that is  
\begin{equation}
\mathbf{c}_j\cdotp x>0,  \qquad\forall\,x\in E_j,\,\forall\,1\le j\le k. 
\end{equation}
In our construction, we will make the assumption\\

(H) not two of the ends are parallel,\\

in the sense that, if there exist $1\le j\ne i\le k$ and $\lambda\in\R$ such that $\mathbf{c}_j=\lambda\mathbf{c}_i$, then $\lambda<0$. This assumption is equivalent to say that, given two ends $E_i$ and $E_j$, $i\ne j$, it is always possible to find two disjoint open cones $C_j$ and $C_i$ such that $E_j\subset C_j$ and $E_i\subset C_i$. Moreover, one expects some \textit{non degeneracy} assumption about the surface to be necessary, in the sense that the Jacobi operator  
$$\mathcal{J}_\Sigma:=\Delta_\Sigma+|A_\Sigma|^2$$
has no $L^2(\Sigma)$-kernel. In other words, defining a \textit{Jacobi field} to a be nontrivial solution to 
$$\mathcal{J}_\Sigma \phi=0,$$
non degeneracy is equivalent to say that there are no Jacobi fields in $L^2(\Sigma)$. This hypothesis, which is common to most of the papers dealing with such constructions (see \cite{DKPW,DKW,HK1,KLPW,PR}), is of course required here, however we will see that it is not sufficient. Additionally, we need every end to be \textit{regular}. In order to explain this notion, we observe that, for any $\Sigma\in\mathcal{M}_{g,k}$, it is known that any of the ends gives rise to at least $5$ \textit{globally defined} linearly independent Jacobi fields with at most mild exponential growth along the ends (for a precise definition of what mild exponential growth means, we refer to section \ref{section_non_deg}). In view of this fact, we say that an end is regular if there are exactly $6$ linearly independent Jacoby fields of $\Sigma$ associated to it.\\

Before stating our main result, we introduce some notation. We observe that $\Sigma$ divides $\R^3$ into two connected components, that we will denote by $\Sigma^\pm$, the interior and the exterior respectively. For any multi index $\beta:=(\beta_1,\beta_2,\beta_3)\in\N^3$, we set
$$\partial^\beta_x:=\partial^{\beta_1}_{x_1}\partial^{\beta_2}_{x_2}\partial^{\beta_3}_{x_3}.$$%\qquad |\beta|:=\beta_1+\beta_2+\beta_3.$$
\begin{theorem}
Let $\Sigma$ be a non compact non degenerate complete Alexandrov embedded constant mean curvature surface in $\R^3$ with $k$ ends. Assume furthermore that every end is regular and (H) is satisfied. Then there exists $\eps_0>0$ such that, for any $0<\eps<\eps_0$, there exists a solution $u_\eps$ to equation (\ref{cahn_hilliard}) such that
$$u_\eps\to\pm 1, \qquad \partial^\beta_x u_\eps \to 0 \qquad\text{as $\eps\to 0$}$$
uniformly on compact subsets of $\Sigma^\pm$, for any multi index $\beta$, with $1\le|\beta|\le 2$. 
\label{main_th}
\end{theorem}
As regards the Lagrange multiplier, it follows from the construction that, as the case considered in \cite{HK1},
$$\ell_\eps=-\frac{1}{2}H_\Sigma\int_\R (v'_\star(t))^2 dt+O(\eps)<0.$$
Our assumptions are fulfilled by a large class of surfaces. In fact, in the previous notations, the \textit{balancing formula}
\begin{equation}
\label{balancing_formula}
\sum_{j=1}^k \tau_j^2\mathbf{c}_j=0
\end{equation}
holds (see \cite{JP}). Moreover, it is known that, if $\Sigma\in \mathcal{M}_{g,k}$ is contained in a cylinder, then $k=2$ and hence $\Sigma$ is a Delaunay surface. This fact, together with the balancing formula, yields that, if $k=3$, assumption (H) is automatically satisfied. In fact, if two of the ends were parallel, then, by the balancing formula, $\Sigma$ would be contained in a cylinder, which is not possible. For $k\ge 3$, some examples of surfaces fulfilling our hypothesis are known. For instance, the construction in \cite{JP} relies on the existence of a family of $k$-ended surfaces whose ends are planar (that is all the $\mathbf{c}_j$'s are orthogonal to a fixed vector) and such that the angle between an end the following is $2\pi/k$. What we get is a one-parameter family of surfaces, parametrised by the Delaunay parameter $\tau$ of the Delaunay surfaces to which the ends are asymptotic. In this case the parameter $\tau$ is the same for all the ends, since they only differ by a rotation of angle $2\pi/k$. It is known that, for any $k\ge 3$, these surfaces are non degenerate, all the ends are regular and (H) is satisfied by construction. In the case $k=3$, playing with the angle $\alpha$ between two of the ends and with the Delaunay parameter $\tau$ of one of them, it is also possible to construct another family of planar surfaces, depending on the parameters $\tau$ and $\alpha$, in which only two of the ends have the same Delaunay parameter, the other being determined by (\ref{balancing_formula}). Even in this case, we get a family of surfaces which fulfil all our assumptions.\\

Our proof relies on perturbation techniques, such as fixed point theorems and the infinite-dimensional Lyapunov-Schmidt reduction, iterated $k$ times, for any end, in order to improve the error of approximation, and then applied again to an equation which sees the whole surface, which can be solved thanks to the non degeneracy of $\Sigma$ and the regularity of the ends. Moreover, we will need a very small global correction, whose existence is proved thanks to assumption (H), which enables us to use coercivity. \\ %This assumption looks a bit technical, however it is quite general, due to the previous discussion.

Similar arguments were applied, for instance, in (\cite{KLPW}), where a family of solutions to the Allen-Cahn equation is constructed. However their proof relies on a gluing procedure, which starts from a given, known, family of $4$-ended solutions to the Allen-Cahn equation. A similar end-to-end construction is used, for instance, in \cite{JP}, in a geometric context, were new constant mean curvature surfaces are produced by gluing together known surfaces. Here our technique is slightly different, since our proof is self-contained, in the sense it does not rely on any known solution to our PDE, although it turns out to be asymptotic, close to any of the ends, to a translated and rotated copy of a solution constructed by Hernandez and Kowalczyk in \cite{HK1}. \\

The plan of the paper is the following: in Section \ref{section_surface} we explain the geometric background of the paper, that is, in Section \ref{section_Delaunay}, we recollect the basic information that we need about Delaunay surfaces and their Jacobi fields and, in Section \ref{section_non_deg}, we give a detailed explanation of what we mean by non degeneracy. Then, in Section \ref{section_SIgma_d}, we construct a family of auxiliary surfaces, which will be the nodal set of our approximate solutions. In Section \ref{section_Fermi} we introduce the Fermi coordinates.

In Section \ref{section_approx_solution} we construct a family of approximated solutions. In particular we will see that the first approximation constructed in Section \ref{section_first_approximation} is not enough to solve our equation, thus, in Section \ref{section_improvement}, we overcome this technical issue by adding some suitable corrections, whose explicit construction is carried out in Sections \ref{section_aux_end} and \ref{section_bifo_end}, thanks to a Lyapunov-Schmidt reduction. 

Finally, Section \ref{section_proof} is devoted to the construction of a family of global approximations, through a gluing procedure, carried out in Section \ref{section_gluing}, which enables us to find a true solution of our equation, by adding a small global approximation, studied in detail in Section \ref{subsec_far}, and by solving a problem near $\Sigma$, in Section \ref{section_near}.

\section{About the surface}\label{section_surface}

\setcounter{equation}{0}

\subsection{Delaunay surfaces}\label{section_Delaunay}

The aim of this subsection is to recall the main properties of the Delaunay surfaces, that is a family of noncompact complete constant mean curvature surfaces in $\R^3$ obtained by rotating a periodic curve around a fixed axis, that we assume to be the $x_3$-axis. For further details, we refer to \cite{MP}. These surfaces admit a parametrisation of the form 
$$R(x_3,\vartheta):=(\rho(x_3)\cos\vartheta,\rho(x_3)\sin\vartheta,x_3), \qquad\forall\,(x_3,\vartheta)\in\R\times S^1,$$
where $\rho:\R\to\R$ is a periodic function determined in such a way that the curvature is identically equal to $1$. This condition is equivalent to the ODE 
\begin{eqnarray}
\partial^2_{x_3}\rho-\frac{1}{\rho}(1+(\partial_{x_3}\rho)^2)+(1+(\partial_{x_3}\rho)^2)^{3/2}=0.\label{ODE_rho}
\end{eqnarray}
It is known that, for any $\tau\in(0,1)$, there exists a unique periodic solution $\rho_\tau$ to (\ref{ODE_rho}) such that
\begin{eqnarray}
\rho_\tau(0)=\epsilon:=1-\sqrt{1-\tau^2},\qquad\partial_{x_3}\rho_\tau(0)=0, \qquad\epsilon\le \rho_\tau\le 2-\epsilon. 
\end{eqnarray}
We denote its period by period $T_\tau$. Therefore, we have a family of constant mean curvature surfaces, known as \textit{Delaunay surfaces}, that we denote by $D_\tau$. %Another possible choice would be to parametrise them with $\epsilon\in(0,1)$, which is related to $\tau$ by formula $\tau(\epsilon)=\sqrt{\epsilon(2-\epsilon)}$. 
In the sequel, we will be interested in the Jacobi operator of $D_\tau$. This operator is defined in a variational way, in fact it appears as the linearisation of the functional
$$v\mapsto H_{\Sigma_v},$$
where $v:D_\tau\to\R$ is a real-valued function and $H_{\Sigma_v}$ is the mean curvature of the normal graph
$$\Sigma_v:=\{p+v(p)\nu_\tau(p):\, p\in D_\tau\},$$
$\nu_\tau(p)$ being the inward-pointing unit normal to $D_\tau$ at $p$. In other words, assuming, for instance, that $v\in C^2(D_\tau)$, the Jacobi operator applied to $v$ is defined by the relation
$$\mathcal{J}_{D_\tau}v:=\frac{d}{dt}\bigg\arrowvert_{t=0}H_{\Sigma_{tv}}.$$
It turns out that   
\begin{eqnarray}\notag
\mathcal{J}_{D_\tau}=\Delta_{D_\tau}+|A_{D_\tau}|^2,
\end{eqnarray}
where $\Delta_{D_\tau}$ is the Laplace-Beltrami operator and $|A_{D_\tau}|^2$ is the squared norm of the second fundamental form, that is the sum of the squares of the principal curvatures of $D_\tau$. The Jacobi operator is particularly simple if we introduce isothermal coordinates on $D_\tau$, that is we use the parametrisation
$$X_\tau(s,\vartheta):=(\tau e^{\sigma_\tau(s)}\cos\vartheta,\tau e^{\sigma_\tau(s)}\sin\vartheta, k_\tau(s)),\qquad\forall \, (s,\vartheta)\in \R\times S^1$$
where $\sigma_\tau$ and $k_\tau$ are defined by the relations
\begin{eqnarray}
(1+(\partial_{x_3}\rho_\tau\circ k_\tau)^2)\partial_s k_\tau^2=\rho_\tau\circ k_\tau ,\qquad\partial_{s}k_\tau>0,\qquad 
\tau e^{\sigma_\tau}=\rho_\tau\circ k_\tau.\label{def_sigma_k}
\end{eqnarray}
We note that $\sigma_\tau$ is periodic of period $s_\tau$, satisfying $T_\tau=\frac{1}{2}k_\tau(s_\tau)$, and $k_\tau$ is strictly increasing with linear growth, since 
$$\partial_s k_\tau=\frac{\tau^2}{2}(1+e^{2\sigma_\tau})$$
is periodic. In these coordinates, the metric and the second fundamental form are given by
\begin{eqnarray}
g_{D_\tau}=(\tau e^{\sigma_\tau})^2(ds^2+d\vartheta^2), \qquad A_{D_\tau}=-\frac{\tau e^{\sigma_\tau}\partial_{ss}\sigma_\tau}{\sqrt{1-\sigma_\tau^2}}ds^2+\tau e^{\sigma_\tau}\sqrt{1-\sigma_\tau^2}d\vartheta^2,\label{g_A_Delaunay}
\end{eqnarray}
hence the Jacobi operator reduces to
\begin{eqnarray}
\tilde{L}_\tau=\frac{1}{\tau^2 e^{2\sigma_\tau}}L_\tau, \qquad L_\tau:=\partial^2_s+\partial^2_\vartheta+\tau^2\cosh(2\sigma_\tau).\label{def_jacobi_operator}
\end{eqnarray}
For further details about these formulas, we refer to \cite{MP}.

\subsubsection{Jacobi fields of $D_\tau$}\label{section_Jacobi_Delaunay}
The \textit{Jacobi fields} of $D_\tau$, that is the solutions to the homogeneous equation $\mathcal{J}_{D_\tau}\phi=0$, are of special interest in this paper. %For $j\in\Z$, we introduce the functions
%\begin{eqnarray}
%\chi_j(\vartheta):=
%\begin{cases}
%(1/\sqrt{\pi})\sin(j\vartheta), &\text{for $j<0$,}\\
%1/\sqrt{2\pi}, &\text{for $j=0$,}\\
%(1/\sqrt{\pi})\cos(j\vartheta), &\text{for $j>0$,}
%\end{cases}
%\end{eqnarray}
%which constitute an orthonormal system in $L^2(S^1)$. 
It is known that there are exactly $6$ linearly independent Jacobi fields with at most exponential growth strictly smaller than $\cosh^{\sqrt{2+\tau^2}}(s)$. %We will denote them by $\Psi_\tau^{T,e_l}$, $1\le l\le 3$, $\Psi_\tau^{R,e_e}$, $1\le l\le 2$, and $\Psi_\tau^D$., while their expression in isothermal coordinates will be $\Phi^\bullet_\tau$.
\begin{lemma}[\cite{MP}]
Let $a\in(0,\sqrt{2+\tau^2})$ and let $\Phi$ be a solution to $L_\tau \Phi=0$ such that $\Phi\cosh^{-a}(s)\in L^\infty(\R\times S^1)$. Then 
$$\Phi\in \mathop{span}\{\Phi_\tau^{T,e_l},\, 1\le l\le 3, \,\Phi_\tau^{R,e_l},\, 1\le l\le 2, \,\Phi_\tau^{D} \},$$ 
where
\begin{eqnarray}\notag
\Phi^{T,e_3}_\tau=\partial_s \sigma_\tau,\qquad \Phi^{D}_\tau=\sqrt{1-\tau^2}\bigg(\frac{1}{\tau}\partial_s \sigma_\tau\partial_\tau k_\tau-e^{\sigma_\tau}\cosh\sigma_\tau(1+\tau\partial_\tau\sigma_\tau)\bigg),\\\notag
\Phi^{T,e_1}_\tau=-\tau\cosh\sigma_\tau \cos\vartheta,\qquad \Phi^{R,e_1}_\tau=-\tau k_\tau(\cosh\sigma_\tau+\partial_s \sigma_\tau e^{\sigma_\tau})\cos\vartheta,\\\notag
\Phi^{T,e_2}_\tau=-\tau\cosh\sigma_\tau \sin\vartheta, \qquad \Phi^{R,e_2}_\tau=-\tau k_\tau(\cosh\sigma_\tau+\partial_s \sigma_\tau e^{\sigma_\tau})\sin\vartheta.
\end{eqnarray} 
\label{lemma_Jacobi_fields}
\end{lemma}
We note that, since $\sigma_\tau$ is periodic and $k_\tau$ is strictly increasing, due to (\ref{def_sigma_k}), the Jacobi fields $\Phi^{T,e_l}_\tau$ are bounded and periodic, for $1\le l\le 3$, while $\Phi^{R,e_l}_\tau$, $l=1,2$, and $\Phi^D_\tau$ are linearly growing in $s$, hence they are not in $L^2(\R\times S^1)$. The Jacobi fields have a quite explicit geometric meaning, that explains our notations. In fact, 
$\Phi^{T,e_l}_\tau$ arises from a translation along the $x_l$-axis, $\Phi^{R,e_l}_\tau$ arises from a rotation of the $x_3$-axis toward the $x_l$ axis and $\Phi^D_\tau$ arises from changing the Delaunay parameter, in the sense that will explain below.\\ 

For $\eta>0$ (small) and $\mathbf{a}$ in the ball $B_\eta\subset\R^6$, we will use the notation $\mathbf{a}=(\mathbf{a}^T,\mathbf{a}^R,\mathbf{a}^D)\in \R^3\times\R^2\times\R$, in order to underline that $\mathbf{a}^T=(\mathbf{a}^{T_1},\mathbf{a}^{T_2},\mathbf{a}^{T_3})$ is related to translations, $\mathbf{a}^R=(\mathbf{a}^{R_1}, \mathbf{a}^{R_2})$ is related to rotations and $\mathbf{a}^D$ is related to changing the Delaunay parameter. For $\mathbf{a}\in B_M$, we define
$$\mathcal{R}_{\mathbf{a}^{R_1}}:=\begin{bmatrix}
\cos (\mathbf{a}^{R_1}) & 0 & \sin (\mathbf{a}^{R_1}) \\
0 & 1 & 0 \\
-\sin (\mathbf{a}^{R_1}) & 0 & \cos (\mathbf{a}^{R_1})
\end{bmatrix}, \qquad
\mathcal{R}_{\mathbf{a}^{R_2}}:=\begin{bmatrix}
1 & 0 & 0 \\
0 & \cos (\mathbf{a}^{R_2}) & -\sin (\mathbf{a}^{R_2}) \\
0 & \sin (\mathbf{a}^{R_2}) & \cos (\mathbf{a}^{R_2})
\end{bmatrix}
$$
to be small rotations of angles $\mathbf{a}^{R_l}$ toward the axes $x_l$, $l=1,2$, and we set $\mathcal{R}_{\mathbf{a}^R}=\mathcal{R}_{\mathbf{a}^{R_1}}\circ\mathcal{R}_{\mathbf{a}^{R_2}}$. For $\mathbf{a}\in B_\eta$, with $\eta$ small enough, we define the new Delaunay surface $D_\tau(\mathbf{a})$ to be the image of the Delaunay surface $D_{\tau(\epsilon+\mathbf{a}^D)}$ under the rotation $\mathcal{R}_{\mathbf{a}^R}$ composed with the translation $x\mapsto x+\mathbf{a}^T$. We recall that $\epsilon$ is such that $\epsilon=1-\sqrt{1-\tau^2}$, thus the Delaunay parameter of $D_\tau(\mathbf{a})$ is not the same as the one of $D_\tau$. With this notation, $D_\tau=D_\tau(0)$. We denote its isothermal parametrisation of $D_\tau(\mathbf{a})$ by $X_\tau(\mathbf{a})$. It is possible to prove that, if $\eta$ is small enough and we restrict ourselves to a compact subset, then the new surface $D_\tau(\mathbf{a})$ can be locally seen as the normal graph over $D_\tau$ of a function which, in isothermal coordinates, is given by
\begin{eqnarray}
\Phi_\tau(\mathbf{a}):=\sum_{l=1}^3 \mathbf{a}^{T_l}\Phi^{T,e_l}_\tau+\sum_{l=1}^2 \mathbf{a}^{R_l}\Phi^{R,e_l}_\tau+\mathbf{a}^D\Phi^D_\tau\label{def_Phi_tau},
\end{eqnarray}
plus a term of order $|\mathbf{a}|^2$. More precisely, we have the following result.
\begin{lemma}
Let $s_1,s_2\in\R$, $s_1<s_2$ and $\tau\in(0,1)$. Then there exists $\eta>0$ depending on $s_2-s_1$ such that, for any $\mathbf{a}\in B_{\eta}$, there exist $s'_1=s'_1(\mathbf{a}),s'_2=s'_2(\mathbf{a})\in\R$, $s'_1<s'_2$, such that for any $y\in\mathring{D}_\tau(\mathbf{a})(s_1,s_2):=X_\tau(\mathbf{a})((s_1,s_2)\times S^1)$, there exists a point $y'(\mathbf{a})\in\mathring{D}_\tau(s'_1,s'_2):=X_\tau((s'_1,s'_2)\times S^1)$ and a function $w(\mathbf{a})\in C^\infty(\mathring{D}_\tau(s'_1,s'_2))$ such that
\begin{eqnarray}
y=y'(\mathbf{a})+w(\mathbf{a})(y'(\mathbf{a}))\nu_\tau(y'(\mathbf{a})),\label{wa_normal_graph}
\end{eqnarray}
Moreover, $w(\mathbf{a})$ is of the form
\begin{eqnarray}
w(\mathbf{a})=\sum_\bullet\mathbf{a}^\bullet_\tau\frac{\partial w}{\partial\mathbf{a}^\bullet_\tau}(0)+\psi(\mathbf{a}),\label{exp_wa}
\end{eqnarray}
where $\frac{\partial w}{\partial\mathbf{a}^\bullet_\tau}(0)$ are Jacobi fields of $\mathcal{J}_{D_\tau}$ and
\begin{eqnarray}
\begin{aligned}
\|\psi(\mathbf{a})\|_{C^\infty(\mathring{D}_\tau(s'_1,s'_2))}&\le C|\mathbf{a}|^2,\qquad\forall \mathbf{a}\in B_M,\\
\|\psi(\mathbf{a}^1)-\psi(\mathbf{a}^2)\|_{C^\infty(\mathring{D}_\tau(s_1,s_2))}&\le C|\mathbf{a}^1-\mathbf{a}^2|, \qquad\forall \mathbf{a}^1,\mathbf{a}^2\in B_\eta.
\end{aligned}\label{size_wa}
\end{eqnarray}
\label{lemma_perturbe_D}
\end{lemma}
\begin{proof}
It follows from the geometry that (\ref{wa_normal_graph}) holds true, and $w(\mathbf{a}):\mathring{D}_\tau(s_1,s_2)\to\R$ is an unknown function, which depends on the parameter $\mathbf{a}$. Using that the mean curvature of $D_\tau$ is the same as the one of $D_\tau(\mathbf{a})$ and the variational definition of the Jacobi operator, we have
$$0=H_{D_\tau(\mathbf{a})}-H_{D_\tau}=\mathcal{J}_{D_\tau}w(\mathbf{a})(y'(\mathbf{a}))+\mathcal{Q}_{D_\tau}(w(\mathbf{a}),\nabla w(\mathbf{a}),\nabla^2 w(\mathbf{a}))(y'(\mathbf{a})).$$
Taking the derivative of the above expression with respect to $\mathbf{a}^\bullet$ in $\mathbf{a}=0$ and using the fact that $w(0)=0$, we have
\[0=\mathcal{J}_{D_\tau}\big(\frac{\partial w}{\partial\mathbf{a}^\bullet}(0)\big)+\frac{d}{d\mathbf{a}^\bullet}\bigg\vert_{\mathbf{a}=0}\bigg(\mathcal{Q}_{D_\tau}(w(\mathbf{a}),\nabla w(\mathbf{a}),\nabla^2 w(\mathbf{a}))(y'(\mathbf{a}))\bigg).\]
By the quadratic nature of $\mathcal{Q}_{D_\tau}$, it is possible to see that
\[\frac{d}{d\mathbf{a}^\bullet}\bigg\vert_{\mathbf{a}=0}\bigg(\mathcal{Q}_{D_\tau}(w(\mathbf{a}),\nabla w(\mathbf{a}),\nabla^2 w(\mathbf{a}))(y'(\mathbf{a}))\bigg)=0,\]
which proves that $\frac{\partial w}{\partial\mathbf{a}^\bullet}(0)$ are Jacobi fields. Therefore, taking the Taylor expansion in $\mathbf{a}$ of $w$ and using once again that $w(0)=0$, we can see that (\ref{exp_wa}) holds true with a remainder $\psi(\mathbf{a})$ satisfying (\ref{size_wa}).
\end{proof}
\begin{remark}\label{rem_Jacobi_fields}
It is proved in \cite{MP} that the expression of the Jacobi fields of $D_\tau$ in isothermal coordinates is given by $\Phi^\bullet_\tau$, defined in the statement of Lemma \ref{lemma_Jacobi_fields}, that is 
\[\frac{\partial w}{\partial \mathbf{a}^\bullet}(0)\circ X_\tau(s,\vartheta)=\Phi^\bullet_\tau(s,\vartheta), \qquad\forall\, (s,\vartheta)\in(s_1,s_2)\times S^1\]
\end{remark}
\begin{remark}\label{notation}
Above and in what follows we agree that whenever the dependence on the parameter ${\mathbf {a}}$ is indicated, as for example in $D_\tau({\mathbf{a}})$, then this dependence is carried over to the evaluation of the parameter $\tau=\tau({\mathbf{a}})$. When the parameter $\mathbf{a}$ is omitted then it is implicitly assumed that $\mathbf{a}=0$. For instance with this convention the Jacobi fields $\Phi^\bullet_\tau$ in (\ref{def_Phi_tau}) are of the original surface $D_\tau$ i.e. $\tau=\tau(0)$. Also, $X_\tau(\mathbf{a})$ denotes  the isothermal parametrisation of $D_\tau(\mathbf{a})$.
\end{remark}

\subsection{Nondegeneracy of the surface}\label{section_non_deg}
In order to explain our assumptions about the surface $\Sigma$, we give an idea of the moduli space theory. We assume that $\Sigma$ belongs to the set $\mathcal{M}_{g,k}$ of complete, non compact Alexandrov embedded constant mean curvature surfaces of genus $g$ with $k$ ends. As regards the regularity of $\Sigma$, it is not restrictive to assume that it is $C^\infty$, since any $C^{2,\alpha}$ CMC surface is $C^\infty$, due to the fact that any surface is locally a graph and the fact that any solution $u$ to mean curvature equation
\begin{eqnarray}
\diver(\frac{\nabla u}{\sqrt{1+|\nabla u|^2}})=1.\label{mean_curv_eq}
\end{eqnarray}
which is $C^{2,\alpha}$ in an open set $U\subset\R^N$ is actually $C^\infty(U)$. This is a consequence of Theorem $6.17$ of \cite{GT} and a bootstrap argument.\\

Any surface in $\mathcal{M}_{g,k}$ can be written as
$$\Sigma=\mathbf{K}\cup\big(\cup_{1\le j\le k} E_j\big),$$
where $\mathbf{K}$ is compact and each of the ends $E_j$ is asymptotic to a rotated and translated copy of the Delaunay surface $D_{\tau_j}$, in a sense that will be made precise below. Moreover, we assume that \\

(H) none of two symmetry axes of the ends are parallel (we refer to the introduction for further explanations).\\

For $a>0$, $n\in\mathbb{N}$ and $\alpha\in(0,1)$, we say that a function $w:(0,\infty)\times S^1\to\R$ is in $C^{n,\alpha}_{a}((0,\infty)\times S^1)$ if
\begin{eqnarray}
\|w\|_{C^{n,\alpha}_{a}((0,\infty)\times S^1)}:=\|e^{as}w\|_{C^{n,\alpha}((0,\infty)\times S^1)}
\end{eqnarray}
is finite. \\

%For any fixed unit vector $\mathbf{a}=(\sin\varphi_1\cos\varphi_2,\sin\varphi_1\sin\varphi_2,\cos\varphi_1)\in S^2$, $(\varphi_1,\varphi_2)\in[0,\pi)\times[0,2\pi)$, we define $D_\tau^{\mathbf{a}}$ to be the image of $D_\tau$ under a rotation $R_{(\varphi_1,\varphi_2)}$ which sends $e_3$ into $\mathbf{a}$, for instance we can take
%$$R_{(\varphi_1,\varphi_2)}:=\begin{bmatrix}
%\sin\varphi_2 & \cos\varphi_1\cos\varphi_2 & \sin\varphi_1\cos\varphi_2 \\
%-\cos\varphi_2 & \cos\varphi_1\sin\varphi_2 & \sin\varphi_1\sin\varphi_2 \\
%0 & -\sin\varphi_1 & \cos\varphi_1
%\end{bmatrix},$$
%and we denote its isothermal parametrization by $X_\tau^{\mathbf{a}}$. Moreover, for any $\mathbf{b}\in\R^3$, $D_\tau^{\mathbf{a}}+\mathbf{b}$ will denote the surface obtained by translating $D_\tau^{\mathbf{a}}$ of vector $\mathbf{b}$, parametrized by $X_\tau^{\mathbf{a}}+\mathbf{b}$.
It is known that, up to a translation and a rotation, each of the ends $E_j$ admits a parametrisation of the form
\begin{eqnarray}
Y_j(s,\vartheta):=X_{\tau_j}(s,\vartheta)+v_j(s,\vartheta) N_{\tau_j}(s,\vartheta),\label{def_Yj}
\end{eqnarray} 
where $(s,\vartheta)$ are the isothermal coordinates of $D_{\tau_j}$, $N_{\tau_j}$ is the expression of the outward-pointing unit normal in isothermal coordinates and $v_j$ is a function in $C^{2,\alpha}_{\bar{a}_j}((0,\infty)\times S^1)$, where $\bar{a}_j:=\sqrt{2+\tau_j^2}$ is the corresponding indicial root (see \cite{JP}, Theorem $2.2$ of \cite{KMP} and \cite{KKS}). Moreover, each of the intersections $E_j\cap \mathbf{K}$ is homeomorphic to an annulus $(0,1]\times S^1$. %The assumption (H) means precisely that not two of the vectors $\mathbf{a}_j$ are parallel. 
\begin{remark}
The functions $v_j$ are in $C^\infty((0,\infty)\times S^1)$, since the surface $\Sigma$ is smooth and the Delaunay unduloid $D_\tau$ is also smooth, because $\rho_\tau\in C^\infty(\R)$.\label{rem_regularuty_v}
\end{remark}
%\begin{proof}
%We use the fact that every end $E_j$ is locally a graph. More precisely, for every $p\in E_j$, there exists a neighbourhood $\mathcal{U}\subset E_j$ of $p$ and an open set $U\subset \R^2$ such that $\mathcal{U}$ is the graph of a function $u:U\to\R$. Moreover, there exists an open set $V\subset (0,\infty)\times S^1$ and $C^\infty$-diffeomorphism $\Phi:V\to U$ such that
%\begin{eqnarray}
%X_j+v_j N_{\tau_j}=(\Phi,u\circ\Phi),\qquad \forall (s,\vartheta)\in V.\label{rel_vu}
%\end{eqnarray}
%Since $v_j\in C^{2,\alpha}_{loc}((0,\infty)\times S^1)$ and (\ref{rel_vu}) holds, then $u$ is a solution to (\ref{mean_curv_eq}) in $C^{2,\alpha}(U)$, thus, by Step \textit{(i)}, $u\in C^\infty(U)$. Moreover, since $|N_{\tau_j}(s,\vartheta)|=1$, for any $(s,\vartheta)\in (0,\infty)\times \R$, then at least one component does not vanish in $V$, provided $\mathcal{U}$ is small enough. Hence, once again by (\ref{rel_vu}), we conclude that $v_j\in C^\infty(V)$. Since $p$ is arbitrary, the statement of the Lemma is true. 
%\end{proof}

In order to solve the geometric problem, we need some non degeneracy assumption about $\Sigma$. For $a>0$, $n\in\mathbb{N}$ and $\alpha\in(0,1)$, we say that a function $v:\Sigma\to\R$ is in $\mathcal{C}^{n,\alpha}_a(\Sigma)$ if the norm
\begin{eqnarray}
\|v\|_{\mathcal{C}^{n,\alpha}_a(\Sigma)}=\|v|_{\mathbf{K}}\|_{\mathcal{C}^{n,\alpha}(\Sigma)}+\sum_{j=1}^k \|v|_{E_j}\circ Y_j\|_{C^{n,\alpha}_a((0,\infty)\times S^1)}
\end{eqnarray}
is finite.

\begin{definition}[\cite{JP}]
We say that $\Sigma$ is non degenerate if the Jacobi operator
$$\mathcal{J}_\Sigma:=\Delta_\Sigma+|A_\Sigma|^2: \mathcal{C}^{2,\alpha}_a(\Sigma) \to \mathcal{C}^{0,\alpha}_a(\Sigma)$$ 
is injective for any $a>0$ and $\alpha\in(0,1)$.
\end{definition}
Moreover, it is known that, for each end $E_j$, there exist $5$ globally defined linearly independent Jacobi fields $\Phi^{T,e_l}_{E_j}$, $1\le l\le 3$, and $\Phi^{R,e_l}_{E_j}$, $1\le l\le 2$, which, along $E_j$, are exponentially close to the Jacobi fields of the Delaunay unduloid $D_{\tau_j}$, that is
\begin{eqnarray}
\Phi^{T,e_l}_{E_j}\circ Y_j-\Phi^{T,e_l}_{\tau_j} \in C^{2,\alpha}_{\bar{a}_j}((0,\infty)\times S^1),\qquad\Phi^{R,e_l}_{E_j}\circ Y_j-\Phi^{R,e_l}_{\tau_j} \in C^{2,\alpha}_{\bar{a}_j}((0,\infty)\times S^1),\label{Jacobi_fields_Sigma}
\end{eqnarray}
for any $\alpha\in(0,1)$. Since the Delaunay surface $D_{\tau_j}$ also has another linearly independent Jacobi field $\Phi^D_{\tau_j}$, related to the variation of the Delaunay parameter, we expect the existence of additional  $k$ linearly independent Jacobi fields, globally defined on $\Sigma$, each of them exponentially close to $\Phi^D_{\tau_j}$ on $E_j$. However it is not always the case, since in general it is possible to construct such a Jacobi field just on $E_j$, but it is not necessarily globally defined.
\begin{definition}\cite{JP}
The end $E_j$ is said to be regular if there exists a Jacobi field $\Phi^D_{E_j}$, globally defined on $\Sigma$, such that
\begin{eqnarray}
\Phi^D_{E_j}\circ Y_j-\Phi^D_{\tau_j} \in C^{2,\alpha}_{\bar{a}_j}((0,\infty)\times S^1),
\end{eqnarray}
for any $\alpha\in(0,1)$.
\label{def_non_deg_end}
\end{definition}
We note that, in (\ref{Jacobi_fields_Sigma}) and in Definition \ref{def_non_deg_end}, $Y_j$ is exactly the parametrisation defined in (\ref{def_Yj}) if the axis of $E_j$ coincides with the $x_3$-axis, otherwise there are slight modifications to take into account, such as a rotation and a translation. In the sequel, we will assume that our surface is non degenerate and each end is regular. Now we will see that non degeneracy enables us to find a unique solution to the linear equation 
\begin{eqnarray}
\mathcal{J}_\Sigma h:=\Delta_\Sigma h+|A_\Sigma|^2h=g, \qquad g\in \mathcal{C}^{0,\alpha}_a(\Sigma),\label{eq_lin_geom}
\end{eqnarray}
at least for $0<a<\bar{a}:=\min_{1\le j\le k}{\bar{a}_j}$, and to give a quite precise description of this solution. The definition of the  Jacobi operator $\mathcal{J}_\Sigma$ parallels the one of the Jacobi operator of $D_\tau$. We define the $6k$-dimensional \textit{deficiency space} 
\begin{equation}
\mathcal{D}(\Sigma):=\oplus_{j=1}^k \mathop{span}\{\eta_j\Phi^{T,e_l}_{E_j}:\,1\le l\le 3\}\oplus \, \oplus_{j=1}^k \mathop{span}\{\eta_j\Phi^{R,e_l}_{E_j}:\, 1\le l\le 2\}\, \oplus \, \oplus_{j=1}^k \mathop{span}\{\eta_j\Phi^D_{E_j}\},\label{def_deficiency}
\end{equation}
where $\eta_j\in C^\infty(\Sigma)$ are cutoff functions vanishing on $\Sigma\backslash \big(\cup_{1\le j\le k} E_j\big)$ and identically equal to $1$ on $Y_j((1,\infty)\times S^1)$. The next Proposition is the key result to understand when (\ref{eq_lin_geom}) is solvable and what is the structure of the set its solutions.
\begin{proposition}[\cite{KMP}]
Let $\alpha\in(0,1)$. Assume that $\Sigma$ is non degenerate and fix $a\in(0,\bar{a})$, $\bar{a}:=\min_{1\le j\le k} \bar{a}_j$. Then the mapping
$$\mathcal{J}_\Sigma: \mathcal{C}^{2,\alpha}_a(\Sigma)\oplus \mathcal{D}(\Sigma)\to \mathcal{C}^{0,\alpha}_a(\Sigma) $$
is surjective and has a kernel of dimension $3k$. In addition, there exists a $3k$-dimensional subspace $\mathcal{K}(\Sigma)\subset \mathcal{D}(\Sigma)$ such that
$$ \mathop{ker}{\mathcal{J}_\Sigma}\subset \mathcal{C}^{2,\alpha}_a(\Sigma)\oplus \mathcal{K}(\Sigma).$$
Finally, given any $3k$-dimensional subspace $\mathcal{E}(\Sigma)\subset\mathcal{D}(\Sigma)$ such that $\mathcal{E}(\Sigma)\oplus\mathcal{K}(\Sigma)=\mathcal{D}(\Sigma)$, the mapping
$$\mathcal{J}_\Sigma: \mathcal{C}^{2,\alpha}_a(\Sigma)\oplus\mathcal{E}(\Sigma)\to \mathcal{C}^{0,\alpha}_a(\Sigma) $$
is an isomorphism.
\label{prop_non_deg}
\end{proposition}

\subsection{An auxiliary surface}\label{section_SIgma_d}
In this subsection we construct an auxiliary surface $\Sigma(\mathbf{d})$, depending on a parameter $\mathbf{d}\in (B_\eta)^k\subset \R^{6k}$, $\eta>0$ small, which will be useful to construct a family of approximate solutions to the Cahn-Hilliard equation (\ref{cahn_hilliard}). %Indeed, their nodal set will be close to $\Sigma(\mathbf{d})$ in a sense that will be clear in Section \ref{section_approx_solution}. 
As we explained above, each of the ends of $\Sigma$ is asymptotic to a rotated and translated copy of $D_{\tau_j}$, therefore after a rotation and a translation, we can assume that the axis of the end $E_j$ is the $x_3$-axis. Then we take $\mathbf{d}_j\in B_\eta\subset \R^6$ and we construct the new end $E_j(\mathbf{d}_j)$ of $\Sigma(\mathbf{d})$ in such a way that it is asymptotic to $D_{\tau_j}(\mathbf{d}_j)$. In order to do so, we introduce a smooth cutoff function $\xi:\R\to\R$, such that $\xi=0$ in $(-\infty,s_0)$ and $\xi=1$ in $(s_0+1,\infty)$, for some $s_0>0$, and, given the parametrisation of $E_j$ in (\ref{def_Yj}), we define the perturbed end through the parametrisation
$$Y_j(\mathbf{d}_j)(s,\vartheta):=(1-\xi(s))Y_j(s,\vartheta)+\xi(s)Z_j(\mathbf{d}_j)(s,\vartheta), \qquad \forall\,(s,\vartheta)\in \R\times S^1,$$
where
$$Z_j(\mathbf{d}_j)(s,\vartheta):=X_{\tau_j}(\mathbf{d}_j)(s,\vartheta)+v_j(s,\vartheta)N_{\tau_j}(\mathbf{d}_j)(s,\vartheta), \qquad \forall\,(s,\vartheta)\in \R\times S^1,$$
and $N_{\tau_j}(\mathbf{d}_j)$ is the inward-pointing normal vector to $D_{\tau_j}(\mathbf{d}_j)$ expressed in  isothermal coordinates. Roughly speaking, we cut the end $E_j$ of the original surface at level $s=s_0$ and we glue it with a new end, which is asymptotic to $D_{\tau_j}(\mathbf{d}_j)$. Doing the same with every end, we obtain a new surface 
$$\Sigma(\mathbf{d}):=\cup_{j=1}^k Y_j(\mathbf{d}_j)((0,\infty)\times S^1)\cup\bigg(\Sigma\backslash \cup_{j=0}^k E_j\bigg).$$
with $k$ ends, each of them being asymptotic to a  Delaunay surface. 

\begin{remark}
We point out that our construction guarantees that $\Sigma(0)=\Sigma$, since $Y_j(0)=Y_j$, which is crucial to guarantee the convergence to $\pm 1$ on compact sets of $\Sigma^\pm$ of the solution $u_\eps$ constructed in Theorem \ref{main_th}.
\end{remark}
This new surface $\Sigma(\mathbf{d})$ is diffeomorphic to $\Sigma$, for any $\mathbf{d}\in (B_M)^k$, and the diffeomorphism
$$\beta(\mathbf{d}):\Sigma(\mathbf{d})\to\Sigma$$
is given by
\begin{eqnarray}\notag
\beta(\mathbf{d})(\mathbf{y}):=
\begin{cases}
\mathbf{y} &\text{if $\mathbf{y}\in\Sigma\backslash \cup_{j=1}^k E_j$}\medskip\\
Y_j\circ Y_j^{-1}(\mathbf{d}_j)(\mathbf{y}) &\text{if $\mathbf{y}\in Y_j(\mathbf{d}_j)((0,\infty)\times S^1)$}.
\end{cases}
\end{eqnarray}
We set
\begin{eqnarray}
\Phi(\mathbf{d})(y):=\sum_{j=1}^k \tilde{\eta}_j\bigg(\sum_{l=1}^3 \mathbf{d}^{T_l}_j\Phi^{T,e_l}_{E_j}+\sum_{l=1}^2 \mathbf{d}^{R_l}_j\Phi^{R,e_l}_{E_j}+\mathbf{d}^D_j\Phi^D_{E_j}\bigg), \qquad\forall\, y\in\Sigma,
\end{eqnarray}
where
\begin{eqnarray}
\tilde{\eta}_j(y):=
\begin{cases}
\xi(s) &\text{for $y=Y_j(s,\vartheta)\in E_j$}\medskip\\
0 &\text{for $y\in\Sigma\backslash (\cup_{1\le j\le k}E_j).$}
\end{cases}\label{def_tilde_eta_j}
\end{eqnarray}
The new cutoff functions $\tilde{\eta}_j$ satisfy $\tilde{\eta}_j\eta_j=\tilde{\eta}_j$ and are not exactly equal to $\eta_j$, however they can be used in the definition of $\mathcal{D}(\Sigma)$ instead of the $\eta_j$'s, since modifying the basis of $\mathcal{D}(\Sigma)$ on a compact set does not change the validity of Proposition \ref{prop_non_deg}. We stress that the mean curvature of $\Sigma(\mathbf{d})$ is constant in the set
$$\Sigma(\mathbf{d})\backslash \cup_{1\le j\le k} Y_j(\mathbf{d}_j)\big((s_0,\infty)\times S^1\big),$$
where it also satisfies
\begin{eqnarray}
H_{\Sigma(\mathbf{d})}(\mathbf{y})=H_\Sigma+\mathcal{J}_\Sigma \Phi(\mathbf{d})(\beta(\mathbf{y})), \label{aux_MC}
\end{eqnarray}
since, in this set, $\Phi(\mathbf{d})\equiv 0$. Now it remains to compare the mean curvature of $\Sigma(\mathbf{d})$ to the one of $\Sigma$ along the ends. For this purpose, it is useful to note that, in the annular regions parametrized by $(s,\vartheta)\in(s_0,s_0+1)\times S^1$, $\Sigma(\mathbf{d})$ is close to the normal graph over $\Sigma$ of $\Phi(\mathbf{d})$, which suggests that the mean curvature is given, up to lower order terms, by (\ref{aux_MC}). %For $s>s_0+1$, the decay of $v_j$ plays a crucial role. In order to clarify this closeness, we prove the following result.

\begin{lemma}
Let $\mathbf{d}=(\mathbf{d}_1,\dots,\mathbf{d}_k)\in (B_\eta)^k$, $a\in(0,\bar{a})$ and $\alpha\in(0,1)$. Then, for $\eta>0$ small enough, the mean curvature of $H_{\Sigma(\mathbf{d})}$ fulfills
\begin{eqnarray}
\|H_{\Sigma(\mathbf{d})}\circ Y_j(\mathbf{d}_j)-H_\Sigma-\eps^2 \mathcal{J}_\Sigma\Phi(\mathbf{d})\circ Y_j\|_{C^{0,\alpha}_a((0,\infty)\times S^1)}\le C|\mathbf{d}_j|e^{-(\bar{a}-a)s_0},\label{curvature_Sigma_d}
\end{eqnarray}
for some constant $C>0$, for any $1\le j\le k$.
\label{lemma_H_Sigmad}
\end{lemma}
\begin{proof}
By Lemma \ref{lemma_perturbe_D} and Remark \ref{rem_Jacobi_fields}, we have
\[X_{\tau_j}(\mathbf{d}_j)(s,\vartheta)=X_{\tau_j}(s',\vartheta')
+(\Phi_{\tau_j}(\mathbf{d}_j)+\tilde{\psi}(\mathbf{d}_j))(s',\vartheta')N_{\tau_j}(s',\vartheta'),\]
for some $(s',\vartheta'):=T_{\mathbf{d}_j}(s,\vartheta)$, where $\tilde{\psi}(\mathbf{d}_j)=\psi(\mathbf{d}_j)\circ Y_j$ and $T_{\mathbf{d}_j}$ is close to the identity, in the sense that
\[\|T_{\mathbf{d}_j}-Id\|_{C^{0,\alpha}((s_1,s_2)\times S^1)}\le C|\mathbf{d}_j|,\]
for some constant $C>0$. Therefore, using the definition of $Z_j(\mathbf{d}_j)$, it follows that, for $(s,\vartheta)\in(s_0-1,s_0+2)\times S^1$,
\begin{equation}
\begin{aligned}
Y_j(\mathbf{d}_j)&=Y_j+\xi(X_{\tau_j}(\mathbf{d}_j)+v_j N_{\tau_j}(\mathbf{d}_j)-Y_j)\\
&=
Y_j+\xi(\Phi(\mathbf{d})\circ Y_j)N_\Sigma\circ Y_j+\xi(X_{\tau_j}\circ T_{\mathbf{d}_j}- X_{\tau_j}+(\tilde{\psi}(\mathbf{d}_j)N_{\tau_j})\circ T_{\mathbf{d}_j} )\\
&\quad +\xi\left[(\Phi_{\tau_j}(\mathbf{d}_j) N_{\tau_j})\circ T_{\mathbf{d}_j} -\Phi(\mathbf{d})N_\Sigma\circ Y_j\right]
+\xi v_j(N_{\tau_j}(\mathbf{d}_j)-N_{\tau_j}).
\end{aligned}
\label{normal_graph}
\end{equation}
By the variational definition of the Jacobi operator, the curvature of the surface parametrised by
$$(s,\vartheta)\in(s_0-1,s_0+2)\times S^1\mapsto Y_j(s,\vartheta)+\xi(s)\big(\Phi(\mathbf{d})N_\Sigma(s,\vartheta)\big)\circ Y_j(s,\vartheta),$$
is given, at first order, by the mean curvature of $\Sigma$ plus the Jacobi operator applied to the function of which we take the normal graph, that is,
\begin{eqnarray}
H_\Sigma+ \mathcal{J}_\Sigma\Phi(\mathbf{d})\circ Y_j+ \mathcal{Q}_\Sigma(\Phi(\mathbf{d}),\nabla_\Sigma\Phi(\mathbf{d}),\nabla_\Sigma^2\Phi(\mathbf{d}))\circ Y_j.
\end{eqnarray}
where the $\mathcal{Q}_\Sigma(\Phi(\mathbf{d}),\nabla_\Sigma\Phi(\mathbf{d}),\nabla_\Sigma^2\Phi(\mathbf{d}))\circ Y_j\in C^{0,\alpha}((s_1,s_2)\times S^1)$ is quadratic in $\Phi(\mathbf{d})$ and its derivatives. Since, by (\ref{normal_graph}), we are considering a small variation of this surface, we have
$$\|H_{\Sigma(\mathbf{d})}\circ Y_j(\mathbf{d}_j)-H_\Sigma-\mathcal{J}_\Sigma\Phi(\mathbf{d})\circ Y_j\|_{C^{0,\alpha}((s_0-1,s_0+2)\times S^1)}\le C|\mathbf{d}_j|,$$
due to the decay of $v_j$. In order to estimate the curvature outside a compact set, we just use the fact that here the ends of $\Sigma(\mathbf{d})$ are asymptotic to $D_{\tau_j}(\mathbf{d}_j)$. In fact, for $s>s_0+1$, the curvature is given by
\begin{eqnarray}\notag
H_{\Sigma(\mathbf{d})}\circ Y_j(\mathbf{d}_j)
=H_{D_{\tau_j}(\mathbf{d}_j)}+\tilde{L}_{\tau_j(\mathbf{d}_j)}v_j+\bar{\mathcal{Q}}_{D_{\tau_j}(\mathbf{d}_j)}(v_j,\nabla v_j,\nabla^2 v_j),
\end{eqnarray}
where $\tilde{L}_{\tau_j(\mathbf{d}_j)}$ is the expression of the Jacobi operator of $D_{\tau_j}(\mathbf{d}_j)$ in isothermal coordinates (see (\ref{def_jacobi_operator})) and $\bar{\mathcal{Q}}_{D_{\tau_j}(\mathbf{d}_j)}$ is quadratic in $v_j$ and its derivatives, up to order $2$. Since $\Sigma$ and $D_{\tau_j}(\mathbf{d}_j)$ are constructed in such a way that $H_\Sigma=H_{D_{\tau_j}}=H_{D_{\tau_j}(\mathbf{d}_j)}$, then 
$$\tilde{L}_{\tau_j}v_j+\bar{\mathcal{Q}}_{D_{\tau_j}}(v_j,\nabla v_j,\nabla^2 v_j)=0,$$
therefore
\[
H_{\Sigma(\mathbf{d})}\circ Y_j(\mathbf{d}_j)=H_\Sigma+(\tilde{L}_{\tau_j(\mathbf{d}_j)}-\tilde{L}_{\tau_j})v_j
+\bar{\mathcal{Q}}_{D_{\tau_j}(\mathbf{d}_j)}(v_j,\nabla v_j,\nabla^2 v_j)
-\bar{\mathcal{Q}}_{D_{\tau_j}}(v_j,\nabla v_j,\nabla^2 v_j),
\]
for $s>s_0+1$. By (\ref{def_jacobi_operator}) and the decay of $v_j$ and its derivatives, 
$$\|(\tilde{L}_{\tau_j(\mathbf{d}_j)}-\tilde{L}_{\tau_j})v_j+\bar{\mathcal{Q}}_{D_{\tau_j}(\mathbf{d}_j)}(v_j,\nabla v_j,\nabla^2 v_j)
-\bar{\mathcal{Q}}_{D_{\tau_j}}(v_j,\nabla v_j,\nabla^2 v_j)\|_{C^{0,\alpha}_a((s_0,\infty)\times S^1)}\le C|\mathbf{d}_j|e^{-(\bar{a}-a)s_0}.$$
In conclusion, using that, along the end $E_j$, the Jacobi operator of $\Sigma$ is close to the one of $D_{\tau_j}$ (see formula (2.3) of \cite{KMP}),
\[
\begin{aligned}
& \|H_{\Sigma(\mathbf{d})}\circ Y_j(\mathbf{d}_j)-H_\Sigma\circ Y_j-\mathcal{J}_\Sigma\Phi(\mathbf{d})\circ Y_j\|_{C^{0,\alpha}_a((s_0+1,\infty)\times S^1)}\\
& \le\|H_{\Sigma(\mathbf{d})}\circ Y_j(\mathbf{d}_j)-H_\Sigma\circ Y_j\|_{C^{0,\alpha}_a((s_0+1,\infty)\times S^1)}+\|\mathcal{J}_\Sigma\Phi(\mathbf{d})\circ Y_j\|_{C^{0,\alpha}_a((s_0+1,\infty)\times S^1)}\\
&\le C|\mathbf{d}_j| e^{-(\bar{a}-a)s_0},
\end{aligned}
\]
for some constant $C>0$.
\end{proof}
We also need to compare the metric and the second fundamental form of the two surfaces, in order to compare their Jacobi operators. Before giving a precise statement, let us fix some notation. For any $\mathbf{d}=(\mathbf{d}_1,\dots,\mathbf{d}_k)\in (B_\eta)^k$, $g_{\Sigma(\mathbf{d})}$ denotes the metric of $\Sigma(\mathbf{d})$ and $g(\mathbf{d}_j)$ denotes the metric of $D_{\tau_j}(\mathbf{d}_j)$. Similarly, $A_{\Sigma(\mathbf{d})}$ and $A_{D_{\tau_j}}(\mathbf{d}_j)$ denote, respectively,  the second fundamental form of $\Sigma(\mathbf{d})$  and of $D_{\tau_j}(\mathbf{d}_j)$. When the evaluation at $\mathbf{d}_j$ or at $\mathbf{d}$ is omitted, it is understood that we are evaluating at $\mathbf{d}=0$, that is we are referring to the geometric quantities either of $D_{\tau_j}$ or of $\Sigma$. When we refer to the metric and the second fundamental form of a Delaunay surface, it is understood that we are using their expression in isothermal coordinates.
\begin{lemma}
Let $\mathbf{d}^i\in (B_\eta)^k$, $i=1,2$, and $\alpha\in(0,1)$. Then, for $\eta>0$ small enough,
\begin{eqnarray}
\|(g_{\Sigma(\mathbf{d}^1)})_{lm}\circ Y_j(\mathbf{d}_j^1)-(g_{\Sigma(\mathbf{d}^2)})_{lm}\circ Y_j(\mathbf{d}_j^2)\|_{C^{1,\alpha}((0,\infty)\times S^1)}\le C|\mathbf{d}_j^1-\mathbf{d}_j^2|,\label{metric_Sigma_d}\\
%\|\Delta_{\Sigma(\mathbf{d}^1)}h(\mathbf{d}^1)\circ Y_j(\mathbf{d}^1_j)-\Delta_{\Sigma(\mathbf{d}^1)}h(\mathbf{d}^2_j)\circ Y_j(\mathbf{d}^2)\|_{\mathcal{C}^{0,\alpha}_a((0,\infty)\times S^1)}\le c\eps^2\|h_0\|_{C^{2,\alpha}_a(\Sigma)}|\mathbf{d}^1_j-\mathbf{d}^2_j|,\label{lapl_Sigma_d}\\
\|(A_{\Sigma(\mathbf{d}^1)})_{lm}\circ Y_j(\mathbf{d}_j^1)-(A_{\Sigma(\mathbf{d}^2)})_{lm}\circ Y_j(\mathbf{d}_j^2)\|_{C^{1,\alpha}((0,\infty)\times S^1)}\le C|\mathbf{d}_j^1-\mathbf{d}_j^2|,\label{secondff_Sigma_d}
\end{eqnarray}
for any $1\le l,m\le 2$ and $1\le j\le k$, for some constant $C>0$.
\label{lemma_Sigma_d}
\end{lemma}

\begin{remark}
In particular, applying (\ref{metric_Sigma_d}) and (\ref{secondff_Sigma_d}) with, for instance, $\mathbf{d}^2=0$, we have
$$\|(A_{\Sigma(\mathbf{d})})_{lm}\circ Y_j(\mathbf{d}_j)\|_{C^{0,\alpha}((0,\infty)\times S^1)}+
\|(g_{\Sigma(\mathbf{d})})_{lm}\circ Y_j(\mathbf{d}_j)\|_{C^{1,\alpha}((0,\infty)\times S^1)}\le C, \qquad\forall\, \mathbf{d}\in (B_\eta)^k,$$
for any $1\le l,m\le 2$ and $1\le j\le k$, for some constant $C>0$.
\end{remark}
\begin{remark}
We note that, since $\Sigma$ is smooth and $v_j\in C^\infty((0,\infty)\times S^1)$, all the geometric quantities of $\Sigma(\mathbf{d})$ are also smooth (see Remark \ref{rem_regularuty_v}). However, the norms taken into account in the estimates of Lemma \ref{lemma_Sigma_d} are enough for our purposes.   
\label{rem_regularity_Sigma_d}
\end{remark}

\begin{proof}
Since $\Sigma(\mathbf{d})$ agrees with $\Sigma$ in a compact set, from now on we fix an end and we prove our estimates using the coordinates $(s,\vartheta)\in(s_0,\infty)\times S^1$. We recall that, outside a ball of radius $R$ large enough, the end $E_j(\mathbf{d}_j)$ agrees with a normal graph over the Delaunay surface $D_{\tau_j}(\mathbf{d}_j)$, explicitly parametrized by
$$Z_j(\mathbf{d}_j):=X_{\tau_j}(\mathbf{d}_j)+v_j N_{\tau_j}(\mathbf{d}_j).$$
Therefore, the restriction of the metric of $\Sigma(\mathbf{d})$ to $E_j(\mathbf{d}_j)\backslash B_R$ is given by 
\[
\begin{aligned}
\hat{g}_{ss}(\mathbf{d}_j)&=g_{ss}(\mathbf{d}_j)+2(A_{D_{\tau_j}}(\mathbf{d}_j))_{ss}v_j+v_j^2 g^{ss}(\mathbf{d}_j)(A_{D_{\tau_j}}(\mathbf{d}_j))_{ss}^2+(\partial_s v_j)^2,\\
\hat{g}_{\vartheta\vartheta}(\mathbf{d}_j)&=g_{\vartheta\vartheta}(\mathbf{d}_j)+2(A_{D_{\tau_j}}(\mathbf{d}_j))_{\vartheta\vartheta}
v_j+v_j^2 g^{\vartheta\vartheta}(\mathbf{d}_j) (A_{D_{\tau_j}}(\mathbf{d}_j))_{\vartheta\vartheta}^2+(\partial_\vartheta v_j)^2,\\
\hat{g}_{s\vartheta}(\mathbf{d}_j)&=\partial_s v_j\partial_\vartheta v_j.
\end{aligned}
\]
Thus, using the explicit expression of the metric and the second fundamental form of a Delaunay surface, that is (\ref{g_A_Delaunay}), we have, for instance
\[
\begin{aligned}
\|\hat{g}_{ss}(\mathbf{d}_j^1)-\hat{g}_{ss}(\mathbf{d}_j^2)\|_{C^{1,\alpha}((s_0,\infty)\times S^1)}&\le \|\hat{g}_{ss}(\mathbf{d}_j^1)-g_{ss}(\mathbf{d}_j^1)+\hat{g}_{ss}(\mathbf{d}_j^2)-g_{ss}(\mathbf{d}_j^2)\|_{C^{1,\alpha}((s_0,\infty)\times S^1)}\\
&\quad +\|g_{ss}(\mathbf{d}_j^1)-g_{ss}(\mathbf{d}_j^2)\|_{C^{1,\alpha}((s_0,\infty)\times S^1)}\\
&\le C|\mathbf{d}_j^1-\mathbf{d}_j^2|.
\end{aligned}
\]
Using the explicit form of the parametrisation $Y_j(\mathbf{d})$, it is possible to compute
\[
\begin{aligned}
(g_{\Sigma(\mathbf{d})})_{ss}-(g_\Sigma)_{ss}&=\xi^2(\hat{g}_{ss}(\mathbf{d})-(g_\Sigma)_{ss})+2\xi(1-\xi)\partial_s Y_j\cdotp \partial_s(Z_j(\mathbf{d}_j)-Y_j)+2\xi\xi'\partial_s Z_j(\mathbf{d}_j)\cdotp (Z_j(\mathbf{d}_j)-Y_j)\\
&\quad +(\xi')^2|Z_j(\mathbf{d}_j)-Y_j|^2+2(1-\xi)\xi'\partial_s Y_j\cdotp(Z_j(\mathbf{d}_j)-Y_j).
\end{aligned}
\]
Similar computations show that
\[
\begin{aligned}
(g_{\Sigma(\mathbf{d})})_{\vartheta\vartheta}(\mathbf{d})-(g_\Sigma)_{\vartheta\vartheta}&=
\xi^2(\hat{g}_{\vartheta\vartheta}(\mathbf{d})-(g_\Sigma)_{\vartheta\vartheta})+2\xi(1-\xi)\partial_\vartheta Y_j\cdotp \partial_\vartheta(Z_j(\mathbf{d}_j)-Y_j),\\
(g_{\Sigma(\mathbf{d})})_{s\vartheta}-(g_\Sigma)_{s\vartheta}&=\xi^2(\hat{g}_{s\vartheta}(\mathbf{d}_j)-(g_\Sigma)_{s\vartheta})+\xi(1-\xi)(\partial_s Y_j\cdotp \partial_\vartheta(Z_j(\mathbf{d}_j)-Y_j)+\partial_\vartheta Y_j\cdotp \partial_s(Z_j(\mathbf{d}_j)-Y_j))\\
&\quad +\xi'(1-\xi)\partial_\vartheta Y_j\cdotp (Z_j(\mathbf{d}_j)-Y_j)+\xi\xi'\partial_\vartheta Z_j(\mathbf{d}_j)\cdotp (Z_j(\mathbf{d}_j)-Y_j),
\end{aligned}
\]
hence (\ref{metric_Sigma_d}) is true. The proof of (\ref{secondff_Sigma_d}) is similar, we just use the definition of second fundamental form and the fact that also the normal vectors satisfy inequalities like (\ref{metric_Sigma_d}).

%\textit{Proof of (\ref{lapl_Sigma_d}).}\\

%As above, we restrict ourselves to an end and we use the relation
%$$\big(\Delta_{\Sigma(\mathbf{d})}h(\mathbf{d})\big)\circ Y_j(\mathbf{d}_j)=
%\frac{1}{\sqrt{|\tilde{g}(\mathbf{d}_j)|}}\partial_m\bigg(\sqrt{|\tilde{g}(\mathbf{d}_j)|}\tilde{g}(\mathbf{d}_j)^{ml}
%\partial_l\big(h(\mathbf{d})\circ Y_j(\mathbf{d}_j)\big)\bigg).$$
%The conclusion follows from (\ref{metric_Sigma_d}).\\

%We recall that we have denoted the Delaunay surface to which the end $E_j$ of $\Sigma$ is asymptotic by $D^{\mathbf{a}_j}_{\tau_j}+\mathbf{b}_j$. In order to compare the geometric quantities of $\Sigma(\mathbf{d})$ with the ones of $\Sigma$, we observe that, in the annular regions where the cutoff functions $\tilde{\eta}_j$ are non constant, the Delaunay surface parametrized by $X_{\mathbf{d}^j}$, defined in (\ref{new_delaunay}), can be seen a the normal graph over $D^{\mathbf{a}_j}_{\tau_j}+\mathbf{b}_j$. 
\end{proof}

\subsection{The Fermi coordinates}\label{section_Fermi}
We are interested in the Fermi coordinates of $\Sigma(\eps^2\mathbf{d})$, for $\mathbf{d}\in (B_M)^k$ with $M>0$ large (to be fixed later) and $\eps>0$ small enough. For $\delta>0$ small, we define 
the tubular neighbourhood of $\Sigma(\eps^2\mathbf{d})$ of width $10\delta$ as
$$\mathcal{N}_{10\delta}:=\{x\in\R^3:\, \mathop{dist}(x,\Sigma(\eps^2\mathbf{d}))<10\delta\}.$$
Denoting the inward-pointing normal vector to $\Sigma(\eps^2\mathbf{d})$ at $\mathbf{y}$ by $\nu_{\eps,\mathbf{d}}(\mathbf{y})$, the mapping 
$$\mathcal{Y}_{\eps,\mathbf{d}}(\mathbf{y},z):=\mathbf{y}+z \nu_{\eps,\mathbf{d}}(\mathbf{y})$$
is a diffeorphism between $\Sigma(\eps^2\mathbf{d})\times(-10\delta,10\delta)$ and $\mathcal{N}_{10\delta}$, provided $\delta>0$ is small enough, thus it defines a change of coordinates on $\mathcal{N}_{10\delta}$. The new coordinates $(\mathbf{y},t)\in\Sigma(\eps^2\mathbf{d})\times\R$ associated to this diffeomorphism are known as the \textit{Fermi coordinates} of $\Sigma(\eps^2\mathbf{d})$. We recall that, if a surface is $C^{n,\alpha}$, then the diffeomorphism defining the Fermi coordinates is $C^{n-1,\alpha}$, due to the fact that the normal vector depends on the first derivative of the parametrisations. Here we want our approximate solutions to be of class $C^{2,\alpha}$, thus we need to start from a surface which is at least $C^{3,\alpha}$. However, this is not a problem, since $\Sigma$ is $C^\infty$, and, by Remark \ref{rem_regularuty_v}, the functions $v_j$ are $C^\infty$, thus $\Sigma(\eps^2\mathbf{d})$ is of class $C^\infty$ too. This problem also arises in \cite{PR}, where it is solved by introducing regularising operators in order to get a $C^{3,\alpha}$ manifold.\\ 

The metric of $\mathcal{N}_{10\delta}$ in these coordinates is given by
\[
\begin{aligned}
G_{lm}(\eps^2\mathbf{d})&=(g_{\Sigma(\eps^2\mathbf{d})})_{lm}+2z(A_{\Sigma(\eps^2\mathbf{d})})_{lm}
+z^2(A_{\Sigma(\eps^2\mathbf{d})})_{lk}(A_{\Sigma(\eps^2\mathbf{d})})_{ms}(g_{\Sigma(\eps^2\mathbf{d})})^{ks}, \qquad 1\le m,l\le 2\label{metric_CF}\\
G_{lz}(\eps^2\mathbf{d})&=G_{zl}(\eps^2\mathbf{d})=0, \qquad G_{zz}(\eps^2\mathbf{d})=1, \qquad 1\le l\le 2.
\end{aligned}
\]
It is understood that $G_{lm}(\eps^2\mathbf{d})$ is evaluated at $(\mathbf{y},z)$, while the metric and the second fundamental form of $\Sigma(\eps^2\mathbf{d})$ are evaluated at $\mathbf{y}$. With this notation, $G_{lm}(\eps^2\mathbf{d})(\mathbf{y},0)=(g_{\Sigma(\eps^2\mathbf{d})})_{lm}(\mathbf{y})$. The expression of the Laplacian in Fermi coordinates is given by
$$
\Delta=\Delta_{\Sigma(\eps^2\mathbf{d})_z}-H_{\Sigma(\eps^2\mathbf{d})_z}\partial_z+\partial^2_z,
$$
where $\Sigma(\eps^2\mathbf{d})_z$ is the surface parallel to $\Sigma(\eps^2\mathbf{d})$ at distance $|z|<10\delta$, that is 
$$\Sigma(\eps^2\mathbf{d})_z:=\{\mathcal{Y}_{\eps,\mathbf{d}}(\mathbf{y},z):\, \mathbf{y}\in \Sigma(\eps^2\mathbf{d})\}.$$
For a proof of this well-known formula we refer, for instance, to \cite{PR}. We further expand
$$
\Delta_{\Sigma(\eps^2\mathbf{d})_z}=\Delta_{\Sigma(\eps^2\mathbf{d})}+z\mathbb{A}_{\Sigma(\eps^2\mathbf{d})}(\mathbf{y},t),
\qquad H_{\Sigma(\eps^2\mathbf{d})_z}=H_{\Sigma(\eps^2\mathbf{d})}+z|A_{\Sigma(\eps^2\mathbf{d})}|^2+z^2\mathbb{Q}_{\Sigma(\eps^2\mathbf{d})}(\mathbf{y},z),
$$
where 
\[
\begin{aligned}
\mathbb{A}_{\Sigma(\eps^2\mathbf{d})}(\mathbf{y},z)&=a^{lm}_{\Sigma(\eps^2\mathbf{d})}(\mathbf{y},z)\partial_{lm}
+b^l_{\Sigma(\eps^2\mathbf{d})}(\mathbf{y},z)\partial_l,\\
za_{\Sigma(\eps^2\mathbf{d})}^{lm}(\mathbf{y},z)&=G^{lm}(\eps^2\mathbf{d})(\mathbf{y},z)-g_{\Sigma(\eps^2\mathbf{d})}^{lm}(\mathbf{y}),\\
z(b_{\Sigma(\eps^2\mathbf{d})})^l(\mathbf{y},z)&=G^{km}(\eps^2\mathbf{d})(\mathbf{y},z)\Gamma^l_{km}(\eps^2\mathbf{d})(\mathbf{y},z)
-g_{\Sigma(\eps^2\mathbf{d})}^{km}(\mathbf{y})\Gamma^l_{km}(\eps^2\mathbf{d})(\mathbf{y}),
\end{aligned}
\]
where $\Gamma^l_{km}(\eps^2\mathbf{d})$ are the Christoffel symbols of $\Sigma(\eps^2\mathbf{d})_z$. We stress that $\mathbb{A}_{\Sigma(\eps^2\mathbf{d})}$ does not contain derivatives in $z$ and the coefficients satisfy
\begin{eqnarray}
\|a^{lm}_{\Sigma(\eps^2\mathbf{d})}\|_{C^{0,\alpha}(\Sigma(\eps^2\mathbf{d})\times(-10 \delta,10\delta))}+\|b^l_{\Sigma(\eps^2\mathbf{d})}\|_{C^{0,\alpha}(\Sigma(\eps^2\mathbf{d})\times(-10 \delta,10\delta))}\le C,%(\sup_{\Sigma(\mathbf{d})\times\R}|\nabla^2_{\Sigma(\mathbf{d})} u|+\sup_{\Sigma(\mathbf{d})\times\R}|\nabla_{\Sigma(\mathbf{d})} u|), 
\, 1\le l,m\le 2, 
\label{est_remainder_Fermi_A}
\end{eqnarray}
and $\Q_{\Sigma(\eps^2\mathbf{d})}$ is bounded uniformly in $(\mathbf{y},z)$, that is
\begin{eqnarray}
\|\Q_{\Sigma(\eps^2\mathbf{d})}\|_{C^{0,\alpha}(\Sigma(\mathbf{d})\times(-10 \delta,10\delta))}\le C. \label{est_remainder_Fermi_Q}
\end{eqnarray}
We stress that the constants appearing in (\ref{est_remainder_Fermi_A}) and (\ref{est_remainder_Fermi_Q}) are independent of $\mathbf{d}$ and $\eps$. In the sequel, we will also need some Lipschitz dependence on $\mathbf{d}$, namely, setting $\Omega_\delta:=(0,\infty)\times S^1\times (-10\delta,10\delta)$,
\begin{align}
\|a^{lm}_{\Sigma(\eps^2\mathbf{d}^1)}\circ (Y_j(\eps^2\mathbf{d}^1_j)\times Id_\R)-a^{lm}_{\Sigma(\eps^2\mathbf{d}^2)}\circ (Y_j(\eps^2\mathbf{d}^2_j)\times Id_\R)\|_{C^{0,\alpha}(\Omega_\delta)}&\le C\eps^2|\mathbf{d}^1-\mathbf{d}^2|,\label{lip_dep_a}\\
\|b^l_{\Sigma(\eps^2\mathbf{d}^1)}\circ (Y_j(\eps^2\mathbf{d}^1_j)\times Id_\R)-b^l_{\Sigma(\eps^2\mathbf{d}^2)}\circ (Y_j(\eps^2\mathbf{d}^2_j)\times Id_\R)\|_{C^{0,\alpha}(\Omega_\delta)}&\le C\eps^2|\mathbf{d}^1-\mathbf{d}^2|\label{lip_dep_b},\\
\|\Q_{\Sigma(\eps^2\mathbf{d}^1)}\circ (Y_j(\eps^2\mathbf{d}^1_j)\times Id_\R)-\Q_{\Sigma(\eps^2\mathbf{d}^2)}\circ (Y_j(\eps^2\mathbf{d}^2_j)\times Id_\R)\|_{C^{0,\alpha}(\Omega_\delta)}
&\le C\eps^2|\mathbf{d}^1-\mathbf{d}^2|.\label{lip_dep_Q}
\end{align}
These estimates follow from Lemma \ref{lemma_Sigma_d} and (\ref{metric_CF}). Finally
\begin{eqnarray}
\Delta=\Delta_{\Sigma(\eps^2\mathbf{d})}+z\mathbb{A}_{\Sigma(\eps^2\mathbf{d})}(z,\mathbf{y})
-(H_{\Sigma(\eps^2\mathbf{d})}+z|A_{\Sigma(\eps^2\mathbf{d})}|^2+z^2\Q_{\Sigma(\eps^2\mathbf{d})}(z,\mathbf{y}))\partial_z
+\partial^2_z.\label{lapl_Fermi}
\end{eqnarray}
Now we introduce a shift on the surface: given a function $h:\Sigma(\eps^2\mathbf{d})\to\R$ of class $C^2$, we introduce the change of variables 
$$\mathcal{Y}_{\eps,h,\mathbf{d}}(t,\mathbf{y})=\mathbf{y}+\eps (t+\eps h(\mathbf{d})(\mathbf{y}))\nu_{\eps,\mathbf{d}}(\mathbf{y}), \qquad\forall\, (\mathbf{y},t)\in\Sigma(\eps^2\mathbf{d})\times(-9\delta/\eps,9\delta/\eps).$$
Finally, the expression of the Laplacian in the \textit{shifted Fermi coordinates} $(t,\mathbf{y})$ is given by
\begin{equation}
\label{lapl_shift}
\begin{aligned}
\Delta&=\Delta_{\Sigma(\mathbf{d})}-\eps \Delta_{\Sigma(\eps^2\mathbf{d})} h \partial_t-2\eps\nabla_{\Sigma(\eps^2\mathbf{d})} h\nabla_{\Sigma(\eps^2\mathbf{d})}\partial_t+\eps^2|\nabla_{\Sigma(\eps^2\mathbf{d})} h|^2\partial_t^2+\eps^{-2}\partial^2_t\\
&\quad +\eps^{-1}(H_{\Sigma(\eps^2\mathbf{d})}+(\eps t+\eps^2 h)|A_{\Sigma(\eps^2\mathbf{d})}|^2+(\eps t+\eps^2 h)^2\mathbb{Q}_{\Sigma(\eps^2\mathbf{d})}(\mathbf{y},\eps t+\eps^2 h))\partial_t\\
&\quad +\eps (t+\eps h)\mathbb{A}_{\Sigma(\eps^2\mathbf{d})}(\mathbf{y},\eps t+\eps^2 h)-\eps^2 (t+\eps h)\mathbb{A}_{\Sigma(\eps^2\mathbf{d})}(\mathbf{y},\eps t+\eps^2 h)[h]\partial_t\\
&\quad -2\eps^2 (t+\eps h)a^{ij}_{\Sigma(\eps^2\mathbf{d})}(\mathbf{y},\eps t+\eps^2 h)\partial_i h\partial_{tj}+\eps^3 (t+\eps h)a^{ij}_{\Sigma(\eps^2\mathbf{d})}(\mathbf{y},\eps t+\eps^2 h)\partial_i h\partial_j h\partial_t^2.
\end{aligned}
\end{equation}

%Using (\ref{est_remainder_Fermi_A}) and  (\ref{est_remainder_Fermi_Q}), we see that 
%\begin{eqnarray}
%|\mathbb{Q}_{\Sigma(\mathbf{d})}(\eps t+\eps^2 h,\mathbf{y})|\le C, \qquad|\mathbb{A}_{\Sigma(\mathbf{d})}(\eps t+\eps^2 h,\mathbf{y})|\le C (\sup_{\Sigma(\mathbf{d})\times\R}|\nabla^2_{\Sigma(\mathbf{d})} u|+\sup_{\Sigma(\mathbf{d})\times\R}|\nabla_{\Sigma(\mathbf{d})} u|),
%\end{eqnarray}
%for some constant $C>0$ independent of $\mathbf{d}$. 

\section{The approximate solution}\label{section_approx_solution}

\setcounter{equation}{0}

\subsection{First approximation}\label{section_first_approximation}

We need to introduce a shift on $\Sigma(\eps^2\mathbf{d})$, that is a function on $\Sigma(\eps^2\mathbf{d})$ which, along the ends, is the sum of a decaying term and a periodic one. More precisely, for $\tau\in(0,1)$, $n\in \N$ and $\alpha\in(0,1)$, we define the spaces
$$D^{n,\alpha}_\tau(\R):=\{{\tt h}\in C^{n,\alpha}_{loc}(\R):{\tt h}(s)={\tt h}(s+s_\tau)={\tt h}(s_\tau-s), \, \forall\, s\in\R\},$$
of functions that are periodic and even with respect to the reflection around $s=s_\tau/2$. Let $\tau_j(\eps^2\mathbf{d}_j):=\tau(\epsilon_j+\eps^2\mathbf{d}^D_j)$, where $\epsilon_j=1-\sqrt{1-\tau_j}$, in such a way that $\tau_j(0)=\tau_j$. Given a function $h_0\in \mathcal{C}^{2,\alpha}_a(\Sigma)$, $k$ functions ${\tt h}_j\in D^{n,\alpha}_{\tau_j(\eps^2\mathbf{d}_j)}(\R)$, $1\le j\le k$, and a smooth cutoff function $\tilde{\xi}:\R\to\R$ equal to $0$ in $(-\infty,s_0+2)$ and equal to $1$ in $(s_0+3,\infty)$, we set
\begin{equation}
\label{def_zeta_j}
\zeta_j(\mathbf{y}):=
\begin{cases}
\tilde{\xi}(s) &\text{if $\mathbf{y}=Y_j(\eps^2\mathbf{d}_j)(s,\vartheta)$, $\forall\,(s,\vartheta)\in(0,\infty)\times S^1$,}\medskip\\
0 &\text{if $\mathbf{y}\in \Sigma(\eps^2\mathbf{d})\backslash Y_j(\eps^2\mathbf{d}_j)((0,\infty)\times S^1)$.}
\end{cases}
\end{equation}
and
\[
\begin{aligned}
 h_j(Y_j(\eps^2\mathbf{d}_j)(s,\vartheta))&:={\tt h}_j(s),\qquad \forall \, (s,\vartheta)\in(0,\infty)\times S^1\\
h(\mathbf{d})(\mathbf{y})&:=h_0\circ \beta(\eps^2\mathbf{d})(\mathbf{y})+\sum_{j=1}^k \zeta_j(\mathbf{y})h_j(\mathbf{y}).
\end{aligned}
\]
There is a slight abuse of notation here, justified by the fact that the cutoff functions $\zeta_j$ vanish in $Y_j(\eps^2\mathbf{d}_j)((0,s_0+2)\times S^1)$, thus the value of $h_j$ is relevant just along an end of $\Sigma(\eps^2\mathbf{d})$, where it is actually defined. Throughout the construction, we will always assume that the functions $h_0,\, \text{h}_1,\dots,\text{h}_k$ are in the set
\begin{equation}
\label{def_H}
\textbf{H}(\mathbf{d}):=\left\{\mathbf{h}:=(h_0,{\tt h}_1,\dots,{\tt h}_k)\left| \begin{aligned} &h_0\in \mathcal{C}^{2,\alpha}_a(\Sigma),\, \|h_0\|_{\mathcal{C}^{2,\alpha}_a(\Sigma)}< M,\\
&{\tt h}_j\in D^{2,\alpha}_{\tau_j(\eps^2\mathbf{d}_j)}(\R),\, \|{\tt h}_j\|_{C^{2,\alpha}(\R)}<K,\, 1\le j\le k \end{aligned}\right.\right\},
\end{equation}
where $K$ and $M$ will be determined later. Note that $M$ is also the radius of the ball $B_M$ to which the parameters $\mathbf{d}_j$ belong. In this way $h(\mathbf{d})$ is the sum of a globally defined decaying term plus $k$ periodic terms, each of them playing a role along one of the ends. In order to define the approximate solution near the surface, we introduce the Fermi coordinates $(\mathbf{y},z)\in\Sigma(\eps^2\mathbf{d})\times(-10\delta,10 \delta)$ in a tubular neighbourhood of $\Sigma(\eps^2\mathbf{d})$, and the change of coordinates $$z=\eps t+\eps^2 h(\mathbf{d})(\mathbf{y}).$$ 
We would like to produce an error which is decaying both along the surface and in $t$. In order to make this concept clear, we introduce some function spaces. Given $a>0$, $\gamma>0$, $n\in\N$, $\alpha\in(0,1)$ and a function $\phi\in C^{n,\alpha}_{loc}(\Sigma\times\R)$, we say that $\phi\in\mathcal{E}^{n,\alpha}_{a,\gamma}(\Sigma\times\R)$ if the norm 
$$\|\phi\|_{\mathcal{E}^{n,\alpha}_{a,\gamma}(\Sigma\times\R)}:=\sum_{0\le m+k\le n}\eps^m \|\partial_t^k D^m_\Sigma\phi\|_{\mathcal{C}^{0,\alpha}_{a,\gamma}(\Sigma\times\R)}$$
is finite, where we have set
\begin{equation}
\label{def_wa}
\begin{aligned}
&\|\phi\|_{\mathcal{C}^{0,\alpha}_{a,\gamma}(\Sigma\times\R)}:=\|\phi w_a(y)\cosh^\gamma(t)\|_{C^{0,\alpha}(\Sigma\times\R)}, \\
&w_a(y):=\tilde{\eta}_0(y)+\sum_{j=1}^k \tilde{\eta}_j(y) w_{a,j}(y), \qquad w_{a,j}(Y_j(s,\vartheta)):=e^{as},\, \forall\, (s,\vartheta)\in(0,\infty)\times S^1.
\end{aligned}
\end{equation}
We recall that $\tilde{\eta}_j$ is defined above (see (\ref{def_tilde_eta_j})) and $\tilde{\eta}_0$ is chosen in such a way that the family $\{\tilde{\eta}_j\}_{0\le j\le k}$ is a partition of unity on $\Sigma$. For a function $\phi\in C^{n,\alpha}_{loc}(\Sigma(\eps^2\mathbf{d})\times\R)$, we say that it is in $\mathcal{E}^{n,\alpha}_{a,\gamma}(\Sigma(\eps^2\mathbf{d})\times\R)$ if the composition $\phi\circ(\beta(\eps^2\mathbf{d})^{-1}\times \mathop{Id}_\R)$ is in $\mathcal{E}^{n,\alpha}_{a,\gamma}(\Sigma\times\R)$ and its norm is defined by
$$\|\phi\|_{\mathcal{E}^{n,\alpha}_{a,\gamma}(\Sigma(\eps^2\mathbf{d})\times\R)}:=
\|\phi\circ(\beta(\eps^2\mathbf{d})^{-1}\times {\mathop{Id}}_\R)\|_{\mathcal{E}^{n,\alpha}_{a,\gamma}(\Sigma\times\R)}.$$ 
In terms of these spaces, we want the error to be in $\mathcal{E}^{0,\alpha}_{a,\gamma}(\Sigma(\eps^2\mathbf{d})\times\R)$.\\ 

A first guess could be that our approximate solution just depends on $t$. Using the expansion of the Laplacian given by (\ref{lapl_shift}), with $h:=h(\mathbf{d})$, we get a formal expression for the error of the form
\[
\begin{aligned}
&\eps\Delta U+\eps^{-1}f(U)-\ell_\eps =\eps^{-1}(U''-\eps H_\Sigma U'+f(U)-\eps\ell_\eps)-(H_{\Sigma(\eps^2\mathbf{d})}-H_\Sigma)U'\\
&\qquad-\eps(t+\eps h)|A_{\Sigma(\eps^2\mathbf{d})}|^2 U'
-\eps^2\Delta_{\Sigma(\eps^2\mathbf{d})} h-\eps^2(t+\eps h)^2\tr A_{\Sigma(\eps^2\mathbf{d})}^3 U'-\eps^3(t+\eps h)^3\tilde{\Q}_{\Sigma(\eps^2\mathbf{d})}(\eps t+\eps^2 h,\mathbf{y})U'\\
&\qquad +\eps^3|\nabla_{\Sigma(\eps^2\mathbf{d})} h|^2 U''-\eps^3(t+\eps h)\mathbb{A}_{\Sigma(\eps^2\mathbf{d})}(\eps t+\eps^2 h,\mathbf{y})[h]U'+\eps^4 (t+\eps h)a^{ij}_{\Sigma(\eps^2\mathbf{d})}(\eps t+\eps^2 h,\mathbf{y})\partial_i h\partial_j h U'',
\end{aligned}
\]
where we have set $f(t):=t-t^3$. In the above computation, we have expanded the term $\Q_{\Sigma(\mathbf{d})}$ as
$$\Q_{\Sigma(\eps^2\mathbf{d})}(\mathbf{y},t)=\tr A_{\Sigma(\eps^2\mathbf{d})}^3(\mathbf{y})+z\tilde{\Q}_{\Sigma(\eps^2\mathbf{d})}(\mathbf{y},t),$$
with $\tilde{\Q}_{\Sigma(\eps^2\mathbf{d})}$ satisfying (\ref{est_remainder_Fermi_Q}). This is not crucial, however the term $\tr A_{\Sigma(\eps^2\mathbf{d})}^3(\mathbf{y})$ is somehow the easiest term to understand in the error, and we will use it as a model to show the properties of the error. To get rid of  the main term we choose $U$ to be a solution to the ODE
$$U''-\eps H_\Sigma U'+f(U)=\eps\ell_\eps.$$
This solution can be found as a small perturbation of $v_\star(t):=\tanh(t/\sqrt{2})$, in the form 
\begin{eqnarray}
U(t)=v_\star(t)+\eps v_1(t).\label{construction_U}
\end{eqnarray} 
We recall that $v_\star$ is the unique monotone increasing solution to $-v''_\star=v_\star-v_\star^3$ vanishing at $t=0$. The function $U$ was also used in \cite{HK1} for the construction of the approximate solution. It follows from their construction that $U$ depends on the curvature $H_\Sigma$ and
$$\ell_\eps=-\frac{1}{2}H_\Sigma\int_\R (v'_\star(t))^2 dt+O(\eps)<0.$$
Formally, this produces an error which is the sum of $-(H_{\Sigma(\eps^2\mathbf{d})}-H_\Sigma)U'$, which is of order $\eps^2$ (see Lemma \ref{lemma_Sigma_d}), and a term of order $\eps$, which is still too large for our purposes. Therefore we add a correction of order $\eps^2$, so that the new approximation has the form
$$u_0(\mathbf{d})(\mathbf{y},t):=U(t)+\eps^2|A_{\Sigma(\eps^2\mathbf{d})}(\mathbf{y})|^2\psi_0(t).$$
Using once again (\ref{lapl_shift}), we get
\begin{equation}
\label{exp_error_t}
\begin{aligned}
&\eps\Delta u_0(\mathbf{d})+\eps^{-1}f(u_0(\mathbf{d}))-\ell_\eps\\
&\quad=
\eps\Delta U+\eps^{-1}f(U)-\ell_\eps+\eps^2(\eps\Delta+\eps^{-1}f'(U))[|A_{\Sigma(\eps^2\mathbf{d})}|^2\psi_0]
-\eps^{3}|A_{\Sigma(\eps^2\mathbf{d})}|^4\psi_0^2(3U+\eps^2|A_{\eps^2\Sigma(\mathbf{d})}|^2\psi_0)\\
&\quad =
-(H_{\Sigma(\eps^2\mathbf{d})}-H_\Sigma)(U'+\eps^2|A_{\Sigma(\eps^2\mathbf{d})}|^2\psi'_0)+\eps|A_{\Sigma(\eps^2\mathbf{d})}|^2(\psi''_0-\eps H_\Sigma \psi'_0+f'(U)\psi_0)-\eps|A_{\Sigma(\eps^2\mathbf{d})}|^2(t+\eps h)U'\\
&\qquad -\eps^2\Delta_{\Sigma(\eps^2\mathbf{d})} h U'-\eps^2(t+\eps h)^2\tr A_{\Sigma(\eps^2\mathbf{d})}^3 U'-\eps^3(t+\eps h)^3\tilde{\Q}_{\Sigma(\eps^2\mathbf{d})}(\eps t+\eps^2 h,\mathbf{y})U'\\
&\qquad-\eps^{3}|A_{\Sigma(\eps^2\mathbf{d})}|^4\psi_0^2(3U+\eps^2|A_{\Sigma(\eps^2\mathbf{d})}|^2\psi_0)+\eps^3|\nabla_{\Sigma(\eps^2\mathbf{d})} h|^2 U''-\eps^3(t+\eps h)\mathbb{A}_{\Sigma(\eps^2\mathbf{d})}(\eps t+\eps^2 h,\mathbf{y})[h]U'\\
&\qquad+\eps^4 (t+\eps h) a^{ij}_{\Sigma(\eps^2\mathbf{d})}(\eps t+\eps^2 h,y)\partial_i h\partial_j h U''+\eps^3\Delta_{\Sigma(\eps^2\mathbf{d})}|A_{\Sigma(\eps^2\mathbf{d})}|^2\psi_0-\eps^4\Delta_{\Sigma(\eps^2\mathbf{d})} h|A_{\Sigma(\eps^2\mathbf{d})}|^2\psi'_0\\
&\qquad -2\eps^4\nabla_{\Sigma(\eps^2\mathbf{d})}h\nabla_{\Sigma(\eps^2\mathbf{d})}|A_{\Sigma(\eps^2\mathbf{d})}|^2\psi'_0
+\eps^5|\nabla_{\Sigma(\eps^2\mathbf{d})} h|^2|A_{\Sigma(\eps^2\mathbf{d})}|^2\psi''_0\\
&\qquad -\eps^3((t+\eps h)|A_{\Sigma(\eps^2\mathbf{d})}|^2+\eps(t+\eps h)^2\tr A_{\Sigma(\eps^2\mathbf{d})}^3+\eps^2(t+\eps h)^3\tilde{\Q}_{\Sigma(\eps^2\mathbf{d})}(\eps t+\eps^2 h,\mathbf{y}))|A_{\Sigma(\eps^2\mathbf{d})}|^2\psi'_0\\
&\qquad+\eps^4 (t+\eps h)\mathbb{A}_{\Sigma(\eps^2\mathbf{d})}(\eps t+\eps^2 h,\mathbf{y})[|A_{\Sigma(\eps^2\mathbf{d})}|^2]\psi_0-\eps^5(t+\eps h)|A_{\Sigma(\eps^2\mathbf{d})}|^2\psi'_0\mathbb{A}_{\Sigma(\eps^2\mathbf{d})}(\eps t+\eps^2 h,\mathbf{y})[h]\\
&\qquad-2\eps^5 (t+\eps h) a^{ij}_{\Sigma(\eps^2\mathbf{d})}(\eps t+\eps^2 h,\mathbf{y})\partial_i h\partial_{j}|A_{\Sigma(\eps^2\mathbf{d})}|^2\psi'_0+\eps^6 (t+\eps h)a^{ij}_{\Sigma(\eps^2\mathbf{d})}(\eps t+\eps^2 h,\mathbf{y})\partial_i h\partial_j h|A_{\Sigma(\eps^2\mathbf{d})}|^2\psi''_0.
\end{aligned}
\end{equation}
In order to correct the term of order $\eps$, $\psi_0$ must be a solution to the ODE
$$\psi''_0-\eps H_\Sigma \psi'_0+f'(U)\psi_0=(t+\eps\lambda) U'.$$
The presence of the Lagrange multiplier is due to the fact that, in order to obtain an exponentially decaying solution, it is enough to impose the right-hand side to be $L^2(\R)$-orthogonal to $U'e^{-\eps H_\Sigma t}$. This condition is equivalent to say that the Lagrange multiplier satisfies
$$\eps\lambda=-\frac{\int_\R t(U')^2 e^{-\eps H_\Sigma t}dt}{\int_\R (U')^2 e^{-\eps H_\Sigma t}dt},$$
which yields that $\lambda=O(1)$, due to (\ref{construction_U}) and the fact that $$\int_\R t(v'_\star)^2 dt=0.$$
We stress that here it is crucial to use the fact that the term to correct is almost $L^2(\R)$-orthogonal to $U'$, otherwise the Lagrange multiplier would be of the same order in $\eps$ as left-hand side, and adding the term $\eps^2|A_{\Sigma(\eps^2\mathbf{d})}|^2 \psi_0$ would not improve the size of the error. Also the construction of $\psi_0$ can be found in \cite{HK1}. In conclusion, the error is given by
\begin{equation}
\label{first_error}
\begin{aligned}
\eps\Delta u_0(\mathbf{d})+\eps^{-1}f(u_0(\mathbf{d}))-\ell_\eps &=-\eps^2(\Delta_{\Sigma(\eps^2\mathbf{d})} h+|A_{\Sigma(\eps^2\mathbf{d})}|^2 h)U'-\eps^2\tr A_{\Sigma(\eps^2\mathbf{d})}^3 t^2 U'\\
&\quad +\eps^2\lambda|A_{\Sigma(\eps^2\mathbf{d})}|^2 U'-(H_{\Sigma(\eps^2\mathbf{d})}-H_\Sigma)U'+\eps^3\mathcal{G}^1_\eps(\mathbf{h},\mathbf{d})(\mathbf{y},t),
\end{aligned}
\end{equation}
where $\mathcal{G}^1_\eps(\mathbf{h},\mathbf{d})$ recollects all the terms in (\ref{exp_error_t}) of order at least $\eps^3$, including $-\eps^2|A_{\Sigma(\eps^2\mathbf{d})}|^2\psi'_0(H_{\Sigma(\eps^2\mathbf{d})}-H_\Sigma)$, which is actually of order $\eps^4$, due to Lemma \ref{lemma_H_Sigmad}. This remainder turns out to be exponentially decaying in $t$, at any rate strictly smaller than $\sqrt{2}$ at $\pm\infty$. However, the error is not decaying along the surface, in other words it is not in $\mathcal{E}^{0,\alpha}_{a,\gamma}(\Sigma(\eps^2\mathbf{d})\times\R)$, and this does not allow to solve the problem, since the inverse of the Jacobi operator of $\Sigma$ is defined in a space of decaying functions (see Proposition \ref{prop_non_deg}). Therefore, we need to improve the approximate solution, at least along the ends.

\subsection{Improvement of the approximation along the ends}\label{section_improvement}

\subsubsection{Asymptotically periodic error}

First we introduce some function spaces. For $\tau\in(0,1)$, $\gamma>0$, $n\in\N$ and $\alpha\in(0,1)$, we say that a function $\varphi\in C^{n,\alpha}_{loc}(\R^2)$ is in $\D^{n,\alpha}_{\tau,\gamma}(\R^2)$ if it satisfies the symmetries of the Delaunay surface $D_\tau$, that is
\begin{eqnarray}
\varphi(s,t)=\varphi(s+s_\tau,t)=\varphi(s_\tau-s,t), \qquad \forall\, (s,t)\in\R^2,\label{symm_delaunay}
\end{eqnarray}
and the weighted norm 
\begin{eqnarray}
\|\varphi\|_{E^{n,\alpha}_\gamma(\R^2)}:=\sum_{0\le m+k\le n}\eps^m \|\partial_t^k \partial^m_s\varphi\|_{\D^{0,\alpha}_{\gamma}(\R^2)},\qquad\|\varphi\|_{\D^{0,\alpha}_\gamma(\R^2)}:=\|\varphi \cosh^\gamma(t)\|_{C^{0,\alpha}(\R^2)}\label{def_wnorm_end_eps}
\end{eqnarray}
is finite. \\

Along the ends of $\Sigma(\eps^2\mathbf{d})$, the error given by (\ref{first_error}) is asymptotically periodic, in a sense that we will explain just below. We take, as an example, the term $\tr A_{\Sigma(\eps^2\mathbf{d})}^3 t^2 U'(t)$, which appears in (\ref{first_error}). We define $\zeta_0:\Sigma(\eps^2\mathbf{d})\to\R$ to be a  smooth cutoff function such that  $\zeta_0$ and  $\{\zeta_j\}_{1\le j\le k}$ form a partition of unity on $\Sigma(\eps^2\mathbf{d})$, so that 
\begin{eqnarray}\notag
\tr A_{\Sigma(\eps^2\mathbf{d})}^3(\mathbf{y}) t^2 U'(t)=\zeta_0(\mathbf{y}) \tr A_{\Sigma(\eps^2\mathbf{d})}^3(\mathbf{y}) t^2 U'(t)+\sum_{j=1}^k \zeta_j(\mathbf{y}) \tr A_{\Sigma(\eps^2\mathbf{d})}^3(\mathbf{y}) t^2 U'(t).
\end{eqnarray}
Using the coordinates $(s,\vartheta)$ along each of the ends, we decompose
\begin{equation}
\label{traceA3}
\begin{aligned}
&(\zeta_j \tr A_{\Sigma(\eps^2\mathbf{d})}^3)(Y_j(\eps^2\mathbf{d}_j)(s,\vartheta)) t^2 U'(t)=\tilde{\xi}(s)\tr A^3_{D_{\tau_j}(\eps^2\mathbf{d}_j)}(s)t^2 U'(t)\\
&\quad +\tilde{\xi}(s)\left[\tr A^3_{\Sigma(\eps^2\mathbf{d})}(Y_j(\eps^2\mathbf{d}_j)(s,\vartheta)) -\tr A^3_{D_{\tau_j}(\eps^2\mathbf{d}_j)}(s)\right]t^2 U'(t),
\end{aligned}
\end{equation}
where $\tr A^3_{D_{\tau_j}(\eps^2\mathbf{d}_j)}(s)t^2 U'(t)$ is in the space $\D^{0,\alpha}_{\tau_j(\eps^2\mathbf{d}_j),\gamma}(\R^2)$, since $A_{D_{\tau_j}(\eps^2\mathbf{d}_j)}(s)$, the second fundamental form of $D_{\tau_j}(\eps^2\mathbf{d}_j)$ at $X_{\tau_j}(\eps^2\mathbf{d}_j)(s,\vartheta)$,  depends only on $s$ and satisfies the symmetries (\ref{symm_delaunay}). The other term is decaying both in $t$ and in $s$, more precisely it is the product of a function of $s\in(0,\infty)$ and $\vartheta\in S^1$, which measures the difference between the second fundamental form of $\Sigma(\eps^2\mathbf{d})$ and the one of the Delaunay surfaces $D_{\tau_j}(\eps^2\mathbf{d}_j)$ to which the ends are asymptotic, and the function $t^2 U'(t)$, which decays exponentially at any rate $\gamma\in (0,\sqrt{2})$. Roughly, if we imagined to replace the ends of $\Sigma(\eps\mathbf{d})$ with exact Delaunay surfaces and we took $h_0=0$, the error would exactly coincide with the periodic part. Setting $\textbf{h}:=(h_0,{\tt h}_1,\dots,{\tt h}_k)$ and using a similar argument to control the other terms in (\ref{first_error}),
%\begin{eqnarray}
%\bar{\gamma}_\eps:=\min\{\sqrt{-f'(1+\sigma_\eps^+)},\sqrt{-f'(-1+\sigma_\eps^-)}\}.\label{upper_bound_gamma}
%\end{eqnarray}
 %we can decompose 
%$$\mathcal{G}^1_\eps(\mathbf{h},\mathbf{d})(\mathbf{y},t)=\sum_{j=1}^k \zeta_j(\mathbf{y})\tilde{\mathcal{G}^1}_{\eps,j}(\text{h}_j,\mathbf{d}_j)(\mathbf{y},t)
%+\check{\mathcal{G}}^1_\eps(\mathbf{h},\mathbf{d})(\mathbf{y},t),$$
%where $\tilde{\mathcal{G}}_{\eps,j}(\text{h}_j,\mathbf{d})(Z_j(s,\vartheta),t)$ is actually a function in $\D^{0,\alpha}_{\tau_j+\eps^2\mathbf{d}^D_j,\gamma}(\R^2)$ and $\check{\mathcal{G}}_\eps(\mathbf{h},\mathbf{d})\in \mathcal{C}^{0,\alpha}_{a,\gamma}(\Sigma(\mathbf{d})\times\R)$ is decaying along the surface. As a result, 
we have
\begin{equation}
\begin{aligned}
\label{asymptotically_periodic_error}
\eps\Delta u_0(\mathbf{d})+\eps^{-1}f(u_0(\mathbf{d}))-\ell_\eps &=-\eps^2(\Delta_{\Sigma(\eps^2\mathbf{d})}+|A_{\Sigma(\eps^2\mathbf{d})}|^2) (h_0\circ\beta(\eps^2\mathbf{d}))U'-(H_{\Sigma(\eps^2\mathbf{d})}-H_\Sigma)U'\\
&\quad +\eps^2\mathcal{G}^2_\eps(\mathbf{h},\mathbf{d})+\sum_{j=1}^k \zeta_j\mathcal{H}^1_\eps(\text{h}_j,\mathbf{d}_j),
\end{aligned}
\end{equation}
where $\mathcal{G}^2_\eps(\mathbf{h},\mathbf{d})\in\mathcal{E}^{0,\alpha}_{a,\gamma}(\Sigma(\eps^2\mathbf{d})\times\R)$ is bounded uniformly in $\mathbf{d}$, $\mathbf{h}$ and $\eps$, in the sense that 
\begin{eqnarray}
\|\mathcal{G}^2_\eps(\mathbf{h},\mathbf{d})\|_{\mathcal{E}^{0,\alpha}_{a,\gamma}(\Sigma(\eps^2\mathbf{d})\times\R)}\le \tilde{C}_1(K),\label{est_G2}
\end{eqnarray}
for some constant $\tilde{C}_1(K)>0$ depending on $K$, due to the presence of commutator  terms like $[\Delta_{\Sigma(\eps^2\mathbf{d})},\zeta_j]h_j$, but independent of $\mathbf{d}$, $\mathbf{h}$, $M$, and $\eps$, while $\mathcal{H}^1_\eps({\tt h}_j,\mathbf{d}_j)$ recollects all the periodic terms. Here $a\in(0,\bar{a})$ and $\alpha\in(0,1)$ are arbitrary. It turns out that
\begin{eqnarray}
\mathcal{H}^1_\eps({\tt h}_j,\mathbf{d}_j)\circ Y_j(\eps^2\mathbf{d}_j):=-\eps^2 \tilde{L}_{0,\tau_j(\eps^2\mathbf{d}_j)}{\tt h}_j U'
+\eps^2\mathcal{H}^2_\eps({\tt h}_j,\mathbf{d}_j),\label{per_error}
\end{eqnarray}
where
\begin{eqnarray}
\tilde{L}_{0,\tau}:=\frac{1}{\tau^2 e^{2\sigma_\tau}}L_{0,\tau},\qquad L_{0,\tau}:=\partial^2_s+\tau^2\cosh(2\sigma_\tau)\label{def_L_tau}
\end{eqnarray}
and, for any $\gamma\in(0,\sqrt{2})$, $\alpha\in(0,1)$ and $\eps$ small, $\mathcal{H}^2_\eps({\tt h}_j,\mathbf{d}_j)$ is actually a function in $\D^{0,\alpha}_{\tau_j(\eps^2\mathbf{d}_j),\gamma}(\R^2)$, such that
\begin{eqnarray}
\|\mathcal{H}^2_\eps({\tt h}_j,\mathbf{d}_j)\|_{E^{0,\alpha}_\gamma(\R^2)}\le C.\label{est_H2}
\end{eqnarray}
Even in this case, the constant $C>0$ appearing in (\ref{est_H2}) is independent of $\mathbf{d}_j$, $\text{h}_j$, $M$, $K$ and $\eps$. The Lipschitz dependence on the data $\mathbf{d}$ and $\mathbf{h}$ is dealt with in the following Lemma. This is a quite technical issue, however it is worth to mention it, since we will use the size of the Lipschitz constants in the forthcoming proofs, in order to apply the contraction mapping theorem.
\begin{lemma}
Let $M>0$, $K>0$, $a\in(0,\bar{a})$, $\gamma\in(0,\sqrt{2})$ and $\alpha\in(0,1)$. Then there exists a constant $\tilde{C}_2(K)>0$ depending on $K$ such that, for $\eps$ small enough,
\begin{equation}\notag
\begin{aligned}
&\|\mathcal{G}^2_\eps(\mathbf{h}^1,\mathbf{d}^1)\circ(\beta(\eps^2\mathbf{d}^1)^{-1}\times {\mathop{Id}}_\R)-\mathcal{G}^2_\eps(\mathbf{h}^2,\mathbf{d}^2)\circ(\beta(\eps^2\mathbf{d}^2)^{-1}\times {\mathop{Id}}_\R)\|_{\mathcal{E}^{0,\alpha}_{a,\gamma}
(\Sigma\times\R)}\\
&\le \tilde{C}_2(K) \big(\eps^2|\mathbf{d}^1-\mathbf{d}^2|+\sum_{j=1}^k\|{\tt h}^1_j-{\tt h}^2_j\|_{C^{2,\alpha}(\R)}+\eps \|h^1_0-h^2_0\|_{\mathcal{C}^{2,\alpha}_a(\Sigma\times\R)}\big),
\end{aligned}
\end{equation}
for any $\mathbf{d}^i\in (B_M)^k$, for any $\mathbf{h}^i:=(h^i_0,{\tt h}^i_1,\dots,{\tt h}^i_k)\in \mathbf{H}(\mathbf{d}^i)$, $i=1,2$ (see (\ref{def_H}) for the definition of $\mathbf{H}(\mathbf{d})$).
%\textit{(ii)}
%\begin{eqnarray}\notag
%\|\mathcal{G}^2_\eps((h^1_0,\text{h}_1,\dots,\text{h}_k),\mathbf{d})\circ(\beta(\mathbf{d})^{-1}\times {\mathop{Id}}_\R)
%-\mathcal{G}^2_\eps((h^1_0,\text{h}_1,\dots,\text{h}_k),\mathbf{d})\circ(\beta(\mathbf{d})^{-1}\times {\mathop{Id}}_\R)\|_{\mathcal{E}^{0,\alpha}_{a,\gamma}(\Sigma\times\R)}\\\notag
%\le C\eps \|h^1_0-h^2_0\|_{\mathcal{C}^{2,\alpha}_a(\Sigma\times\R)},
%\end{eqnarray}
%for any $\mathbf{h}^1:=(h^1_0,\text{h}_1,\dots,\text{h}_k),\, \mathbf{h}^2:=(h^2_0,\text{h}_1,\dots,\text{h}_k)\in \mathbf{H}(\mathbf{d})$ and $\mathbf{d}\in (B_M)^k$.\\

%\textit{(iii)} 
%\begin{eqnarray}\notag
%\|\mathcal{G}^2_\eps((h_0,\text{h}^1_1,\dots,\text{h}_k),\mathbf{d})\circ(\beta(\mathbf{d})^{-1}\times Id_\R)
%-\mathcal{G}^2_\eps((h_0,\text{h}^2_1,\dots,\text{h}_k),\mathbf{d})\circ(\beta(\mathbf{d})^{-1}\times Id_\R)\|_{\mathcal{E}^{0,\alpha}_{a,\gamma}(\Sigma\times\R)}\\\notag
%\le C\eps \|\text{h}^1_1-\text{h}^2_1\|_{C^{2,\alpha}(\R)},
%\end{eqnarray}
%for any $\mathbf{h}^1:=(h_0,\text{h}^1_1,\dots,\text{h}_k),\, \mathbf{h}^2:=(h_0,\text{h}^2_1,\dots,\text{h}_k)\in \mathbf{H}$ and $\mathbf{d}\in (B_M)^k$. (The same is true for any component $\text{h}_j$, $1\le j\le k$.)
\label{lemma_lip_G2}
\end{lemma}
\begin{proof}
We treat, as an example, the term coming from $\tr A^3_{\Sigma(\eps^2\mathbf{d})}$. According to (\ref{traceA3}), along any of the ends, it is the sum of a periodic term, given by $\tr A_{D_{\tau_j}(\eps^2\mathbf{d}^j)}^3$, which is multiplied by a cutoff function and is not involved in $\mathcal{G}_\eps^2$, and a decaying term, that is 
$$\tilde{\xi}(\tr A^3_{\Sigma(\eps^2\mathbf{d})}\circ Y_j(\eps^2\mathbf{d}_j)-\tr A_{D_{\tau_j}(\eps^2\mathbf{d}^j)}^3),$$
which fulfils the required estimate, due to Lemma \ref{lemma_Sigma_d}. All the other terms of (\ref{exp_error_t}) are similar, including the ones involving $h$ and its derivatives, since, by definition, 
\[h_0\circ\beta(\eps^2\mathbf{d}^1)\circ Y_j(\eps^2\mathbf{d}^1_j)=h_0\circ\beta(\eps^2\mathbf{d}^2)\circ Y_j(\eps^2\mathbf{d}^2_j)=h_0\circ Y_j.\] 
The size in $\eps$ of the Lipschitz constant follows from a careful analysis of the size in $\eps$ of all the terms in (\ref{exp_error_t}).
\end{proof}
Similar arguments hold for the remainder in the periodic term $\mathcal{H}^2_\eps({\tt h}_j,\mathbf{d}^1_j)$, whose main term is given by $-\tr A_{D_{\tau_j}(\eps^2\mathbf{d}_j)}^3 t^2 U'+\lambda |A_{D_{\tau_j}(\eps^2\mathbf{d}_j)}|^2 U'$. The difference is that this term depends on the geometric quantities of $D_{\tau_j}(\eps^2\mathbf{d}_j)$ and on ${\tt h}_j$, thus it also fulfils the symmetries of the Delaunay surface $D_{\tau_j}(\eps^2\mathbf{d}_j)$, that is it is even with respect to $s_{\tau_j(\eps^2\mathbf{d}_j)/2}$.
\begin{lemma}
Let $M>0$, $K>0$, $\gamma\in(0,\sqrt{2})$ and $\alpha\in(0,1)$. Then there exists a constant $C>0$ such that, for $\eps$ small enough,
$$\|\mathcal{H}^2_\eps({\tt h}^1_j,\mathbf{d}^1_j)-\mathcal{H}^2_\eps({\tt h}^2_j,\mathbf{d}^2_j)\|_{E^{0,\alpha}_\gamma(\R^2)}\le C (\eps\|{\tt h}^1_j-{\tt h}^2_j\|_{C^{2,\alpha}(\R)}+\eps^2|\mathbf{d}^1_j-\mathbf{d}^2_j|),$$
for any $\mathbf{d}^i_j\in B_M$, for any ${\tt h}^i_j\in D^{2,\alpha}_{\tau_j(\eps^2\mathbf{d}^i_j)}(\R)$, with $\|{\tt h}^i_j\|_{C^{2,\alpha}(\R)}<K$, $i=1,2,\,1\le j\le k$.
\label{lemma_lip_H2}
\end{lemma}
The proof is similar to the one of Lemma \ref{lemma_lip_G2}.

\subsubsection{Corrections along the ends: a Lyapunov-Schmidt reduction.}

The idea is to choose a family of approximate solutions so that the periodic terms in the error vanish. We take $\varphi_j\in \D^{2,\alpha}_{\tau_j(\eps^2\mathbf{d}_j),\gamma}(\R^2)$ and we define our new approximation as
$$v(\mathbf{d})(\mathbf{y},t):=u_0(\mathbf{d})(\mathbf{y},t)+\sum_{j=1}^k \zeta_j(\mathbf{y})\phi_j(\mathbf{y},t), \qquad \phi_j(Y_j(\eps^2\mathbf{d}_j)(s,\vartheta),t):=\varphi_j(s,t)\in \D^{2,\alpha}_{\tau_j(\eps^2\mathbf{d}_j),\gamma}(\R^2).$$
The corrections $\varphi_j$ will be determined to correct the periodic part of the error. We are interested in computing the new error, at least near $\Sigma(\eps^2\mathbf{d})$. For this purpose, we introduce the smooth cutoff function $\chi:\R\to\R$ equal to $1$ in $(-\infty,1)$ and equal to $0$ in $(2,\infty)$ and we set
$$\chi_{l,\eps}(t):=\chi(\frac{\eps|t|}{\delta}-(l-1)), \qquad l=1,\dots 5.$$
Far from the surface, our approximate solution will coincide with the constant solutions $\pm 1+\sigma_\eps^\pm$, as we will see in section \ref{section_proof}. Now we introduce the operator
\begin{eqnarray}\notag
\mathbb{L}_{\eps,\mathbf{d}}:=\eps\Delta_{\Sigma(\eps^2\mathbf{d})}+\eps^{-1}\partial^2_t+\eps^{-1}f'(v_\star)+\text{L}_{\eps,\mathbf{d}},
\end{eqnarray}
with
\begin{equation}
\begin{aligned}
\text{L}_{\eps,\mathbf{d}}:=&\eps^{-1}\big(f'(u_0(\mathbf{d}))-f'(v_\star)\big)+\eps^2 \Delta_{\Sigma(\eps^2\mathbf{d})} h \partial_t-2\eps^2\nabla_{\Sigma(\eps^2\mathbf{d})} h\nabla_{\Sigma(\eps^2\mathbf{d})}\partial_t+\eps^3|\nabla_{\Sigma(\eps^2\mathbf{d})} h|^2\partial_t^2\\
& +(H_{\Sigma(\eps^2\mathbf{d})}+(\eps t+\eps^2 h)\chi_{4,\eps}|A_{\Sigma(\eps^2\mathbf{d})}|^2+(\eps t+\eps^2 h)^2\chi_{4,\eps}\mathbb{Q}_{\Sigma(\eps^2\mathbf{d})}(\mathbf{y},\eps t+\eps^2 h))\partial_t\\
& +\eps^2 (t+\eps h)\chi_{4,\eps}\mathbb{A}_{\Sigma(\eps^2\mathbf{d})}(\mathbf{y},\eps t+\eps^2 h)-\eps^3 (t+\eps h)\chi_{4,\eps}\mathbb{A}_{\Sigma(\eps^2\mathbf{d})}(\mathbf{y},\eps t+\eps^2 h)[h]\partial_t\\
& -2\eps^3 (t+\eps h)\chi_{4,\eps}a_{\Sigma(\eps^2\mathbf{d})}^{ij}(\mathbf{y},\eps t+\eps^2 h)\partial_i h\partial_{tj}+\eps^4 (t+\eps h)\chi_{4,\eps}a_{\Sigma(\eps^2\mathbf{d})}^{ij}(\mathbf{y},\eps t+\eps^2 h)\partial_i h\partial_j h\partial_t^2.
\end{aligned}
\end{equation}
This operator is basically the linearization of the Cahn-Hilliard equation (\ref{cahn_hilliard}) around $u_0(\mathbf{d})$ expressed in the shifted Fermi coordinates $(\mathbf{y},t)$ of $\Sigma(\mathbf{d})$, which is involved in the computation of the error produced by the new approximation. The cutoff function $\chi_{4,\eps}$ is introduced in order to be able to estimate $\eps|t|\chi_{4,\eps}\le 5\delta$, which is not true for $\eps|t|$, and this would cause problems. In these notations, we have
\begin{eqnarray}
\chi_{3,\eps}(\eps\Delta v(\mathbf{d})+\eps^{-1}f(v(\mathbf{d}))-\ell_\eps)=\chi_{3,\eps}(\eps\Delta u_0(\mathbf{d})+\eps^{-1}f(u_0(\mathbf{d}))-\ell_\eps)\\\notag
+\chi_{3,\eps}\sum_{j=1}^k[\mathbb{L}_{\eps,\mathbf{d}},\zeta_j]\phi_j
+\chi_{3,\eps}\sum_{j=1}^k\zeta_j \mathbb{L}_{\eps,\mathbf{d}}\phi_j+\chi_{3,\eps}\eps^{-1}Q_{u_0(\mathbf{d})}(\sum_{j=1}^k\zeta_j\phi_j),
\end{eqnarray}
where $Q_v(t):=-t^2(3v+t)$ is quadratic in $t$. We stress that, since the cutoff functions $\zeta_j$ do not depend on $t$, the commutators satisfy
\begin{eqnarray}
\|[\mathbb{L}_{\eps,\mathbf{d}},\zeta_j]\phi_j\|_{\mathcal{E}^{0,\alpha}_{a,\gamma}(\Sigma\times\R)}\le C\|\varphi_j\|_{E^{2,\alpha}_\gamma(\R^2)}.\label{est_commutators}
\end{eqnarray}
Recollecting all this information, using (\ref{asymptotically_periodic_error}) and the decay of $v_j$, we can rewrite the error as
\begin{equation}
\begin{aligned}
&\chi_{3,\eps}(\eps\Delta v(\mathbf{d})+\eps^{-1}f(v(\mathbf{d}))-\ell_\eps)=\chi_{3,\eps}\big(-\eps^2\mathcal{J}_{\Sigma(\eps^2\mathbf{d})}(h_0\circ \beta(\eps^2\mathbf{d}))U'-(H_{\Sigma(\eps^2\mathbf{d})}-H_\Sigma)U'\big)\label{error_correct}\\
&+\chi_{3,\eps}\mathcal{G}^3_\eps(\mathbf{h},\mathbf{d},\mathbf{\varphi})
+\chi_{3,\eps}\sum_{j=1}^k\zeta_j \mathcal{H}^3_\eps({\tt h}_j,\mathbf{d}_j,\varphi_j),
\end{aligned}
\end{equation}
where the new periodic term is given by
\begin{equation}
\label{per_error_all}
\begin{aligned}
&\mathcal{H}^3_\eps({\tt h}_j,\mathbf{d}_j,\varphi_j)\circ Y_j(\eps^2\mathbf{d}_j)=\eps (\tau_j(\eps^2\mathbf{d}_j) e^{\sigma_{\tau_j(\eps^2\mathbf{d}_j)}})^{-2}\partial^2_s\varphi_j+\eps^{-1}\partial^2_t\varphi_j+\eps^{-1}f'(v_\star)\varphi_j\\
&-\eps^2 \tilde{L}_{\tau_j(\eps^2\mathbf{d}_j)}{\tt h}_j U'+\eps^2\mathcal{H}^2_\eps({\tt h}_j,\mathbf{d}_j)+\mathcal{H}^4_\eps({\tt h}_j,\mathbf{d}_j,\varphi_j),
\end{aligned}
\end{equation}
(see (\ref{per_error})). %We recall that $g_{ss}(\mathbf{d}_j)=$ is the coefficient of the metric of the Delaunay surface $D_{\tau_j}(\mathbf{d}_j)$, obtained by translating, rotating and changing the Delaunay parameter to $D_{\tau_j}$. 
The remainder $\mathcal{H}^4_\eps({\tt h}_j,\mathbf{d}_j,\varphi_j)\in \D^{0,\alpha}_{\tau_j(\eps^2\mathbf{d}_j),\gamma}(\R^2)$ recollects all the quadratic terms in $\varphi_j$ and the linear terms coming from the periodic part of $\text{L}_{\eps,\mathbf{d}}$, and satisfies
\begin{eqnarray}
\|\mathcal{H}^4_\eps({\tt h}_j,\mathbf{d}_j,\varphi_j)\|_{E^{0,\alpha}_\gamma(\R^2)}\le C(\|\varphi_j\|_{E^{2,\alpha}_\gamma(\R^3)}+\eps^{-1}\|\varphi_j\|^2_{E^{2,\alpha}_\gamma(\R^3)}).\label{est_H4}
\end{eqnarray}
and 
\begin{eqnarray}
\mathcal{G}^3_\eps(\mathbf{h},\mathbf{d},\mathbf{\varphi})=\eps^2\mathcal{G}^2_\eps(\mathbf{h},\mathbf{d})
+\mathcal{G}^4_\eps(\mathbf{h},\mathbf{d},\mathbf{\varphi}),\label{def_G4}
\end{eqnarray}
with $\mathcal{G}^4_\eps(\mathbf{h},\mathbf{d},\mathbf{\varphi})\in\mathcal{E}^{0,\alpha}_{a,\gamma}(\Sigma(\mathbf{d})\times\R)$ satisfying
\begin{eqnarray}
\|\mathcal{G}^4_\eps(\mathbf{h},\mathbf{d},\mathbf{\varphi})\|_{\mathcal{E}^{0,\alpha}_{a,\gamma}(\Sigma(\mathbf{d})\times\R)}\le 
C\sum_{j=1}^k(\|\varphi_j\|_{E^{2,\alpha}_\gamma(\R^2)}+\eps^{-1}\|\varphi_j\|^2_{E^{2,\alpha}_\gamma(\R^2)}).\label{est_G4}
\end{eqnarray}
(See also (\ref{est_commutators})). For future reference, we underline that also $\mathcal{H}^4_\eps$ depends on the data in a Lipschitz way. As regards the dependence on $\mathbf{d}_j$ and ${\tt h}_j$, we have
\begin{equation}
\label{lip_H4_varphi}
\begin{aligned}
&\quad\|\mathcal{H}^4_\eps({\tt h}^1_j,\mathbf{d}^1_j,\varphi^1_j)-\mathcal{H}^4_\eps({\tt h}^2_j,\mathbf{d}^2_j,\varphi^2_j)\|
_{E^{0,\alpha}_\gamma(\R^2)}\le\\
&\quad C(\eps\|{\tt h}^1_j-{\tt h}^2_j\|_{C^{2,\alpha}(\R)}+\eps^2|\mathbf{d}^1-\mathbf{d}^2|)(\|\varphi^1_j\|_{E^{2,\alpha}_\gamma(\R^2)}
+\|\varphi^2_j\|_{E^{2,\alpha}_\gamma(\R^2)})\\ 
&\quad C\|\varphi^1_j-\varphi^2_j\|_{E^{2,\alpha}_\gamma(\R^2)} (1+\eps^{-1}\|\varphi^1_j\|_{E^{2,\alpha}_\gamma(\R^2)}+\eps^{-1}\|\varphi^2_j\|_{E^{2,\alpha}_\gamma(\R^2)}),
\end{aligned}
\end{equation}
for any $\varphi^i_j\in \D^{2,\alpha}_{\tau_j(\eps^2\mathbf{d}^i_j),\gamma}(\R^2)$ with $\|\varphi^i_j\|_{E^{2,\alpha}_\gamma(\R^2)}<1$, for any ${\tt h}^i_j\in D^{2,\alpha}_{\tau_j(\eps^2\mathbf{d}_j)}(\R)$, $\|{\tt h}^i_j\|_{C^{2,\alpha}(\R)}<K$, for any $\mathbf{d}^i_j\in B_M$, $i=1,2$, $1\le j\le k$.\\ 

The behaviour of $\mathcal{G}^4_\eps$ is similar, in the sense that, in the previous notations, the dependence on $\mathbf{d}$ and $\textbf{h}$ is given by  
\begin{equation}
\label{lip_G4_varphi}
\begin{aligned}
&\|\mathcal{G}^4_\eps(\mathbf{h}^1,\textbf{d}^1,\mathbf{\varphi}^1)\circ(\beta(\eps^2\mathbf{d}^1)^{-1}\times Id_\R)
-\mathcal{G}^4_\eps(\mathbf{h}^2,\textbf{d}^2,\mathbf{\varphi}^2)\circ(\beta(\eps^2\mathbf{d}^2)^{-1}\times Id_\R)\|_{\mathcal{E}^{0,\alpha}_{a,\gamma}(\Sigma\times\R)}\\
&\le C(\eps\sum_{j=1}^k \|{\tt h}^1_j-{\tt h}^2_j\|_{C^{2,\alpha}(\R)}+\eps\|h^1_0-h^2_0\|_{\mathcal{C}^{2,\alpha}_a(\Sigma)}+
\eps^2|\mathbf{d}^1-\mathbf{d}^2|)\big(\sum_{j=1}^k (\|\varphi^1_j\|_{E^{2,\alpha}_\gamma(\R^2)}+\|\varphi^2_j\|_{E^{2,\alpha}_\gamma(\R^2)})\big)\\
&+C\sum_{j=1}^k \|\varphi^1_j-\varphi^2_j\|_{E^{2,\alpha}_\gamma(\R^2)} (1+\eps^{-1}\|\varphi^1_j\|_{E^{2,\alpha}_\gamma(\R^2)}+\eps^{-1}\|\varphi^2_j\|_{E^{2,\alpha}_\gamma(\R^2)}),
\end{aligned}
\end{equation}
for any $\mathbf{d}^i \in(B_M)^k$, $\mathbf{h}^i\in\mathbf{H}(\mathbf{d}^i)$ and $\mathbf{\varphi}^i=(\varphi^1_1,\dots,\varphi^k_k)\in D^{2,\alpha}_{\tau(\eps^2\mathbf{d}^1),\gamma}(\R^2)\times\dots\times D^{2,\alpha}_{\tau(\eps^2\mathbf{d}^k),\gamma}(\R^2)$, $i=1,2$. Once again, in (\ref{lip_H4_varphi}) and in (\ref{lip_G4_varphi}), the constants $M>0$ and $K>0$ involved in the definition of $\mathbf{H}(\mathbf{d}^i)$, the decay rates $a\in(0,\bar{a})$ and $\gamma\in(o,\sqrt{2})$ and the Holder constant $\alpha\in(0,1)$ are arbitrary, and $\eps$ is small. In order to get a decaying error, we have to choose, for any $\mathbf{d}_j\in B_M$, ${\tt h}_j:={\tt h}_j(\mathbf{d}_j)$ and $\varphi_j:=\varphi_j(\mathbf{d}_j)$ in such a way that
\begin{eqnarray}
\mathcal{H}^3_\eps({\tt h}_j,\varphi_j,\mathbf{d}_j)=0, \qquad 1\le j\le k.\label{eq_ls_end}
\end{eqnarray}
\begin{proposition}
Let $M>0$, $\alpha\in(0,1/2)$, $\gamma\in(0,\sqrt{2})$ and $\mathbf{d}_j\in B_M$. Then there exist $K>0$ such that, for $\eps$ small enough, there exist ${\tt h}_j:={\tt h}_j(\mathbf{d}_j)\in D^{2,\alpha}_{\tau_j(\eps^2\mathbf{d}_j)}(\R)$ and $\varphi_j:=\varphi_j(\mathbf{d}_j)\in \D^{2,\alpha}_{\tau_j(\eps^2\mathbf{d}_j),\gamma}(\R^2)$ such that (\ref{eq_ls_end}) is satisfied. Moreover, ${\tt h}_j$ and $\varphi_j$ enjoy the following properties
\begin{enumerate}
\item $\|{\tt h}_j\|_{C^{2,\alpha}(\R)}\le K$.
\item $\|\varphi_j\|_{E^{2,\alpha}_\gamma(\R^2)}\le \Lambda_1\eps^{3-\alpha}$, for some constant $\Lambda_1>0$ independent of $\mathbf{d}_j$.
\item $$\int_\R \varphi_j(s,t)v'_\star(t)dt=0, \qquad\forall\, s\in\R.$$
\item Moreover,
\begin{eqnarray}
\|\varphi_j(\mathbf{d}_j^1)-\varphi_j(\mathbf{d}_j^2)\|_{E^{2,\alpha}_\gamma(\R^2)}\le C\eps^{5-\alpha}|\mathbf{d}^1_j-\mathbf{d}^2_j|\label{lip_varphi_d}\\
\|{\tt h}_j(\mathbf{d}_j^1)-{\tt h}_j(\mathbf{d}_j^2)\|_{C^{2,\alpha}(\R)}\le C\eps^2|\mathbf{d}^1_j-\mathbf{d}^2_j|\label{lip_h_d},
\end{eqnarray}
for some constant $C>0$, for any $\mathbf{d}^i_j\in B_M$, $i=1,2$.
\end{enumerate}
\label{prop_ls_end}
\end{proposition}
We note that, in Proposition \ref{prop_ls_end}, $M>0$ is still arbitrary, while $K>0$ has to be chosen large enough. The proof is based on a Lyapunov-Schmidt reduction. We note that here we need to take $\alpha\in(0,1/2)$, in order to be able to solve some fixed point problem. For $\tau\in(0,1)$, $\gamma>0$ and $\alpha\in(0,1)$, we define the space
\begin{eqnarray}\notag
\mathcal{Z}_{\tau}:=\{\varphi\in \D^{0,\alpha}_{\tau,\gamma}(\R^2):\int_\R \varphi(s,t)v'_\star(t)dt=0,&\text{ for a. e. $s\in\R$}\},
\end{eqnarray}
and we decompose equation (\ref{eq_ls_end}) into the system
\begin{equation}
\label{eq_aux_end}
\eps (\tau_j(\eps^2\mathbf{d}_j) e^{\sigma_{\tau_j(\eps^2\mathbf{d}_j)}})^{-2}\partial^2_s\varphi_j+\eps^{-1}\partial^2_t\varphi_j+\eps^{-1}f'(v_\star)\varphi_j=
\Pi_{\mathcal{Z}_{\tau_j(\eps^2\mathbf{d}_j)}}\mathcal{H}^5_\eps({\tt h}_j,\varphi_j,\mathbf{d}_j)
\end{equation}
\begin{equation}
\label{eq_bifo_end}
(Id-\Pi_{\mathcal{Z}_{\tau_j(\eps^2\mathbf{d}_j)}})\mathcal{H}^5_\eps({\tt h}_j,\varphi_j,\mathbf{d}_j)=0,
\end{equation}
where
\begin{equation}
\label{def_H5}
\begin{aligned}
&\mathcal{H}^5_\eps({\tt h}_j,\varphi_j,\mathbf{d}_j):=\mathcal{H}^3_\eps({\tt h}_j,\varphi_j,\mathbf{d}_j)
-\big(\eps (\tau_j(\eps^2\mathbf{d}_j) e^{\sigma_{\tau_j(\eps^2\mathbf{d}_j)}})^{-2}\partial^2_s\varphi_j+\eps^{-1}\partial^2_t\varphi_j+\eps^{-1}f'(v_\star)\varphi_j\big)=\\
&-\eps^2\tilde{L}_{\tau_j(\eps^2\mathbf{d}_j)}{\tt h}_j U'+\eps^2\mathcal{H}^2_\eps({\tt h}_j,\mathbf{d}_j)+\mathcal{H}^4_\eps({\tt h}_j,\mathbf{d}_j,\varphi_j),
\end{aligned}
\end{equation}
and $\Pi_{\mathcal{Z}_\tau}$ is the projection onto $\mathcal{Z}_\tau$, defined by
\[
\Pi_{\mathcal{Z}_\tau} g(s,t):=g(s,t)-\frac{\int_\R g(s,\mu) v'_\star(\mu)d\mu}{\int_\R(v'_\star(\mu))^2 d\mu}v'_\star(t), \qquad\forall\, g\in \D^{0,\alpha}_{\tau,\gamma}(\R^2).
\]
First we fix $M,K>0$, $\mathbf{d}_j$ and ${\tt h}_j$ and we solve the \textit{auxiliary equation} (\ref{eq_aux_end}) with respect to $\varphi_j$. 
\begin{proposition}
Let $M>0$, $K>0$, $\gamma\in(0,\sqrt{2})$, $\alpha\in(0,1/2)$, $\mathbf{d}_j\in B_M$ and ${\tt h}_j\in D^{2,\alpha}_{\tau_j(\eps^2\mathbf{d}_j)}(\R)$, with $\|{\tt h}_j\|_{D^{2,\alpha}_{\tau_j(\eps^2\mathbf{d}_j)}(\R)}<K$. Then there exists a unique solution $\varphi_j:=\varphi_j({\tt h}_j,\mathbf{d}_j)\in \D^{2,\alpha}_{\tau_j(\eps^2\mathbf{d}_j),\gamma}(\R^2)\cap\mathcal{Z}_{\tau_j(\eps^2\mathbf{d}_j)}$ to (\ref{eq_aux_end}) satisfying
$$\|\varphi_j({\tt h}_j,\mathbf{d}_j)\|_{E^{2,\alpha}_\gamma(\R^2)}\le \Lambda_1\eps^{3-\alpha},$$
for some constant $\Lambda_1>0$, independent of $\mathbf{d}_j$. Moreover, there exists a constant $C>0$, independent of $\mathbf{d}_j$, such that
\begin{equation}\notag
\|\varphi_j(\text{h}^1_j,\mathbf{d}^1_j)-\varphi_j(\text{h}^2_j,\mathbf{d}^2_j)\|_{E^{2,\alpha}_\gamma(\R^2)}\le C(\eps^{4-\alpha}\|{\tt h}^1_j-{\tt h}^2_j\|_{C^{2,\alpha}(\R)}+\eps^{5-\alpha}|\mathbf{d}^1_j-\mathbf{d}^2_j|),
\end{equation}
for any $\mathbf{d}^i_j\in B_M$, ${\tt h}^i_j\in D^{2,\alpha}_{\tau_j(\eps^2\mathbf{d}^i_j)}(\R)$, with $\|{\tt h}^i_j\|_{C^{2,\alpha}(\R)}<K$, $i=1,2$, $1\le j\le k$.
\label{prop_aux_end}
\end{proposition}
The proof of this result will be given in section \ref{section_aux_end}. The next step is to solve the \textit{bifurcation equation} (\ref{eq_bifo_end}), by means of the operator $L_{0,\tau_j(\eps^2\mathbf{d}_j)}$, which is proportional to the Jacobi operator of the Delaunay unduloid $D_{\tau_j}(\eps^2\mathbf{d}_j)$ for functions satisfying the symmetries of the surface (see (\ref{def_L_tau})). In fact, plugging $\varphi_j({\tt h}_j,\mathbf{d}_j)$ in (\ref{eq_ls_end}), multiplying the equation by $v'_\star$ and integrating over $\R$ yields that (\ref{eq_bifo_end}) is equivalent to
\begin{eqnarray}\notag
\int_\R \mathcal{H}^5_\eps({\tt h}_j,\varphi_j({\tt h}_j,\mathbf{d}_j),\mathbf{d}_j)(s,t) v'_\star(t) dt=0, \qquad\forall\, s\in\R,
\end{eqnarray}
that is
\begin{eqnarray}
L_{0,\tau_j(\eps^2\mathbf{d}_j)}{\tt h}_j=\mathcal{F}^1_\eps({\tt h}_j,\mathbf{d}_j),
\label{eq_bifo_end_pr}
\end{eqnarray}
with 
\begin{equation}
\label{def_F1}
\begin{aligned}
\mathcal{F}^1_\eps({\tt h}_j,\mathbf{d}_j):=&\frac{(\tau_j(\eps^2\mathbf{d}_j)e^{\sigma_{\tau_j(\eps^2\mathbf{d}_j)}})^2}{c_\star} \int_\R (\mathcal{H}^2({\tt h}_j,\mathbf{d}_j)
+\eps^{-2}\mathcal{H}^4_\eps({\tt h}_j,\varphi_j({\tt h}_j,\mathbf{d}_j),\mathbf{d}_j))(s,t)v'_\star(t)dt\\
&-L_{0,\tau_j(\eps^2\mathbf{d}_j)}{\tt h}_j(s)\int_\R(U'(t)-v'_\star(t)) v'_\star(t)dt.
\end{aligned}
\end{equation}
and
$$c_\star:=\int_\R (v'_\star)^2 dt>0.$$
It turns out that $\mathcal{F}^1_\eps$ satisfies
\begin{eqnarray}
\|\mathcal{F}^1_\eps({\tt h}_j,\mathbf{d}_j)\|_{C^{0,\alpha}(\R)}\le C,\label{est_F}
\end{eqnarray}
for some constant $C>0$, for any $\mathbf{d}_j\in B_M$, for any ${\tt h}_j\in D^{2,\alpha}_{\tau_j(\eps^2\mathbf{d}_j)}(\R)$. Moreover, 
\begin{eqnarray}
\|\mathcal{F}^1_\eps({\tt h}^1_j,\mathbf{d}^1_j)-\mathcal{F}^1_\eps({\tt h}^2_j,\mathbf{d}^2_j)\|_{C^{0,\alpha}(\R)}\le C(\eps^2|\mathbf{d}^1_j-\mathbf{d}^2_j|+\eps\|{\tt h}^1_j-{\tt h}^2_j\|_{C^{2,\alpha}(\R)}),\label{lip_F_d_h}
\end{eqnarray}
for some constant $C>0$, for any $\mathbf{d}^i_j\in B_M$, $\text{h}^i_j\in D^{2,\alpha}_{\tau_j(\mathbf{d}_j)}(\R)$, with $\|{\tt h}^i_j\|_{C^{2,\alpha}(\R)}<K$, $1\le i\le 2$, $1\le j\le k$. Once again, the constant $C$ does not depend on the parameters, and $M,K>0$ are arbitrary. This follows from the size of $\varphi_j({\tt h}_j,\mathbf{d}_j)$, given by Proposition \ref{prop_aux_end}, the properties of the $\mathcal{H}^2_\eps$ and $\mathcal{H}^4_\eps$ (see (\ref{est_H2}), (\ref{est_H4}), (\ref{lip_H4_varphi}) and Lemma \ref{lemma_lip_H2}, which is satisfied also by $\mathcal{H}^4_\eps$).
\begin{proposition}
Let $M>0$, $\mathbf{d}_j\in B_M$. Then there exists $K>0$ independent of $M$ such that, for $\eps$ small enough, there exists a solution ${\tt h}_j:={\tt h}_j(\mathbf{d}_j)\in D^{2,\alpha}_{\tau_j(\eps^2\mathbf{d}_j)}(\R)$ to equation (\ref{eq_bifo_end_pr}) fulfilling 
\[
\|{\tt h}_j\|_{C^{2,\alpha}(\R)}<K.
\]
Moreover,
\begin{eqnarray}
\|{\tt h}_j(\mathbf{d}^1_j)-{\tt h}_j(\mathbf{d}^2_j)\|_{C^{2,\alpha}(\R)}\le C\eps^2|\mathbf{d}^1_j-\mathbf{d}^2_j|,
\end{eqnarray}
for some constant $C>0$, for any $\mathbf{d}^i_j\in B_M$, $i=1,2$.\label{prop_bifo_end}
\end{proposition}
This result will be proved in section \ref{section_bifo_end}. It will be clear from the proof that the constant $K$ is chosen exactly in this step. The aim is to have an error of order $\eps^2$ in a suitable decaying norm. 
\begin{lemma}
Let $M>0$, $a\in(0,\bar{a})$, $\gamma\in(0,\sqrt{2})$, $\alpha\in(0,1)$ and $\mathbf{d}\in (B_M)^k$. Then the error is given by
\begin{eqnarray}\notag
\chi_{3,\eps}(\eps\Delta v(\mathbf{d})+\eps^{-1}f(v(\mathbf{d}))-\ell_\eps)=\chi_{3,\eps}\big(-\eps^2\mathcal{J}_{\Sigma(\eps^2\mathbf{d})}(h_0\circ \beta(\eps^2\mathbf{d}))-(H_{\Sigma(\eps^2\mathbf{d})}-H_\Sigma)\big)U'\\\notag
+\chi_{3,\eps}\mathcal{G}^5_\eps(h_0,\mathbf{d}),
\end{eqnarray}
with $\mathcal{G}^5_\eps$ such that
\begin{equation}
\label{est_G5}
\begin{aligned}
\|\chi_{3,\eps}\mathcal{G}^5_\eps(h_0,\mathbf{d})\circ(\beta(\eps^2\mathbf{d})^{-1}\times Id_\R)\|_{\mathcal{E}^{2,\alpha}_{a,\gamma}(\Sigma\times\R)}\le C\eps^2,\\
\|\chi_{3,\eps}\mathcal{G}^5_\eps(h^1_0,\mathbf{d}^1)\circ(\beta(\eps^2\mathbf{d}^1)^{-1}\times Id_\R)-\chi_{3,\eps}\mathcal{G}^5_\eps(h^2_0,\mathbf{d}^2)\circ(\beta(\eps^2\mathbf{d}^2)^{-1}\times Id_\R)\|_{\mathcal{E}^{2,\alpha}_{a,\gamma}(\Sigma\times\R)}\\%\label{lip_G5_d} 
\le C(\eps^4|\mathbf{d}^1_j-\mathbf{d}^2_j|+\eps^3\|h^1_0-h^2_0\|_{\mathcal{C}^{2,\alpha}_a(\Sigma)}),\\
%\|\chi_{3,\eps}\mathcal{G}^5_\eps(h^1_0,\mathbf{d})\circ(\beta(\eps^2\mathbf{d})^{-1}\times Id_\R)-\chi_{3,\eps}\mathcal{G}^5_\eps(h^2_0,\mathbf{d})\circ(\beta(\eps^2\mathbf{d})^{-1}\times Id_\R)\|_{\mathcal{E}^{2,\alpha}_{a,\gamma}(\Sigma\times\R)}\label{lip_G5_h0}\\\notag
%\le C\eps^3\|h^1_0-h^2_0\|_{\mathcal{C}^{2,\alpha}_a(\Sigma\times\R)},
\end{aligned}
\end{equation}
for some constant $C>0$, for any $\mathbf{d}^i\in (B_M)^k$, for any $h^i_0\in \mathcal{C}^{0,\alpha}_a(\Sigma)$, $\|h^i_0\|_{\mathcal{C}^{0,\alpha}_a(\Sigma)}<M$.
\label{lemma_dec_error}
\end{lemma}
\begin{proof}
Going back to (\ref{error_correct}) and using the corrections $\varphi_j$ and $\text{h}_j$ determined in Proposition \ref{prop_ls_end}, it is possible to see that
\begin{eqnarray}\notag
\chi_{3,\eps}(\eps\Delta v(\mathbf{d})+\eps^{-1}f(v(\mathbf{d}))-\ell_\eps)=\chi_{3,\eps}\big(-\eps^2\mathcal{J}_{\Sigma(\eps^2\mathbf{d})}(h_0\circ \beta(\eps^2\mathbf{d}))-(H_{\Sigma(\eps^2\mathbf{d})}-H_\Sigma)\big)U'\\\notag
+\chi_{3,\eps}\mathcal{G}^3_\eps((h_0,{\tt h}_1(\mathbf{d}_1),\dots,{\tt h}_k(\mathbf{d}_k)),\mathbf{d},
(\varphi_1(\mathbf{d}_1),\dots,\varphi_k(\mathbf{d}_k))).
\end{eqnarray}
%The term involving the mean curvatures, that is $-(H_{\Sigma(\mathbf{d})}-H_\Sigma)U'$, can be controlled by Lemma \ref{lemma_H_Sigmad}. 
The only remaining unknowns here are $h_0$ and $\mathbf{d}$. The size of 
$$\mathcal{G}^5_\eps(h_0,\mathbf{d}):=\mathcal{G}^3_\eps((h_0,{\tt h}_1(\mathbf{d}_1),\dots,{\tt h}_k(\mathbf{d}_k)),\mathbf{d},
(\varphi_1(\mathbf{d}_1),\dots,\varphi_k(\mathbf{d}_k))),$$ 
given by (\ref{est_G5}), follows from (\ref{est_G2}), (\ref{def_G4}), (\ref{est_G4}) and Proposition \ref{prop_ls_end}, while the Lipschitz dependence on the data follows from Lemma \ref{lemma_lip_G2}, (\ref{lip_G4_varphi}), (\ref{lip_varphi_d}) and (\ref{lip_h_d}).
\end{proof}

\begin{remark}
By Lemma \ref{lemma_dec_error} and Lemma \ref{lemma_H_Sigmad}, we have produced a new error, which, for any $a\in(0,\bar{a})$, $\gamma\in(0,\sqrt{2})$ and $\eps$ small enough, is decaying both along $\Sigma(\eps^2\mathbf{d})$, at rate $a$, and in $t$, at rate $\gamma$, and satisfies
$$\|\chi_{3,\eps}(\eps\Delta v(\mathbf{d})+\eps^{-1}f(v(\mathbf{d}))-\ell_\eps)\|_{\mathcal{E}^{0,\alpha}_{a,\gamma}(\Sigma(\eps^2\mathbf{d})\times\R)}\le C\eps^2,$$
for some constant $C>0$ independent of the parameters. The expression of the error for $t$ large does not really matter, since, as we will see in the sequel, the approximate solution will be chosen to be constant far from its nodal set, in such a way that the error vanishes outside a tubular neighbourhood of $\Sigma(\eps^2\mathbf{d})$. \label{remark_decaying_error}
\end{remark}

\subsection{The auxiliary equation along an end}\label{section_aux_end}

In order to solve equation (\ref{eq_aux_end}), we first study the inverse of the operator $\eps(\tau e^{\sigma_\tau})^{-2}\partial^2_s+\eps^{-1}\partial^2_t+\eps^{-1}f'(v_\star)$ on the space of functions that are $L^2(\R)$-orthogonal to the kernel, spanned by $v'_\star$, then we apply a fixed point argument.

\subsubsection{The linear problem}\label{section_aux_end_linear}
We treat the linear problem
\begin{eqnarray}
\eps(\tau e^{\sigma_\tau})^{-2}\partial^2_s\varphi+\eps^{-1}\partial^2_t\varphi+\eps^{-1}f'(v_\star)\varphi=g &\text{in $\R^2$.}\label{aux_lin_end}
\end{eqnarray}
We assume that the right-hand side satisfies the orthogonality condition
\begin{eqnarray}
\int_\R g(s,t)v'_\star(t)dt=0, \qquad\forall \,s\in\R, \label{cond_ort_end}
\end{eqnarray}
and we look for solutions fulfilling the same. In order to simplify the proof, it is convenient to rescale the variables, thus we introduce the coordinates
$$
\begin{cases}
{\tt t}=t\\
{\tt s}=\eps^{-1}s.
\end{cases}$$ 
and we study the equation
\begin{eqnarray}
(\tau e^{\sigma_\tau})^{-2}\partial^2_{{\tt s}}\varphi+\partial^2_{{\tt t}}\varphi+f'(v_\star)\varphi=g &\text{in $\R^2$.}\label{aux_lin_scaled_end}
\end{eqnarray}
For this scaled problem, we prove existence and uniqueness under the orthogonality condition
\begin{eqnarray}
\int_{\R}\varphi({\tt s},{\tt t})v'_\star({\tt s})d{\tt s}=0, \qquad\forall \,{\tt s}\in\R \label{cond_ort_scaled_end}.
\end{eqnarray}
In order to estimate the inverse, we set, for $\varphi\in\D^{2,\alpha}_{\tau,\gamma}(\R^2)$, $\varphi_\eps({\tt s},{\tt t}):=\varphi(\eps{\tt s},{\tt t})$ and we introduce the spaces
\[\D^{n,\alpha}_{\eps,\tau,\gamma}(\R^2):=\{\varphi_\eps:\,\varphi\in \D^{n,\alpha}_{\tau,\gamma}(\R^2)\}.\]
On these spaces, we introduce the norms
\begin{eqnarray}
\|\varphi\|_{\D^{n,\alpha}_\gamma(\R^2)}:=\sum_{0\le m+k\le n}\|\partial_{{\tt t}}^k \partial^m_{{\tt s}}\varphi\|_{\D^{0,\alpha}_\gamma(\R^2)},\label{norm_scaled_aux}
\end{eqnarray}
where
$$\|\varphi\|_{\D^{0,\alpha}_\gamma(\R^2)}:=\|\varphi \cosh^\gamma({\tt t})\|_{C^{0,\alpha}(\R)}.$$
We note that
\begin{equation}
\label{compare_norms_per}
c\eps^\alpha\|\varphi\|_{E^{n,\alpha}_\gamma(\R^2)}\le \|\varphi_\eps\|_{\D^{n,\alpha}_\gamma(\R^2)}\le C\|\varphi\|_{E^{n,\alpha}_\gamma(\R^2)}, \qquad 0<c<C,\, n\ge 0.
\end{equation}
\begin{lemma}
Let $0<\underline{\tau}<\overline{\tau}<1$, $\tau\in(\underline{\tau},\overline{\tau})$, $\gamma\in(0,\sqrt{2})$, $\alpha\in (0,1)$ and $g\in \D^{0,\alpha}_{\eps,\tau,\gamma}(\R^2)$. Assume that $g$ satisfies (\ref{cond_ort_scaled_end}). Then, for $\eps$ small enough, there exists a unique solution $\varphi_\tau\in \D^{2,\alpha}_{\eps,\tau,\gamma}(\R^2)$ to (\ref{aux_lin_scaled_end}) satisfying (\ref{cond_ort_scaled_end}). Moreover,
\begin{eqnarray}
\|\varphi_\tau\|_{\D^{2,\alpha}_\gamma(\R^2)}\le C\|g\|_{\D^{0,\alpha}_\gamma(\R^2)},\label{est_wnorm_aux_end}
\end{eqnarray}
for some constant $C>0$.
%Moreover, for $i=1,2$, let $\tau_i\in(\underline{\tau},\overline{\tau})$, $g_i\D^{0,\alpha}_{\eps,\tau_i,\gamma}(\R^2)$ be fixed functions and let $\varphi_{\tau_i}:=\varphi_{\tau_i}(g)$. Then
%\begin{eqnarray}
%\|\varphi_{\tau_1}-\varphi_{\tau_2}\|_{\D^{2,\alpha}_\gamma(\R^2)}\le C|\tau_1-\tau_2|\|g\|_{\D^{0,\alpha}_\gamma(\R^2)},\label{lip_varphi_gamma}
%\end{eqnarray}
%for some constant $C>0$ independent of $\eps$ and $\tau$.
\label{lemma_ls_scaled}
\end{lemma}
\begin{proof}
\textit{step (i): Existence and uniqueness in a stripe.}\\

First we solve equation (\ref{aux_lin_scaled_end}) in the stripe $(0,\eps^{-1}s_\tau)\times\R$ with zero Neumann boundary conditions. We use that, thanks to the spectral properties of the operator $-\partial^2_{{\tt t}}-f'(v_\star({\tt t}))$, the quadratic form  
\begin{eqnarray}\notag
\mathcal{Q}_\tau(\varphi):=\int_{(0,\eps^{-1}s_\tau)\times\R} (|\nabla\varphi|^2-f'(v_\star)\varphi^2) d\text{s}d\text{t}
\end{eqnarray}
is positive definite on the space
$$\mathcal{Z}_{\eps,\tau}:=\{\varphi\in H^1((0,\eps^{-1}s_\tau)\times\R):\int_{\R}\varphi({\tt s},{\tt t})v'_\star(\text{s})d{\tt s}=0, \text{ for a. e. ${\tt s}\in\R$}\},$$
that is it satisfies
$$\mathcal{Q}_\tau(\varphi)\ge C\|\varphi\|^2_{H^1((0,\eps^{-1}s_\tau)\times\R)}, \qquad\forall\, \varphi\in \mathcal{Z}_{\eps,\tau},$$
for some constant $C>0$ independent of $\tau$. Therefore, by the Riesz representation theorem, there exists a unique $\bar{\varphi}_\tau\in \mathcal{Z}_{\eps,\tau}$ such that 
\begin{eqnarray}
\int_{(0,\eps^{-1}s_\tau)\times\R}(\nabla\bar{\varphi}_\tau,\nabla\psi)-f'(v_\star)\bar{\varphi}_\tau\psi d{\tt s}d{\tt t}=\int_{(0,\eps^{-1}s_\tau)\times\R} g\psi d{\tt s}d{\tt t}, \qquad\forall\psi\in \mathcal{Z}_{\eps,\tau}. \label{weak_sol_aux_lin_end}
\end{eqnarray}
In order to prove that $\bar{\varphi}_\tau$ is actually a weak solution, we have to check that (\ref{weak_sol_aux_lin_end}) is also verified for any $\psi$ in the orthogonal complement
$$\mathcal{Z}_{\eps,\tau}^\bot=\{a({\tt s})v'_\star({\tt t}):\,a\in H^1((0,\eps^{-1}s_\tau))\},$$
which is the case since $\bar{\varphi}_\tau\in \mathcal{Z}_{\eps,\tau}$ and $v'_\star$ is in the kernel of $-\partial^2_{{\tt t}}-f'(v_\star({\tt t}))$. If $g\in \D^{0,\alpha}_{\eps,\tau,\gamma}(\R^2)$, then, by the regularity theory (see \cite{GT}), we see that $\bar{\varphi}_\tau\in C^{2,\alpha}_{loc}((0,\eps^{-1}s_\tau)\times\R)$ is a classical solution, and the boundary conditions are satisfied in the classical sense. Therefore, by Theorem $6.26$ of \cite{GT}, we see that $\bar{\varphi}_\tau\in C^2([0,\eps^{-1}s_\tau]\times\R)$. \\

\textit{Step (ii): Symmetry and extension to an entire solution.}\\

Due to uniqueness, $\bar{\varphi}_\tau$ inherits the symmetries of $g$, that is it satisfies 
$$\bar{\varphi}_\tau({\tt s},{\tt t})=\bar{\varphi}_\tau(\eps^{-1}s_{\tau}-{\tt s},{\tt t}),\qquad \forall \, {\tt s}\in (0,\eps^{-1}s_\tau),$$
thus it can be extended to an entire solution $\varphi_\tau\in C^{2,\alpha}_{loc}(\R^2)$, which fulfils (\ref{symm_delaunay}).\\

\textit{Step (iii): Boundedness. $\varphi_\tau\in L^\infty(\R^2)$ and}
\begin{eqnarray}
\|\varphi_\tau\|_{L^\infty(\R^2)}\le C(\gamma)\|g\cosh^\gamma({\tt t})\|_{L^\infty(\R^2)}.\label{bound_Linfty}
\end{eqnarray}
Testing the equation with $\varphi_\tau$ gives
$$\|\varphi_\tau\|_{H^1((0,\eps^{-1}s_{\tau})\times\R)}\le C\|g_\tau\|_{L^2((0,\eps^{-1}s_{\tau})\times\R)}.$$
By the Sobolev embeddings, local elliptic estimates (see \cite{GT}, Theorem $8.8$) and periodicity, for any $({\tt s},{\tt t})\in\R^2$, we have
\begin{equation}
\begin{aligned}
&\|\varphi_\tau\|_{L^\infty(B_1(({\tt s},{\tt t})))}\le C\|\varphi_\tau\|_{W^{2,2}(B_1(({\tt s},{\tt t})))}\le C(\|\varphi_\tau\|_{L^2(B_2(({\tt s},{\tt t})))}+\|g\|_{L^2(B_2(({\tt s},{\tt t})))})\\
&\le C(\|\varphi_\tau\|_{L^2((0,\eps^{-1}s_{\tau})\times\R)}+\|g\|_{L^2((0,\eps^{-1}s_{\tau})\times\R)})\le
C\|g\|_{L^2((0,\eps^{-1}s_{\tau})\times\R)}\le C(\gamma)\|g\cosh^\gamma({\tt t})\|_{L^\infty(\R^2)},
\end{aligned}
\end{equation}
where the constants may depend on $\gamma$, but not on $({\tt s},{\tt t})$. As a consequence, $\varphi_\tau\in L^\infty(\R^2)$ and (\ref{bound_Linfty}) holds.

\textit{Step (iv): decay. $\varphi_\tau \cosh^\gamma({\tt t})\in L^\infty(\R^2)$ and there exists $C>0$ such that}
\begin{eqnarray}
\|\varphi_\tau\cosh^\gamma({\tt t})\|_{L^\infty(\R^2)}\le C\|g_\tau\cosh^\gamma({\tt t})\|_{L^\infty(\R^2)}.\label{bound_Linfty_gamma}
\end{eqnarray}
This can be proved by the maximum principle, using the function $\cosh^{-\gamma}(\tt{t})$ as a barrier. In fact, setting $\Omega_{{\tt t}_0}:=\{({\tt s},{\tt t})\in\R^2:{\tt t}>{\tt t}_0\}$, we have
$$\varphi_\tau({\tt s},{\tt t}_0)\le \|\varphi_\tau\|_{L^\infty(\R^2)}\le \lambda \cosh^{-\gamma}({\tt t}_0)$$
on $\partial\Omega_{{\tt t}_0}$, provided $\lambda\ge \|\varphi_\tau\|_{L^\infty(\R^2)}e^{\gamma{\tt t}_0}$. Moreover, in $\Omega_{{\tt t}_0}$ we have
\begin{equation}
\begin{aligned}
&-((\tau e^{\sigma_\tau})^{-2}\partial^2_{{\tt s}}+\partial^2_{{\tt t}}+f'(v_\star))(\varphi_\tau-\lambda \cosh^{-\gamma}({\tt t}))=g_\tau-\lambda(-\gamma^2-f'(v_\star))\cosh^{-\gamma}({\tt t})\\
&\le (\|g\cosh^\gamma({\tt t})\|_{L^\infty(\R^2)}-\lambda(1-\frac{\gamma^2}{2}))\cosh^{-\gamma}({\tt t})\le 0
\end{aligned}
\end{equation}
provided 
$$\lambda\ge \frac{2\|g_\tau\cosh^\gamma({\tt t})\|_{L^\infty(\R^2)}}{2-\gamma^2}.$$
%and $\text{t}_0$ is large enough. 
As a consequence, taking 
$$\lambda:=\frac{2\|g\cosh^\gamma({\tt t})\|_{L^\infty(\R^2)}}{2-\gamma^2}+\|\varphi_\tau\|_{L^\infty(\R^2)}\cosh^{\gamma}({\tt t}_0)$$
and applying the maximum principle in the unbounded domain $\Omega_{{\tt t}_0}$, we have
\begin{eqnarray}\notag
\varphi_\tau({\tt s},{\tt t}) \cosh^\gamma({\tt t})\le \frac{2\|g\cosh^\gamma({\tt t})\|_{L^\infty(\R^2)}}{2-\gamma^2}+2\|\varphi_\tau\|_{L^\infty(\R^2)}\cosh^{\gamma}({\tt t}_0)\le C(\gamma)\|g\cosh^\gamma({\tt t})\|_{L^\infty(\R^2)},
\end{eqnarray}
for any $({\tt s},{\tt t})\in\R^2_+:=\{({\tt s},{\tt t})\in\R^2:\,{\tt t}>0\}$. For the maximum principle in possibly unbounded domains, we refer to Lemma $2.1$ of \cite{BCN}. Applying this last inequality to $-\varphi({\tt s},{\tt t})$ and to $\varphi({\tt s},-{\tt t})$, we have the statement.\\

\textit{Step (v): Estimate of the weighted norm of $\varphi_\tau$.}\\

The estimate of the solution given by (\ref{est_wnorm_aux_end}) follows from step \textit{(iv)} and local elliptic estimates (see \cite{GT}, Corollary $6.3$). We note that the constants appearing in the elliptic estimates do not depend on $\eps$, thanks to the scaling.\\

%\textit{Step (vi): Lipschitz dependence on $\tau$. The proof of (\ref{lip_varphi_gamma}).}\\

%This estimate follows from uniqueness and the fact that

\end{proof}

\begin{lemma}
Let $0<\underline{\tau}<\overline{\tau}<1$, $\tau\in(\underline{\tau},\overline{\tau})$, $\gamma\in(0,\sqrt{2})$, $\alpha\in (0,1)$ and $g\in \D^{0,\alpha}_{\tau,\gamma}(\R^2)$. Assume that $g$ satisfies (\ref{cond_ort_end}). Then, for $\eps$ small enough, there exists a unique solution $\varphi_\tau:=\Psi^1_\tau(g)\in \D^{2,\alpha}_{\tau,\gamma}(\R^2)$ to (\ref{aux_lin_end}) satisfying (\ref{cond_ort_end}) and
\begin{eqnarray}
\|\Psi^1_\tau(g)\|_{E^{2,\alpha}_\gamma(\R^2)}\le C\eps^{1-\alpha}\|g\|_{E^{0,\alpha}_\gamma(\R^2)}.\label{est_wnorm_aux_end_s}
\end{eqnarray}
for some constant $C>0$. Moreover,
\begin{equation}
\label{lip_varphi_gamma_s}
\begin{aligned}
&\|\Psi^1_{\tau_1}(g_1)-\Psi^1_{\tau_2}(g_2)\|_{E^{2,\alpha}_\gamma(\R^2)}\\
&\le C(\eps^{1-\alpha}\|g_1-g_2\|_{E^{0,\alpha}_\gamma(\R^2)}
+\eps^{1-2\alpha}|\tau_1-\tau_2|(\|g_1\|_{E^{0,\alpha}_\gamma(\R^2)}+\|g_2\|_{E^{0,\alpha}_\gamma(\R^2)})),
\end{aligned}
\end{equation}
for any $\tau_i\in(\underline{\tau},\overline{\tau})$, $g_i\in\D^{0,\alpha}_{\tau_i,\gamma}(\R^2)$, $i=1,2$.\label{lemma_aux_lin}
\end{lemma}
\begin{proof}
This can be proved applying Lemma \ref{lemma_ls_scaled} and a scaling argument (see (\ref{compare_norms_per})). Indeed, it is possible to find a solution $\varphi_{\tau,\eps}\in \D^{2,\alpha}_{\eps,\tau,\gamma}(\R^2)\cap \mathcal{Z}_{\eps,\tau}$ to 
\begin{eqnarray}\notag
(\tau e^{\sigma_\tau})^{-2}\partial^2_{{\tt s}}\varphi_{\tau,\eps}+\partial^2_{{\tt t}}\varphi_{\tau,\eps}+f'(v_\star)\varphi_{\tau,\eps}=\eps g_\eps, \qquad g_\eps({\tt s},{\tt t}):=g(\eps{\tt s},{\tt t})=g(s,t).
\end{eqnarray}
Then $\varphi_\tau(s,t):=\varphi_{\tau,\eps}(\eps^{-1}s,t)=\varphi_{\tau,\eps}({\tt s},{\tt t})$ is the required solution. As regards the Lipschitz dependence on the data, it is enough to observe that
$$(\eps(\tau_1 e^{\sigma_{\tau_1}})^{-2}\partial^2_{{\tt s}}+\eps^{-1}\partial^2_{{\tt t}}+\eps^{-1}f'(v_\star))(\varphi_{\tau_1}-\varphi_{\tau_2})=g_1-g_2
-\eps((\tau_1 e^{\sigma_{\tau_1}})^{-2}-(\tau_2 e^{\sigma_{\tau_2}})^{-2})\partial^2_{{\tt s}}\varphi_{\tau_2},$$
and to use the size in $\eps$ of $\varphi_\tau$ and (\ref{est_wnorm_aux_end_s}).
\end{proof}

\subsubsection{The proof of Proposition \ref{prop_aux_end}: a fixed point argument}

Using the right inverse constructed in section \ref{section_aux_end_linear} (see Lemma \ref{lemma_aux_lin}), it is possible to formulate equation (\ref{eq_aux_end}) as a fixed point problem of the form
$$\varphi_j=\Psi^1_{\tau_j(\eps^2\mathbf{d}_j)}\Pi_{\mathcal{Z}_{\tau_j(\eps^2\mathbf{d}_j)}}\mathcal{H}^5_\eps({\tt h}_j,\varphi_j,\mathbf{d}_j)$$
in the ball
$$B_{\Lambda_1}:=\{\varphi_j\in \D^{2,\alpha}_{\tau_j(\eps^2\mathbf{d}_j),\gamma}(\R^2)\cap\mathcal{Z}_{\tau_j(\eps^2\mathbf{d}_j)}:
\|\varphi_j\|_{E^{2,\alpha}_\gamma(\R^2)}<\Lambda_1\eps^{3-\alpha}\}.$$
The right hand side defines a contraction on $B_{\Lambda_1}$ provided $\alpha\in (0,1/2)$ and $\Lambda_1>0$ is large enough, due to the definition of $\mathcal{H}^5_\eps$ (\ref{def_H5}), the estimates (\ref{est_H2}), (\ref{est_H4}), Lemma \ref{lemma_lip_H2} and (\ref{lip_H4_varphi}). The Lipschitz dependence on the data ${\tt h}_j$ and $\mathbf{d}_j$ follows in a similar way.

\subsection{The bifurcation equation along an end}\label{section_bifo_end}

The aim of this subsection is to solve equation (\ref{eq_bifo_end}). This can be done applying once again the contraction mapping theorem, after constructing a right inverse of the linear operator $\partial^2_s+\tau^2 \cosh(2\sigma_\tau)$.

\subsubsection{The linear problem}

We fix $0<\underline{\tau}<\overline{\tau}<1$, $\tau\in(\underline{\tau},\overline{\tau})$ and we study the ODE 
\begin{eqnarray}
L_{0,\tau}{\tt h}=g.\label{leading_ODE}
\end{eqnarray}
It is natural to assume that $g$ is periodic of period $s_\tau$ and even with respect to $s_\tau/2$, namely $g\in D^{0,\alpha}_\tau(\R)$. We recall that we have set
$$L_{0,\tau}:=\partial^2_s+\tau^2\cosh(2\sigma_\tau).$$

\begin{lemma}{[Injectivity]}
Let $0<\underline{\tau}<\overline{\tau}<1$, $\tau\in(0,1)$ and $\alpha\in (0,1)$. If ${\tt h}\in D^{2,\alpha}_\tau(\R)$ is a solution to the equation $L_{0,\tau}{\tt h}=0$, then ${\tt h}=0$.
\label{lemma_injet_tau}
\end{lemma}
\begin{proof}
Thanks to periodicity, we can reduce ourselves to study (\ref{leading_ODE}) on the interval $(0,s_\tau)$. We prove that any solution ${\tt h}$ in the space
$$D^{2,\alpha}_{\tau,0}([0,s_\tau]):=\{{\tt h}|_{[0,s_\tau]}:\,{\tt h}\in D^{2,\alpha}_\tau(\R)\}=\{{\tt h}\in D^{2,\alpha}([0,s_\tau]):\,{\tt h}(s)=\text{h}(s_\tau-s),\,{\tt h}'(0)={\tt h}'(s_\tau)=0\}.$$
is trivial. For this purpose, we note that $L_{0,\tau}$ has two Jacobi fields given by
$$\Phi^{0,+}_\tau:=\partial_s\sigma_\tau,\qquad\Phi^{0,-}_\tau:=\sqrt{1-\tau^2}\bigg(\frac{1}{\tau}\partial_s\sigma\partial_\tau k-e^{\sigma_\tau}\cosh\sigma_\tau(1+\tau\partial_\tau \sigma_\tau)\bigg).$$
Using the equation satisfied by $\sigma_\tau$, we see that
$$
    \begin{pmatrix}
      \partial_s\Phi^{0,+}_\tau(0) \\ \partial_s\Phi^{0,+}_\tau(s_\tau) 
    \end{pmatrix}
    = \sqrt{1-\tau^2} \begin{pmatrix}
      1 \\ 1 
    \end{pmatrix},\qquad \qquad 
    \begin{pmatrix}
     \partial_s\Phi^{0,-}_\tau(0) \\ \partial_s\Phi^{0,-}_\tau(s_\tau)
    \end{pmatrix}
    = \frac{1-\tau^2}{\tau} \begin{pmatrix}
      \partial_\tau k_\tau(0) \\  \partial_\tau k_\tau(s_\tau)
    \end{pmatrix}, 
    $$
hence, due to the linear growth in $s$ of $\partial_\tau k_\tau$, $\partial_\tau k_\tau(0)\ne \partial_\tau k_\tau(s_\tau)$, thus $\Phi^{0,\pm}_\tau$ are linearly independent. Being $L_{0,\tau}$ a second order operator, any solution to the homogeneous equation is a linear combination of these functions. By the above relation, we see that a function ${\tt h}\in D^{2,\alpha}_{\tau,0}([0,s_\tau])$ which is a linear combination of $\Phi^{0,\pm}_\tau$ is trivial, due to the boundary conditions. 
\end{proof}

\begin{lemma}{[A priori estimate]}
Let $0<\underline{\tau}<\overline{\tau}<1$, $\tau\in(\underline{\tau},\overline{\tau})$ and $\alpha\in (0,1)$. Then there exists a constant $C>0$ independent of $\tau$ such that, if ${\tt h}_\tau\in D^{2,\alpha}_\tau(\R)$ and $g\in D^{0,\alpha}_\tau(\R)$ are such that $L_{0,\tau}{\tt h}_\tau=g$, then
$$\|{\tt h}_\tau\|_{C^{2,\alpha}(\R)}\le C\|g\|_{C^{0,\alpha}(\R)}$$
\label{apriori_tau}
\end{lemma}
\begin{proof}
By the elliptic estimates, it is enough to bound the $L^\infty(\R)$-norm of ${\tt h}_{\tau}$. We assume by contradiction that there exist sequences $\tau_k$ in $(\underline{\tau},\overline{\tau})$, ${\tt h}_k\in D^{2,\alpha}_{\tau_k}(\R)$ and $g_k\in D^{2,\alpha}_{\tau_k}(\R)$ such that
$$L_{0,\tau_k}{\tt h}_k=g_k, \qquad \|{\tt h}_k\|_{L^\infty(\R)}=1, \qquad \|g_k\|_{C^{0,\alpha}(\R)}=o(1).$$
Then, by the elliptic estimates, ${\tt h}_k$ turns out to be bounded in $C^{2,\alpha}(\R)$, thus it converges, up to a subsequence, to some function ${\tt h}_\infty$ in $C^2_{loc}(\R)$. In order to compute the equation satisfied by ${\tt h}_\infty$, we observe that, up to a subsequence, $\tau_k\to\tau_\infty\in [\underline{\tau},\overline{\tau}]$, thus, by the continuous dependence on the data, $\sigma_{\tau_k}\to \sigma_{\tau_\infty}$ point wise and $s_{\tau_k}\to s_{\tau_\infty}$, which yields that ${\tt h}_\infty\in D^{2,\alpha}_{\tau_\infty}(\R)$ is a solution to $L_{0,\tau_\infty}{\tt h}_\infty=0$, hence, by Lemma \ref{lemma_injet_tau}, ${\tt h}_\infty\equiv 0$. It is exactly here that we need the bound $\tau\in[\underline{\tau},\overline{\tau}]$, in order to exclude that $\tau_\infty\in\{0,1\}$. However, there exist $\bar{s}_k\in [0,s_{\tau_k}]$ such that ${\tt h}_{\tau_k}(\bar{s}_k)>1/2$. Since $\bar{s}_k$ is bounded too, then, up to a subsequence, $\bar{s}_k\to \bar{s}_\infty\in[0,s_{\tau_\infty}]$ and ${\tt h}_\infty(\bar{s}_\infty)\ge 1/2$, a contradiction.
\end{proof}

\begin{lemma}
Let $0<\underline{\tau}<\overline{\tau}<1$, $\tau\in(\underline{\tau},\overline{\tau})$, $\alpha\in (0,1)$ and $g\in D^{0,\alpha}_\tau(\R)$. Then there exists a unique solution ${\tt h}_\tau:=\Psi^2_\tau(g)\in D^{0,\alpha}_\tau(\R)$ to the equation $L_{0,\tau}{\tt h}_\tau=g$ such that
\begin{equation}
\|\Psi^2_\tau(g)\|_{C^{2,\alpha}(\R)}\le C\|g\|_{C^{0,\alpha}(\R)},\label{est_Psi2}
\end{equation}
for some constant $C>0$. Moreover,
\begin{equation}
\|\Psi^2_{\tau_1}(g_1)-\Psi^2_{\tau_2}(g_2)\|_{C^{2,\alpha}(\R)}\le C\|g_1-g_2\|_{C^{0,\alpha}(\R)}+C|\tau_1-\tau_2|(\|g_1\|_{C^{0,\alpha}(\R)}+\|g_2\|_{C^{0,\alpha}(\R)})\label{lip_Phi2_tau},
\end{equation}
for any $\tau_i\in(\underline{\tau},\overline{\tau})$, $g_i\in D^{0,\alpha}_{\tau_i}(\R)$, $i=1,2$.
\end{lemma}
\begin{proof}
By Lemma \ref{lemma_injet_tau}, we know that $L_{0,\tau}$ is injective. Since this operator is self-adjoint on $D^{2,\alpha}_{\tau}(\R)$, in order to prove that it is also surjective, it is enough to prove that it is Fredholm. This property follows from the fact that the image is closed, due to estimate provided by Lemma \ref{apriori_tau}.
\end{proof}

\subsubsection{The proof of Proposition \ref{prop_bifo_end}: a fixed point argument}

In order to prove Proposition \ref{prop_ls_end}, we write equation (\ref{eq_bifo_end_pr}) as a fixed point problem, in the form
\begin{eqnarray}\notag
{\tt h}_j=\Psi^2_{\tau_j(\eps^2\mathbf{d}_j)}\mathcal{F}^1_\eps({\tt h}_j,\mathbf{d}_j).
\end{eqnarray}
For the definition of $\mathcal{F}^1_\eps$, see (\ref{def_F1}). It is possible to prove that the right-hand side defines a contraction on the ball
$$B_K:=\{{\tt h}\in D^{2,\alpha}_{\tau_j(\eps^2\mathbf{d}_j)}(\R):\|{\tt h}\|_{C^{2,\alpha}(\R)}<K\},$$
provided $K$ is large enough, due to (\ref{est_F}) and (\ref{lip_F_d_h}). The Lipschitz dependence on the datum $\mathbf{d}_j$ follows from (\ref{lip_F_d_h}). We note that in this step we determine the constant $K$ defined in (\ref{def_H}).

\section{The proof of Theorem \ref{main_th}}\label{section_proof}

\setcounter{equation}{0}

\subsection{A gluing procedure}\label{section_gluing}
In section (\ref{section_approx_solution}) we defined the approximate solution near the surface, with the aid of the Fermi coordinates. Here we define a global approximation, by interpolating between that approximation and the constant solutions $\pm 1+\sigma_\eps^\pm$ far from the surface. In order to do so, we introduce the pull back and the push forward of a function. For a function $u:\mathcal{N}_{9 \delta}\to \R$, the pull back is $u_\sharp:=u\circ \mathcal{Y}_{\eps,h(\mathbf{d}),\mathbf{d}}$, $\mathbf{d}\in (B_M)^k$, and, for a function $\phi:\Sigma(\eps^2\mathbf{d})\times(-9\delta/\eps,9\delta/\eps)\to\R$, the push forward is $\phi^\sharp:=\phi\circ \mathcal{Y}^{-1}_{\eps,h(\mathbf{d}),\mathbf{d}}$. In this notation, we set
\begin{eqnarray}
{\tt w}_\eps(\mathbf{d}):=\chi^\sharp_{5,\eps} v(\mathbf{d})^\sharp+ (1-\chi^\sharp_{5,\eps})\mathbb{H}_\eps,
\end{eqnarray}
where 
$$
\mathbb{H}_\eps(x):=
\begin{cases}
1+\sigma_\eps^+ &\text{if $x\in\Sigma(\eps^2\mathbf{d})^+$}\\
-1+\sigma_\eps^- &\text{if $x\in\Sigma(\eps^2\mathbf{d})^-$,}
\end{cases}
$$
$\Sigma(\eps^2\mathbf{d})^+$ is the interior and $\Sigma(\eps^2\mathbf{d})^-$ the exterior of $\Sigma(\eps^2\mathbf{d})$. Since we also need a small correction far from the surface, we need to introduce weighted spaces of functions that are globally defined in $\R^3$ and decaying both along the ends and in the orthogonal direction. For this purpose, we rescale the surface, that is we consider the dilation $\Sigma(\eps^2\mathbf{d})_\eps:=\eps^{-1}\Sigma(\eps^2\mathbf{d})$ and, for $\gamma>0$, we take the Green function $G_\gamma$ of $\Delta-\gamma$ and we set
$$\varphi_\gamma(x):=\int_{\Sigma(\eps^2\mathbf{d})_\eps}G_\gamma(x,y)d\sigma(y),$$
where $\sigma$ is the surface measure on $\Sigma(\eps^2\mathbf{d})_\eps$. This weight is exponentially decaying in the orthogonal direction to the dilated surface. In order to keep track of the behaviour of our functions along the ends, we introduce $2$ families of open cones $\{C_j\}_{1\le j\le k}$ and $\{C'_j\}_{1\le j\le k}$ such that, for $1\le j\le k$,\\
\begin{itemize}
%\item There exists $\bar{R}>0$ such $E_j\backslash B_{\bar{R}}\subset C_j\backslash B_{\bar{R}}$.
\item $ C_j$ and $C'_j$ have a common centre, that we denote by $\eps^{-1}\mathbf{b}_j$ and their axis coincides with the one the end $E_j$. We denote its generator by $\mathbf{c}_j$, with the convention that $\mathbf{c}_j\cdotp x>0$, for any $x\in E_j$, $1\le j\le k$. 
\item $C'_j\subset C_j$, for $1\le j\le k$.
\item $C'_i\cap C'_j=\emptyset$, for $1\le i\ne j\le k$.
\end{itemize}
%As a consequence, there exists $\bar{R}>0$ such that $E_j\backslash B_{\bar{R}}\subset C_j\backslash B_{\bar{R}}$, for $1\le j\le k$. 
We stress that this construction is possible thanks to the fact that there are no parallel ends in $\Sigma$ (see hypothesis (H)). For instance, we can put
\begin{equation}
\begin{aligned}
& C_j:=\{x\in\R^3: \mathbf{c}_j\cdotp(x-\eps^{-1}\mathbf{b}_j)>\theta|\Pi_j(x-\eps^{-1}\mathbf{b}_j)|\},\\
& C'_j:=\{x\in\R^3: \mathbf{c}_j\cdotp(x-\eps^{-1}\mathbf{b}_j)>2\theta|\Pi_j(x-\eps^{-1}\mathbf{b}_j)|\},\\
\end{aligned}
\end{equation}
where $\Pi_j$ is the projection onto the orthogonal complement to $\mathbf{c}_j$ in $\R^3$. If $\theta$ is large enough, then the aforementioned properties are satisfied. Then we take a family $\{\beta_j\}_{1\le j\le k}$ of smooth cutoff functions equal $1$ in $C'_j\backslash B_{\eps^{-1}(\bar{R}+1)}$ and equal $0$ outside $C_j\backslash B_{\eps^{-1}\bar{R}}$, with $\bar{R}>0$ so large that $\Sigma(\eps^2\mathbf{d})\backslash \cup_{1\le j\le k} Z_j((s_0+5,\infty)\times S^1)\subset B_{\bar{R}}$, $\beta_0:=1-\sum_{j=1}^k \beta_j$. In view of this construction, we set, for $\gamma>0$, $a>0$ and $\eps$ small enough,
$$\Gamma_\eps(x):=\varphi_\gamma^{-1}(x)\bigg(\beta_0(x)+\sum_{j=1}^k \beta_j(x) e^{\eps a k^{-1}_{\tau_j(\eps^2\mathbf{d}_j)}\big((x-\eps^{-1}\mathbf{b}_j)\cdotp\mathbf{c}_j\big)}\bigg).$$
This weight is suitable to deal with a dilated version of the problem. For our original equation, we need the weight $\Gamma(x):=\Gamma_\eps(x/\eps)$, thus we introduce the spaces $C^{n,\alpha}_{a,\gamma}(\R^3)$ as the spaces of functions $\psi\in C^{n,\alpha}_{loc}(\R^3)$ for which the norm
$$\|\psi\|_{C^{n,\alpha}_{a,\gamma}(\R^3)}:=\sum_{m=1}^n \big(\eps^m\|\nabla^m(\psi\Gamma)\|_{L^\infty(\R^3)}+\eps^{m+\alpha}[\nabla^m(\psi\Gamma)]_\alpha\big)$$
is finite, where
\[
[u]_\alpha:=\sup_{x\ne y}\frac{|u(x)-u(y)|}{|x-y|^\alpha}.
\]
\begin{remark}
We note that, thanks to the fact that none of two of ends of $\Sigma$ are parallel (see hypothesis (H)), $C^{n,\alpha}_{a,\gamma}(\R^3)\subset L^2(\R^3)$. This is crucial in our approach, since it allows to use coercivity to solve (\ref{eq_far}).
\label{rem_angles} 
\end{remark}
We look for a true solution to (\ref{cahn_hilliard}) of the form
$$u_\eps={\tt w}_\eps(\mathbf{d})+\chi^\sharp_{2,\eps} \phi^\sharp+\psi,$$
where $\phi(\textbf{y},t)=\phi_0(\beta(\eps^2\mathbf{d})(\textbf{y}),t)$, with $\phi_0\in \mathcal{E}^{0,\alpha}_{a,\gamma}(\Sigma\times\R)$, and $\psi\in C^{2,\alpha}_{a',\gamma'}(\R^3)$. Here we take $0<a'<a<\bar{a}$ and $0<\gamma'<\gamma<\sqrt{2}$. The aim now is to reduce equation (\ref{cahn_hilliard}) to a system of $2$ equations, whose unknowns are $\phi_0$, $h_0$, $\mathbf{d}$ and $\psi$. We recall that $h_0$ and $\mathbf{d}$ play a role in the definition of the approximate solutions and in the diffeomeorphism $\mathcal{Y}_{\eps,h(\mathbf{d}),\mathbf{d}}$. In order to do so, we introduce another smooth cutoff function $\tilde{\chi}:\R^3\to\R$ equal to $1$ in $\Sigma(\eps^2\mathbf{d})^+\backslash \mathcal{N}_\delta$ and equal to $0$ in $\Sigma(\eps^2\mathbf{d})^-\backslash \mathcal{N}_\delta$, and we set
$$q:=\tilde{\chi}f(1+\sigma_\eps^+)+(1-\tilde{\chi})f(-1+\sigma_\eps^-), \qquad L_0:=\eps\Delta+\eps^{-1}(f'({\tt w}_\eps(\mathbf{d}))(1-\chi^\sharp_{1,\eps})+q\chi^\sharp_{1,\eps}).$$
Using that $\chi^\sharp_{1,\eps}\chi^\sharp_{2,\eps}=\chi^\sharp_{1,\eps}$, (\ref{cahn_hilliard}) can be written as
\begin{multline*}
\chi^\sharp_{2,\eps}\big(N_\eps({\tt w}_\eps(\mathbf{d}))+L_{{\tt w}_\eps}\phi^\sharp+\eps^{-1}\chi^\sharp_{1,\eps}Q_{{\tt w}_\eps}(\phi^\sharp+\psi)
+\eps^{-1}\chi^\sharp_{1,\eps}(f'({\tt w}_\eps(\mathbf{d}))-q)\psi\big)\\
+L_0\psi+(1-\chi^\sharp_{2,\eps})N_\eps({\tt w}_\eps(\mathbf{d}))+(1-\chi^\sharp_{1,\eps})Q_{{\tt w}_\eps}(\chi^\sharp_{2,\eps}\phi^\sharp+\psi)
+\eps[\Delta,\chi^\sharp_{2,\eps}]\phi^\sharp=0,
\end{multline*}
where  $N_\eps(u):=\eps\Delta u+\eps^{-1}f(u)-\ell_\eps$. The latter equation is solved if we solve simulatenously
\begin{equation}
\label{eq_near}
\begin{aligned}
&\eps\Delta_{\Sigma(\eps^2\mathbf{d})}\phi+\eps^{-1}\partial^2_t\phi+\eps^{-1}f'(v_\star)\phi=-\chi_{3,\eps}
\big(N_\eps(v(\mathbf{d}))+\text{L}_{\eps,\mathbf{d}}\phi\\
&\qquad +\eps^{-1}\chi_{1,\eps}Q_{v(\mathbf{d})}(\phi+\psi_\sharp)+\eps^{-1}\chi_{1,\eps}(f'(v(\mathbf{d}))-q_\sharp)\psi_\sharp\big) &\text{in $\Sigma(\eps^2\mathbf{d})\times\R$,}
\end{aligned} 
\end{equation}
and
\begin{equation}
\label{eq_far}
\begin{aligned}
L_0\psi=-(1-\chi^\sharp_{2,\eps})N_\eps({\tt w}_\eps(\mathbf{d}))-(1-\chi^\sharp_{1,\eps})Q_{{\tt w}_\eps}(\chi^\sharp_{2,\eps}\phi^\sharp+\psi)
-\eps[\Delta,\chi^\sharp_{2,\eps}]\phi^\sharp \quad\text{ in $\R^3$.}
\end{aligned}
\end{equation}
in $\phi$ and $\psi$. The aim now is to reduce equation (\ref{eq_near}) to an equation on $\Sigma\times\R$, through the diffeomrphism $\beta(\eps^2\mathbf{d})$. In order to explain how our argument works, we state a preliminary Lemma, which will be useful in the sequel
\begin{lemma}
Let $M>0$, $a\in(0,\bar{a})$, $\alpha\in(0,1)$, $\mathbf{d}^i\in(B_M)^k$, $i=1,2$, and $h_0\in \mathcal{C}^{2,\alpha}_a(\Sigma)$. Then
$$\|\big(\Delta_{\Sigma(\eps^2\mathbf{d}^1)}(h_0\circ\beta(\eps^2\mathbf{d}^1))\big)\circ\beta(\eps^2\mathbf{d}^1)^{-1}-
\big(\Delta_{\Sigma(\eps^2\mathbf{d}^2)}(h_0\circ\beta(\eps^2\mathbf{d}^2))\big)\circ\beta(\eps^2\mathbf{d}^2)^{-1}\|_{\mathcal{C}^{0,\alpha}_a(\Sigma)}\le C\eps^2|\mathbf{d}^1-\mathbf{d}^2|\|h_0\|_{\mathcal{C}^{2,\alpha}_a(\Sigma)},$$
for some constant $C>0$.\label{lemma_Delta_Sigma_d}
\end{lemma}
\begin{remark}
In particular, taking, for instance $\mathbf{d}^2=0$, we can see that $$\|\big(\Delta_{\Sigma(\eps^2\mathbf{d})}(h_0\circ\beta(\eps^2\mathbf{d}))\big)\circ\beta(\eps^2\mathbf{d})^{-1}-\Delta_\Sigma h_0\|_{\mathcal{C}^{0,\alpha}_a(\Sigma)}\le C\eps^2|\mathbf{d}|\|h_0\|_{\mathcal{C}^{2,\alpha}_a(\Sigma)}.$$
Roughly, passing from $\Sigma(\eps^2\mathbf{d})$ to $\Sigma$ does not preserves exactly the differential operators, however it produces an error which is small enough, of order $\eps^2$.
\end{remark}
\begin{proof}
It is enough to observe that, if we restrict ourselves to an end of $\Sigma(\eps^2\mathbf{d})$, using that $h_0\circ\beta(\eps^2\mathbf{d})\circ Y_j(\eps^2\mathbf{d}_j)=h_0\circ\beta(\eps^2\mathbf{d})$, we have
\begin{multline*}
\Delta_{\Sigma(\eps^2\mathbf{d})}(h_0\circ\beta(\eps^2\mathbf{d}))\circ Y_j(\eps^2\mathbf{d}_j)\\=\frac{1}{\sqrt{|g_{\Sigma(\eps^2\mathbf{d})}\circ Y_j(\eps^2\mathbf{d}_j)|}}\partial_m\bigg(\sqrt{|g_{\Sigma(\eps^2\mathbf{d})}\circ Y_j(\eps^2\mathbf{d}_j)|}g_{\Sigma(\eps^2\mathbf{d})}^{lm}\circ Y_j(\eps^2\mathbf{d}_j)\partial_l(h_0\circ Y_j)\bigg).
\end{multline*}
Hence the conclusion follows from Lemma \ref{lemma_Sigma_d}.
\end{proof}
Composing with the diffeomorphism $\beta(\eps^2\mathbf{d})^{-1}\times Id_\R$ we can see that (\ref{eq_near}) is equivalent to 
\begin{equation}
\label{eq_near_Sigma}
\begin{aligned}
&\eps\Delta_\Sigma\phi_0+\eps^{-1}\partial^2_t\phi_0+\eps^{-1}f'(v_\star)\phi_0=-\chi_{3,\eps}
\big(N_\eps(v(\mathbf{d}))
+\eps^{-1}\chi_{1,\eps}Q_{v(\mathbf{d})}(\phi+\psi_\sharp)\\
&\qquad +\text{L}_{\eps,\mathbf{d}}\phi+\eps^{-1}\chi_{1,\eps}(f'(v(\mathbf{d}))-q_\sharp)\psi_\sharp\big)(\beta(\eps^2\mathbf{d})^{-1}\times Id_\R)+\chi_{3,\eps}\text{R}_{\eps,\mathbf{d}}\phi_0&\text{in $\Sigma\times\R$.}
\end{aligned}
\end{equation}
Arguing as in the proof of Lemma \ref{lemma_Delta_Sigma_d}, we can see that the remainder $\text{R}_{\eps,\mathbf{d}}$, coming from the difference between the Laplace-Beltrami operator of $\Sigma(\eps\mathbf{d})$ and the one of $\Sigma$, satisfies
\begin{equation}
\label{ets_R_epsd}
\|\text{R}_{\eps,\mathbf{d}}\phi_0\|_{\mathcal{E}^{0,\alpha}_{a,\gamma}(\Sigma\times\R)}\le C\eps^2|\mathbf{d}|\|\phi_0\|_{\mathcal{E}^{2,\alpha}_{a,\gamma}(\Sigma\times\R)},
\end{equation}
\begin{equation}
\label{lip_R_eps_d}
\begin{aligned}
&\|\text{R}_{\eps,\mathbf{d}^1}\phi^1_0-\text{R}_{\eps,\mathbf{d}^2}\phi^2_0\|_{\mathcal{E}^{2,\alpha}_{a,\gamma}(\Sigma\times\R)}\le C\|\phi^1_0-\phi^2_0\|_{\mathcal{E}^{2,\alpha}_{a,\gamma}(\Sigma\times\R)}\\
&\qquad +C\eps^2|\mathbf{d}^1_j-\mathbf{d}^2_j|(\|\phi^1_0\|_{\mathcal{E}^{2,\alpha}_{a,\gamma}(\Sigma\times\R)}
+\|\phi^2_0\|_{\mathcal{E}^{2,\alpha}_{a,\gamma}(\Sigma\times\R)}),
\end{aligned}
\end{equation}
for some constant $C>0$, for any $M>0$, for any $\mathbf{d},\mathbf{d}^i\in (B_M)^k$, $\phi_0,\phi^i_0\in \mathcal{E}^{2,\alpha}_{a,\gamma}(\Sigma\times\R)$, with $\|\phi_0\|_{\mathcal{E}^{2,\alpha}_{a,\gamma}(\Sigma\times\R)}<1$, $\|\phi^i_0\|_{\mathcal{E}^{2,\alpha}_{a,\gamma}(\Sigma\times\R)}<1$, $i=1,2$.\\

First we fix $\phi_0$, $\mathbf{d}$ and $h_0$ and we solve (\ref{eq_far}).
\begin{proposition}
Let $M>0$, $a\in(0,\bar{a})$, $\gamma\in(0,\sqrt{2})$, $\alpha\in(0,1/2)$ and $\mathbf{d}\in (B_M)^k$. Let $\phi_0\in \mathcal{E}^{2,\alpha}_{a,\gamma}(\Sigma\times\R)$ be such that $\|\phi_0\|_{\mathcal{E}^{2,\alpha}_{a,\gamma}(\Sigma\times\R)}<1$ and $h_0\in \mathcal{C}^{2,\alpha}_a(\Sigma)$ be such that $\|h_0\|_{C^{2,\alpha}_a(\Sigma)}<M$. Let $a'\in (0,a)$, $\gamma'\in(0,\gamma)$. Then there exists a solution $\psi:=\psi(\phi_0,h_0,\mathbf{d})\in C^{2,\alpha}_{a',\gamma'}(\R^3)$ to (\ref{eq_far}) such that
\begin{eqnarray}\notag
\|\psi(\phi_0,h_0,\mathbf{d})\|_{C^{2,\alpha}_{a',\gamma'}(\R^3)}\le \Lambda_3 e^{-\lambda_3/\eps},
\end{eqnarray}
for some constants $\Lambda_3>0$ and $\lambda_3>0$, with $\lambda_3$ depending on $a'-a$ and $\gamma'-\gamma$. Moreover, there exists $C>0$ such that
\begin{multline*}
\|\psi(\phi_0^1,h_0^1,\mathbf{d}^1)-\psi(\phi_0^2,h_0^2,\mathbf{d}^2)\|_{C^{2,\alpha}_{a',\gamma'}(\R^3)}\\
\le 
 C e^{-\lambda_3/\eps}\big(\|\phi_0^1-\phi_0^2\|_{\mathcal{E}^{2,\alpha}_{a,\gamma}(\Sigma\times\R)}
+|h_0^1-h_0^2\|_{\mathcal{C}^{2,\alpha}_{a}(\Sigma)}+|\mathbf{d}^1-\mathbf{d}^2|\big),
\end{multline*}
for any $h^i_0\in\mathcal{C}^{2,\alpha}_a(\Sigma)$, $\|h^i_0\|_{\mathcal{C}^{2,\alpha}_a(\Sigma)}<M$, $\mathbf{d}^i\in (B_M)^k$, $\phi_0^i\in\mathcal{E}^{2,\alpha}_{a,\gamma}(\Sigma\times\R)$, $\|\phi_0^i\|_{\mathcal{E}^{2,\alpha}_{a,\gamma}(\Sigma\times\R)}<1$, $i=1,2$.
\label{prop_far}
\end{proposition} 
This proposition will be proved in subsection \ref{subsec_far}. After that, it remains to solve equation (\ref{eq_near_Sigma}). In order to do so, we define the space
$$\mathcal{X}:=\{\phi\in L^2 (\Sigma\times\R):\int_\R \phi(y,t) v'_\star(t) dt=0,\text{ for a.e. $y\in\Sigma$} \}$$
and the projection
$$\Pi_{\mathcal{X}}:L^2 (\Sigma\times\R)\to\mathcal{X}$$
by
$$\Pi_{\mathcal{X}}\phi(y,t):=\phi(y,t)-\frac{\int_\R \phi(y,\mu)v'_\star(\mu)d\mu}{\int_\R (v'_\star(\mu))^2 d\mu} v'_\star(t).$$
Denoting the right-hand side of (\ref{eq_near_Sigma}) by $\mathcal{G}^6_\eps(h_0,\phi_0,\mathbf{d})$ and composing with the orthogonal projections $\Pi_{\mathcal{X}}$ and $Id-\Pi_{\mathcal{X}}$, we can reduce ourselves to solve
\begin{equation}
\eps\Delta_\Sigma \phi_0+\eps^{-1}\partial^2_t \phi_0+\eps^{-1}f'(v(\mathbf{d}))\phi_0=\Pi_{\mathcal{X}}\mathcal{G}^6_\eps(h_0,\mathbf{d})\label{eq_X}
\end{equation}
\begin{equation}
(Id-\Pi_{\mathcal{X}})\mathcal{G}^6_\eps(h_0,\mathbf{d})=0\label{eq_Xbot}
\end{equation}
Equation (\ref{eq_X}) can be solved through a fixed point argument, exploiting the orthogonality to $v'_\star$.
\begin{proposition}
Let $M>0$, $\alpha\in(0,1/2)$, $\gamma\in(0,\sqrt{2})$, $a\in(0,\bar{a})$, $\mathbf{d}\in (B_M)^k$ and $h_0\in \mathcal{C}^{2,\alpha}_a(\Sigma)$, with $\|h_0\|_{\mathcal{C}^{2,\alpha}_a(\Sigma)}<M$. Then there exists a solution $\phi_0:=\phi_0(h_0,\mathbf{d})\in\mathcal{E}^{2,\alpha}_{a,\gamma}(\Sigma\times\R)$ to (\ref{eq_X}) such that
$$\|\phi_0(h_0,\mathbf{d})\|_{\mathcal{E}^{2,\alpha}_{a,\gamma}(\Sigma\times\R)}\le \Lambda_4\eps^{3-\alpha}.$$
for some constant $\Lambda_4>0$. Moreover
\begin{equation}
\begin{aligned}
&\|\phi_0(h_0^1,\mathbf{d}^1)-\phi_0(h_0^2,\mathbf{d}^2)\|_{\mathcal{E}^{2,\alpha}_{a,\gamma}(\Sigma\times\R)}\le C (\eps^{4-\alpha}\|h_0^1-h_0^2\|_{\mathcal{C}^{2,\alpha}_{a}(\Sigma)}+\eps^{5-\alpha}|\mathbf{d}^1-\mathbf{d}^2|).
\end{aligned}
\end{equation}
for $\mathbf{d}_i\in (B_M)^k$, $h_0^i\in \mathcal{C}^{2,\alpha}_a(\Sigma)$, $\|h_0^i\|_{\mathcal{C}^{2,\alpha}_a(\Sigma)}<M$, $i=1,2$.
\label{prop_near_X}
\end{proposition}
It remains to solve (\ref{eq_Xbot}). 
\begin{lemma}
Let $M>0$, $\alpha\in(0,1/2)$ and $a\in(0,\bar{a})$. Equation (\ref{eq_Xbot}) is equivalent to an equation of the form
\begin{eqnarray}
\mathcal{J}_\Sigma (h_0+\Phi(\mathbf{d}))=\mathcal{F}^2_\eps(h_0,\mathbf{d}),\label{eq_Sigma}
\end{eqnarray}
where $\mathcal{F}^2_\eps$ satisfies
\begin{equation}\notag
\begin{aligned}
&\|\mathcal{F}^2_\eps(h_0,\mathbf{d})\|_{\mathcal{C}^{0,\alpha}_a(\Sigma)}\le C(1+M e^{(\bar{a}-a)s_0}),\\
&\|\mathcal{F}^2_\eps(h^1_0,\mathbf{d}^1)-\mathcal{F}^2_\eps(h^2_0,\mathbf{d}^2)\|_{\mathcal{C}^{0,\alpha}_a(\Sigma)}\le C(\eps^2+e^{(\bar{a}-a)s_0})|\mathbf{d}^1-\mathbf{d}^2|+C\eps\|h^1_0-h^2_0\|_{\mathcal{C}^{2,\alpha}_a(\Sigma)},
\end{aligned}
\end{equation}
for some constant $C>0$, for any $h_0,h^i_0\in\mathcal{C}^{2,\alpha}_a(\Sigma)$, $\|h^i_0\|_{C^{2,\alpha}(\R)}<M$, $\mathbf{d},\mathbf{d}^i\in (B_M)^k$, $i=1,2$.
\end{lemma}
\begin{proof}
We deal with the term $\chi_{3,\eps}N_\eps(v(\mathbf{d}))$, which is the main term. According to Lemma \ref{lemma_dec_error}, 
\begin{equation}
\begin{aligned}
\chi_{3,\eps}N_\eps(v(\mathbf{d}))=&\chi_{3,\eps}\big(-\eps^2\mathcal{J}_{\Sigma(\eps^2\mathbf{d})}(h_0\circ \beta(\eps^2\mathbf{d}))-(H_{\Sigma(\eps^2\mathbf{d})}-H_\Sigma)\big)U'
+\chi_{3,\eps}\mathcal{G}^5_\eps(\mathbf{h},\mathbf{d})\\
&=\chi_{3,\eps}\big(-\eps^2\mathcal{J}_{\Sigma(\eps^2\mathbf{d})}(h_0\circ \beta(\eps^2\mathbf{d}))-\eps^2\mathcal{J}_\Sigma\Phi(\mathbf{d})\circ\beta(\eps^2\mathbf{d})
-(H_{\Sigma(\eps^2\mathbf{d})}\\
&-\eps^2\mathcal{J}_\Sigma\Phi(\mathbf{d})\circ\beta(\eps^2\mathbf{d})-H_\Sigma)\big)U'+\chi_{3,\eps}\mathcal{G}^5_\eps(\mathbf{h},\mathbf{d})
\end{aligned}
\end{equation}
with $\mathcal{G}^5_\eps$ of order $\eps^2$. By Lemma \ref{lemma_Delta_Sigma_d},
$$\big(\mathcal{J}_{\Sigma(\eps^2\mathbf{d})}(h_0\circ\beta(\eps^2\mathbf{d}))\big)\circ\beta(\eps^2\mathbf{d})^{-1}=\mathcal{J}_\Sigma h_0+\mathcal{F}^3_\eps(h_0,\mathbf{d}),$$
with $\|\mathcal{F}^3_\eps(h_0,\mathbf{d})\|_{\mathcal{C}^{0,\alpha}_a(\Sigma)}\le C\eps^2|\mathbf{d}|\|h_0\|_{\mathcal{C}^{2,\alpha}_a(\Sigma)}$. Therefore, multiplying (\ref{eq_near}) by $v'_\star$, integrating over $\R$, composing with $\beta(\eps^2\mathbf{d})^{-1}$ and recalling Lemma \ref{lemma_H_Sigmad}, we can reduce ourselves to an equation of the form
$$\mathcal{J}_\Sigma(h_0+\Phi(\mathbf{d}))=\mathcal{F}^2_\eps(h_0,\mathbf{d}),$$
with $\mathcal{F}^2_\eps$ satisfying the required estimates. Note that this also follows from the choice of $\varphi_j$ and $\text{h}_j$, which simplifies the periodic part of the error along the ends, as we show in Proposition \ref{prop_ls_end}, and the upper bound (\ref{est_commutators}). Observe  the we also use the precise size of the corrections $\varphi_j$, which are of order $\eps^{3-\alpha}$ in the appropriate weighted norm. 
\end{proof}
\begin{proposition}
Let $\alpha\in(0,1/2)$ and $a\in(0,\bar{a})$. Then equation (\ref{eq_Sigma}) admits a solution $h_0+\Phi(\mathbf{d})\in \mathcal{C}^{2,\alpha}_a(\Sigma)\oplus \mathcal{D}(\Sigma)$, provided $s_0$ and $M$ are large enough and $\eps$ is small enough.
\end{proposition}
\begin{proof}
We can solve this equation thanks to the non degeneracy of $\Sigma$. In fact, denoting by
$$\Pi:\mathcal{C}^{2,\alpha}_a(\Sigma)\oplus\mathcal{D}(\Sigma)\to \mathcal{C}^{2,\alpha}_a(\Sigma)\times\R^{6k}$$
the projection of a function in $\mathcal{C}^{2,\alpha}_a(\Sigma)\oplus\mathcal{D}(\Sigma)$ onto the components, we can reduce (\ref{eq_Sigma}) to finding a fixed point of the mapping,
$$(h_0,\mathbf{d})\in \tilde{B}_M\times (B_M)^k\mapsto \Pi\mathcal{J}^{-1}_\Sigma\mathcal{F}^2_\eps(h_0,\mathbf{d}),$$
where
$$\tilde{B}_M:=\{h_0\in \mathcal{C}^{2,\alpha}_a(\Sigma):\,\|h_0\|_{\mathcal{C}^{2,\alpha}_a(\Sigma)}<M\}.$$
It is possible to show that we actually have a contraction, in fact
$$\|\Pi\mathcal{J}^{-1}_\Sigma\mathcal{F}^2_\eps(h_0,\mathbf{d})\|_{\mathcal{C}^{0,\alpha}(\Sigma)\times\R}\le C(1+Me^{-(\bar{a}-a)s_0})<M$$
if, for instance $M\ge 2C$, $s_0>0$ is large enough and $\eps$ is small enough.
\end{proof}

\subsection{The equation far from $\Sigma$}\label{subsec_far}
In order to solve equation (\ref{eq_far}), we first study the linear theory for the operator $L_0:=\eps\Delta+\eps^{-1}V(x)$, where the potential is given by $V(x):=-f'({\tt w}_\eps(\mathbf{d}))(1-\chi^\sharp_{1,\eps})-q\chi^\sharp_{1,\eps}$ in all $\R^3$. Although we do not make it explicit through the notation, the potential $V$  depends on $\mathbf{d}$. However, as we will see in the sequel, this dependence is mild enough for our purposes. Thanks to coercivity, this operator is invertible, and its inverse will be used to solve our equation by a fixed point argument.
\subsubsection{The linear problem}
We are interested in the equation
\begin{eqnarray}
\eps\Delta \phi-\eps^{-1}V(x)\phi=g \qquad \text{in $\R^3$}, \qquad g\in C^{0,\alpha}_{a,\gamma}(\R^3)\label{eq_far_lin}
\end{eqnarray}
Scaling the variables, we can reduce ourselves to solve the equation 
\begin{eqnarray}
\Delta \phi-V_\eps(x)\phi =g \label{eq_far_lin_scaled} \qquad \text{in $\R^3$}, \qquad V_\eps(x):=V(\eps x),
\end{eqnarray}
with $g\in C^{0,\alpha}_{loc}(\R^3)$ such that $\|g \Gamma_\eps\|_{C^{0,\alpha}(\R^3)}<\infty$. In this setting, we prove the following result.
\begin{lemma}
Let $M>0$, $\alpha\in(0,1)$, $a\in(0,\bar{a})$, $\gamma\in(0,\sqrt{2})$ and $\mathbf{d}\in (B_M)^k$. Let $g\in C^{0,\alpha}_{loc}(\R^3)$ be a function such that $\|g \Gamma_\eps\|_{C^{0,\alpha}(\R^3)}<\infty$. Then, for $\eps$ small enough, there exists a unique solution $\phi\in C^{2,\alpha}_{loc}(\R^3)$ to (\ref{eq_far_lin_scaled}) such that
\begin{eqnarray}
\|\phi \Gamma_\eps\|_{C^{2,\alpha}(\R^3)}\le C\|g \Gamma_\eps\|_{C^{0,\alpha}(\R^3)},\label{est_eq_far}
\end{eqnarray}
for some constant $C>0$.
\label{lemma_far_scaled}
\end{lemma} 
\begin{proof}
The proof of this Lemma is split into three steps.\\

\textit{Step (i). Existence, uniqueness and boundedness.}\\

Firs we consider the truncated problem
\begin{eqnarray}\notag
-\Delta \phi_R+V_\eps\phi_R=g &\text{in $B_R$}, \qquad\phi_R\in H^1_0(B_R).
\end{eqnarray}
Existence and uniqueness follow from variational arguments, since the potential is uniformly positive. Testing the equation with the solution we have the bound
\begin{eqnarray}
\|\phi_R\|_{H^1(B_R)}\le C\|g\|_{L^2(\R^3)},\label{unif_H1_bound}
\end{eqnarray}
for some constant $C>0$ independent of $R$, $M$ and $\mathbf{d}$. Here it is crucial to use that $g\in L^2(\R^3)$, which follows from the fact that not two of the ends of $\Sigma$ are parallel (see also Remark \ref{rem_angles}). Therefore, using the Sobolev embeddings and the elliptic estimates (see theorem $8.12$ of \cite{GT}), we can prove that $\phi_R\in L^\infty(B_R)$ and
\begin{eqnarray}
\|\phi_R\|_{L^\infty(B_R)}\le C\|\phi_R\|_{W^{2,2}(B_R)}\le C_R(\|\phi_R\|_{L^2(B_R)}+\|g\|_{L^2(B_R)})\le C_R\|g\|_{L^2(\R^3)}<\infty. \label{unif_Linfty_bound}
\end{eqnarray}
Note that the constants here do depend on $R$, however this is totally irrelevant, since we just need to prove that $\phi_R\in L^\infty(B_R)$. \\

%\textit{Step (ii). The decay: $\phi\Gamma_\eps	in L^\infty(\R^3)$}\\

%The proof of this fact essentially relies on comparing $\phi$ with a suitable barrier, using the maximum principle. We fix a point $z\in\R^3\backslash\Sigma_\eps$, $0<\delta<dist(z,\Sigma_\eps)$ and we define the barrier as $\lambda\Gamma_\eps^{-1}+\nu\Gamma_\eps$, where $\lambda>0$ and $\nu>0$ are constants that will be chosen in the sequel. The idea is to apply the maximum principle to the bounded domain
%$$\Omega_R:=\{x\in \R^3:\,dist(x,\Sigma_\eps)>\delta\}\cap B_R,$$
%with $R>0$ so large that $z\in\Omega_R$. We can see that, for any $x\in\partial\Omega_R\cap B_R$,
%$$\phi(x)\le\|\phi\|_{L^\infty}\le \lambda \Gamma_\eps^{-1}(x)\le \lambda\Gamma_\eps^{-1}(x)+\nu\Gamma_\eps(x)$$
%provided $\lambda>\|\phi\|_{L^\infty}\sup_{dist(x,\Sigma_\eps)}\Gamma_\eps(x)>c\|\phi\|_{L^\infty} e^{c\gamma\delta}$, for some constant $c>0$. 

\textit{Step (ii). Estimate in the weighted norm: there exists a constant $C>0$ such that}
\begin{eqnarray}\notag
\|\phi_R \Gamma_\eps\|_{L^\infty(B_R)}\le C\|g\Gamma_\eps\|_{L^\infty(\R^3)}.
\end{eqnarray}
We assume that $\phi_R\ne 0$, otherwise the claim is trivial. We compute the equation satisfied by $\tilde{\phi}:=\phi_R\Gamma_\eps$, that is 
\[
\begin{aligned}
&-\Delta \tilde{\phi}+V_\eps\tilde{\phi}=\tilde{g}-2\Gamma_\eps^{-1}\nabla\tilde{\phi}\cdotp\nabla\Gamma_\eps
+\tilde{\phi}\bigg(2\Gamma_\eps^{-2}|\nabla\Gamma_\eps|^2-\Gamma_\eps^{-1}\Delta\Gamma_\eps\bigg)\\
&\qquad \le \tilde{g}-2\Gamma_\eps^{-1}\nabla\tilde{\phi}\cdotp\nabla\Gamma_\eps+
(2\frac{|\nabla\Gamma_\eps|^2}{\Gamma_\eps^2}-\frac{\nabla\Gamma}{\Gamma})\tilde{\phi}.
\end{aligned}
\]
Since the domain is a ball and $\phi_R$ satisfies Dirichlet boundary conditions,  there exists a point $y\in B_R$ such that $|\tilde{\phi}(y)|=\|\tilde{\phi}\|_{L^\infty(B_R)}>0$. If, for instance, $\tilde{\phi}(y)>0$, then $y$ is a maximum point, thus
$$(2+o(1))\tilde{\phi}(y)\le -\Delta\tilde{\phi}(y)+V_\eps(y)\tilde{\phi}(y)\le \|\tilde{g}\|_{L^\infty(\R^3)}+(\gamma^2+o(1))\tilde{\phi}(y),$$
which proves the claim, provided $\gamma\in(0,\sqrt{2})$ and $\eps$ is small enough, due to the fact that the potential is bounded below by $2+o(1)>0$. Here $o(1)$ is taken in the usual sense of functions of one variable, in this case $\eps$, which goes to $0$. Replacing $\phi$ with $-\phi$, we have the statement.\\

\textit{Step (iii): passing to the limit.}\\

Since the constant appearing in (\ref{unif_H1_bound}) is independent of $R$, $\phi_R$ is bounded in $H^1(\R^3)$, thus there exists a sequence $R_n\to\infty$ such that $\phi_{R_n}$ converges weakly in $H^1(\R^3)$ to an entire solution $\phi\in C^{2,\alpha}_{loc}(\R^3)$. Up to a subsequence, the convergence is also strong in $L^2_{loc}(\R^3)$ and pointwise almost everywhere, therefore, by step \textit{(ii)}, $\phi\in C^{0,\alpha}_{a,\gamma}(\R^3)$ and  
$$\|\phi \Gamma_\eps\|_{L^\infty(\R^3)}\le C\|g\Gamma_\eps\|_{L^\infty(\R^3)}.$$
Therefore, by local elliptic estimates (see Corollary $6.3$ of \cite{GT}) applied to the equation satisfied by $\phi\Gamma_\eps$, we can see that $\phi\in C^{2,\alpha}_{a,\gamma}(\R^3)$ and satisfies
\begin{eqnarray}\notag
\|\phi\Gamma_\eps\|_{C^{2,\alpha}(\R^3)}\le C\|g\Gamma_\eps\|_{C^{0,\alpha}(\R^3)},%\label{est_eq_far_R}
\end{eqnarray}
for some constant $C>0$.% independent of $R$. As a consequence, by a compactness argument, to the original equation in the whole $\R^3$. Passing to the limit in (\ref{est_eq_far_R}), we can see that $\phi$ satisfies (\ref{est_eq_far}). 
\end{proof}
 
Now we apply a scaling argument to go back to our original problem

\begin{lemma}
Let $M>0$, $\alpha\in(0,1)$, $a\in(0,\bar{a})$, $\gamma\in(0,\sqrt{2})$ and $\mathbf{d}\in (B_M)^k$. For any $g\in C^{0,\alpha}_{a,\gamma}(\R^3)$ and $\eps$ small enough, there exists a unique solution $\phi:=\Psi^3(g)\in C^{2,\alpha}_{a,\gamma}(\R^3)$ to (\ref{eq_far_lin}) such that
$$\|\Psi^3(g)\|_{C^{2,\alpha}_{a,\gamma}(\R^3)}\le C\eps \|g\|_{C^{0,\alpha}_{a,\gamma}(\R^3)}.$$
\label{lemma_far}
\end{lemma}
\begin{proof}
In view of Lemma \ref{lemma_far_scaled}, there exists a unique solution $\phi_\eps$ to the equation
$$\Delta\phi_\eps-V_\eps(x)\phi_\eps=\eps g_\eps(x), \qquad\text{in $\R^3$}, \qquad g_\eps(x):=g(\eps x)$$
such that 
$$\|\phi_\eps\Gamma_\eps\|_{C^{2,\alpha}(\R^3)}\le C\eps \|g_\eps\Gamma_\eps\|_{C^{2,\alpha}(\R^3)}.$$
Then $\phi(x):=\phi_\eps(x/\eps)$ solves (\ref{eq_far_lin}) and satisfies the required estimate.
\end{proof}
\subsubsection{The proof of Proposition \ref{prop_far}: a fixed point argument}
The argument is based on the contraction mapping theorem, applied to the equation
$$\psi=\Psi^3(-(1-\chi^\sharp_{2,\eps})N_\eps({\tt w}_\eps(\mathbf{d}))-(1-\chi^\sharp_{1,\eps})Q_{w_\eps}(\chi^\sharp_{2,\eps}\phi^\sharp+\psi)
-\eps[\Delta,\chi^\sharp_{2,\eps}]\phi^\sharp)$$
in the ball
$$B_{\Lambda_3}:=\{\psi\in C^{2,\alpha}_{a',\gamma'}(\R^3):\|\psi\|_{C^{2,\alpha}_{a',\gamma'}(\R^3)}\le \Lambda_3 e^{-\lambda_3/\eps}\},$$
where $\lambda_3$ is an appropriate constant, depending on $a-a'$ and $\gamma-\gamma'$. We note that, thanks to the fact that $a'<a$ and $\gamma'<\gamma$, all the norms are actually finite, and, thanks to the cutoff functions $\chi^\sharp_{1,\eps}$ and $\chi^\sharp_{2,\eps}$, we actually have a contraction of $B_{\Lambda_3}$.

%\newpage

\subsection{The equation near $\Sigma$}\label{section_near}

\subsubsection{The linear problem}
This subsection is devoted to the study of the linear equation
\begin{eqnarray}
\eps\Delta_\Sigma\phi+\eps^{-1}\partial^2_t \phi+\eps^{-1}f'(v_\star(t))\phi=g, &\text{in $\Sigma\times\R$.} \label{eq_near_lin}
\end{eqnarray}
%We define the suitable weighted norms to deal with our problem. 
We look for solutions satisfying the orthogonality condition
\begin{eqnarray}
\int_{\R}\phi(y,t)v'_\star(t)dt=0, \qquad\forall \, y\in \Sigma \label{cond_ort}
\end{eqnarray}
Once again, in order to deal with the problem, we need to rescale the variables and consider the equation
\begin{eqnarray}
\Delta_{\Sigma_\eps}\phi+\partial^2_{{\tt t}} \phi+f'(v_\star({\tt t}))\phi=g &\text{in $\Sigma_\eps\times\R$,} \label{eq_near_lin_scaled}
\end{eqnarray}
where $\Sigma_\eps:=\eps^{-1}\Sigma$ is the scaled version of $\Sigma$ and the new coordinates are
$$
\begin{cases}
{\tt t}=t\\
{\tt y}=\eps^{-1}y.
\end{cases}$$ 
For this scaled problem, we prove existence and uniqueness under the orthogonality condition
\begin{eqnarray}
\int_{\R}\phi({\tt y},{\tt t})v'_\star({\tt t})d{\tt t}=0, \qquad\forall \, {\tt y}\in \Sigma_\eps \label{cond_ort_scaled},
\end{eqnarray}
and we estimate the inverse in terms of the norms
$$\|\phi\|_{\mathcal{C}^{n,\alpha}_{a,\gamma}(\Sigma_\eps\times\R)}:=\sum_{0\le m+k\le n}\|\partial_t^k D^m_{\Sigma_\eps}\phi\|_{\mathcal{C}^{0,\alpha}_{a,\gamma}(\Sigma_\eps\times\R)},$$
where $a>0$, $\gamma>0$, $n\in\N$, $\alpha\in(0,1)$ and
\[
\begin{aligned}
&\|\phi\|_{\mathcal{C}^{0,\alpha}_{a,\gamma}(\Sigma_\eps\times\R)}:=\|\phi w_{a,\eps}({\tt y})\cosh^\gamma({\tt t})\|_{C^{0,\alpha}(\Sigma_\eps\times\R)} \\
&w_{a,\eps}({\tt y}):=w_a(\eps {\tt y}), \qquad \forall \, {\tt y}\in\Sigma_\eps. %=\zeta_0(\eps{\tt y})+\sum_{j=1}^k \zeta_j(\eps{\tt y}) e^{a\eps {\tt s}_j({\tt y})}, \qquad 
%\eps {\tt s}_j({\tt y})=s_j(\eps {\tt y}).
\end{aligned}
\]
We recall that $w_a$ is defined in (\ref{def_wa}) and $\zeta_j$ are defined in (\ref{def_zeta_j}). As usual, we say that a function $\phi\in C^{n,\alpha}_{loc}(\Sigma_\eps\times\R)$ is in $\mathcal{C}^{n,\alpha}_{a,\gamma}(\Sigma_\eps\times\R)$ if the above norm is finite.
\begin{lemma}
Let $\alpha\in(0,1)$, $\gamma\in(0,\sqrt{2})$ and $a\in(0,\bar{a})$. Let $g\in \mathcal{C}^{0,\alpha}_{a,\gamma}(\Sigma_\eps\times\R)$ satisfy the orthogonality condition (\ref{cond_ort_scaled}). Then, for $\eps$ small enough, there exists a unique solution $\phi\in \mathcal{C}^{2,\alpha}_{a,\gamma}(\Sigma_\eps\times\R)$ to (\ref{eq_near_lin_scaled}) satisfying (\ref{cond_ort_scaled}). Moreover, it satisfies the estimate
\begin{eqnarray}
\|\phi\|_{\mathcal{C}^{2,\alpha}_{a,\gamma}(\Sigma_\eps\times\R)}\le C\|g\|_{\mathcal{C}^{0,\alpha}_{a,\gamma}(\Sigma_\eps\times\R)}, \label{est_norm_near_scaled}
\end{eqnarray}  
for some constant $C>0$.\label{lemma_lin_near_scaled}
\end{lemma}
Given a function $\phi:\Sigma\times\R\to\R$, we obtain a function $\phi_\eps:\Sigma_\eps\times\R\to\R$ by setting $\phi_\eps(\tt{y},\tt{t}):=\phi(\eps\tt{y},\tt{t})$, for $(\tt{y},\tt{t})\in\Sigma_\eps\times\R$. Their weighted norms satisfy the relation
\begin{eqnarray}
c\eps^\alpha\|\phi\|_{\mathcal{E}^{n,\alpha}_{a,\gamma}(\Sigma\times\R)}\le \|\phi_\eps\|_{\mathcal{C}^{n,\alpha}_{a,\gamma}(\Sigma_\eps\times\R)}\le C \|\phi\|_{\mathcal{E}^{n,\alpha}_{a,\gamma}(\Sigma\times\R)}, \qquad 0<c<C, \, n\ge 0,\label{compare_norms_Sigma}
\end{eqnarray}
thus a scaling arguments proves the following result.
\begin{lemma}
Let $\alpha\in(0,1)$, $\gamma\in(0,\sqrt{2})$ and $a\in(0,\bar{a})$. Let $g\in \mathcal{E}^{0,\alpha}_{a,\gamma}(\Sigma\times\R)$ be a function satisfying the orthogonality condition (\ref{cond_ort}). Then, for $\eps$ small enough, there exists a unique solution $\phi\in\mathcal{E}^{2,\alpha}_{a,\gamma}(\Sigma\times\R)$ to (\ref{eq_near_lin}) satisfying (\ref{cond_ort}). Moreover, it fulfils the estimate
\begin{eqnarray}
\|\phi\|_{\mathcal{E}^{2,\alpha}_{a,\gamma}(\Sigma\times\R)}\le C\eps^{1-\alpha}\|g\|_{\mathcal{E}^{0,\alpha}_{a,\gamma}(\Sigma\times\R)}, \label{est_norm_near}
\end{eqnarray}  
for some constant $C>0$.\label{lemma_lin_near}
\end{lemma}
%Once again, in order to apply the contraction mapping theorem, we need to take $\alpha\in(0,1/2)$.\\
Now we focus on the proof of Lemma \ref{lemma_lin_near_scaled}. We start by proving existence and uniqueness, then we will prove the decay of the solution and the estimate in the suitable weighted norm.

\begin{lemma}
Let $\alpha\in(0,1)$, $\gamma\in(0,\sqrt{2})$ and $a\in(0,\bar{a})$. Let $g\in \mathcal{E}^{0,\alpha}_{a,\gamma}(\Sigma_\eps\times\R)$ be a function satisfying (\ref{cond_ort_scaled}). Then, for $\eps$ small enough, there exists a unique solution $\phi\in H^1(\Sigma_\eps\times\R)$ to (\ref{eq_near_lin_scaled}) satisfying (\ref{cond_ort_scaled}). Moreover, there exists a constant $C(a,\gamma)>0$ depending on $a$ and $\gamma$ such that
\begin{eqnarray}
\|\phi\|_{L^\infty(\Sigma_\eps\times\R)}\le C(a,\gamma)\|g w_{a,\eps}({\tt y})\cosh^\gamma({\tt t})\|_{L^\infty(\Sigma_\eps\times\R)}.\label{est_phi_infty}
\end{eqnarray}
\label{lemma_existence}
\end{lemma}
\begin{proof}
%For $R>0$, we set $\Sigma_{\eps,R}:=\Sigma_\eps\cap B_R$ and consider the truncated problem
%\begin{eqnarray}\notag
%\Delta_{\Sigma_\eps}\phi_R+\partial^2_t \phi_R+f'(v_\star(\text{t}))\phi_R=g &\text{in $\Sigma_{\eps,R}\times\R$,}, \qquad \phi_R=0 &\text{on $\partial\Sigma_{\eps,R}\times\R$.}
%\end{eqnarray}
By the spectral decomposition of the operator $-\partial_{{\tt t}}^2-f'(v_\star({\tt t}))$, we know that the quadratic form
$$\mathcal{Q}(\phi):=\int_{\Sigma_\eps\times\R}\big(|\nabla_{\Sigma_\eps}\phi|^2+(\partial_{{\tt t}}\phi)^2-f'(v_\star({\tt t}))\phi^2\big) d\sigma({\tt y})d{\tt t}$$
is positive definite on the space
$$X:=\{\phi\in H^1(\Sigma_\eps\times\R):\, \int_\R \phi({\tt y},{\tt t}) v'_\star({\tt t}) d{\tt t}=0,\text{ for a. e. } {\tt y}\in \Sigma_\eps\},$$
namely there exists $C>0$ such that
\begin{eqnarray}
\mathcal{Q}(\phi)\ge C\|\phi\|^2_{H^1(\Sigma_\eps\times\R)}, \qquad\forall\,\phi\in X.\label{def_pos_H1}
\end{eqnarray}
As a consequence, by the Riesz representation theorem, there exist a unique $\phi\in X$ such that
\begin{eqnarray}
\int_{\Sigma_\eps\times\R}\big(\langle\nabla_{\Sigma_\eps}\phi,\nabla_{\Sigma_\eps}\psi\rangle_{\Sigma_\eps}
+\partial_{{\tt t}}\phi\partial_{{\tt t}}\psi-f'(v_\star({\tt t}))\phi\psi\big) d\sigma({\tt y})d{\tt t}=\int_{\Sigma_\eps\times\R} g\psi d\sigma({\tt y})d{\tt t}, \, \forall \psi\in X. \label{weak_sol}
\end{eqnarray}
In order for $\phi$ to be a true weak solution, (\ref{weak_sol}) has to be satisfied for any $\psi\in H^1(\Sigma_\eps\times\R)$. This can be verified using a test function in the space
$$X^\bot=\{a({\tt y})v'_\star({\tt t}):\,a\in H^1(\Sigma_\eps)\},$$
which fulfils (\ref{weak_sol}) due to the fact that $\phi$ is orthogonal to $v'_\star$ and $v'_\star$ is in the formal kernel of $\Delta_{\Sigma_\eps}+\partial^2_{{\tt t}}+f'(v_\star({\tt t}))$. We note that here $v'_\star$ is not in $H^1(\Sigma_\eps\times\R)$, however this is not a problem, since the product $a({\tt y})v'_\star({\tt t})$ is in that space provided $a\in H^1(\Sigma_\eps)$. Testing the equation with $\phi$ and using once again (\ref{def_pos_H1}), we have
$$\|\phi\|_{H^1(\Sigma_\eps\times\R)}\le C\|g\|_{L^2(\Sigma_\eps\times\R)}\le C(a,\gamma)\|gw_{a,\eps}({\tt y})\cosh^\gamma({\tt t})\|_{L^\infty(\Sigma_\eps\times\R)},$$
for some constant $C>0$. Therefore, by Sobolev embeddings and local elliptic estimates (see Theorem $8.8$ of \cite{GT}), we have, for any ${\tt x}_0:=({\tt y}_0,{\tt t}_0)\in \Sigma_\eps\times\R$,
$$\|\phi\|_{L^\infty(B_1({\tt x}_0))}\le C\|\phi\|_{W^{2,2}(B_1({\tt x}_0))}\le C(\|\phi\|_{L^2(B_2({\tt x}_0))}+\|g\|_{L^2(B_2({\tt x}_0))})\le C(a,\gamma)\|gw_{a,\eps}({\tt y})\cosh^\gamma({\tt t})\|_{L^\infty(\Sigma_\eps\times\R)},$$
where the constants are independent of ${\tt x}_0$, thus $\phi\in L^\infty(\Sigma_\eps\times\R)$ and (\ref{est_phi_infty}) is true. It is precisely here that we need the scaling, otherwise the constants in the elliptic estimates would depend on $\eps$.
\end{proof}

\begin{lemma}{[Decay in \tt{t}]}
Let $\alpha\in(0,1)$, $\gamma\in(0,\sqrt{2})$ and $a\in(0,\bar{a})$. Let $\phi\in H^1(\Sigma_\eps\times\R)$ be a bounded solution to (\ref{eq_near_lin_scaled}), with $g\in \mathcal{C}^{0,\alpha}_{a,\gamma}(\Sigma_\eps\times\R)$. Then $\phi\cosh^\gamma({\tt t})\in L^\infty(\Sigma_\eps\times\R)$, and there exists a constant $C(a,\gamma)>0$ depending on $a$ and $\gamma$ such that, for $\eps$ small enough,
$$\|\phi\cosh^\gamma({\tt t})\|_{L^\infty(\Sigma_\eps\times\R)}\le C(\|\phi\|_{L^\infty(\Sigma_\eps\times\R)}+\|g\cosh^\gamma({\tt t})\|_{L^\infty(\Sigma_\eps\times\R)}).$$\label{lemma_dec_t}
\end{lemma}
\begin{proof}
The proof relies on comparing the solution with a barrier. For any $\nu>0$ and $\lambda>0$, we set
$$v_{\lambda,\nu}({\tt y},{\tt t}):=(\lambda\cosh^{-\gamma}({\tt t})+\nu\cosh^\gamma({\tt t}))w_{\nu,\eps}({\tt y}).$$
The idea is to apply the maximum principle to bounded domain 
$$\Omega_{{\tt t}_1,{\tt t}_2,R_0}:=\Sigma_{\eps,\eps^{-1}R_0}\times({\tt t}_1,{\tt t}_2), \qquad\Sigma_{\eps,\eps^{-1}R_0}:=\Sigma_\eps\cap B_{\eps^{-1}R_0},$$
with $R_0>0,{\tt t}_1>0,{\tt t}_2>0$ to be determined. First we note that, in order to apply the maximum principle, $f'(v_\star({\tt t}))$ must be uniformly negative, thus we choose ${\tt t}_1>0$ such that $f'(v_\star({\tt t}))<-1-\frac{\gamma^2}{2}$ for $|{\tt t}|>{\tt t}_1$. Then we observe that, since $w_{\nu,\eps}\ge 1$,  for ${\tt y}\in\Sigma_\eps$ and ${\tt t}={\tt t}_1$ we have
$$\phi({\tt y},{\tt t}_1)\le\|\phi\|_{L^\infty(\Sigma_\eps\times\R)}\le \lambda \cosh^\gamma({\tt t}_1) \le v_{\lambda,\nu}({\tt y},{\tt t})$$
provided $\lambda\ge\|\phi\|_{L^\infty(\Sigma_\eps\times\R)}\cosh^{\gamma}({\tt t}_1)$. For $|{\tt t}|={\tt t}_2$, we have
$$\phi({\tt y},{\tt t}_2)\le \|\phi\|_{L^\infty(\Sigma_\eps\times\R)}\le \nu\cosh^\gamma({\tt t}_2)\le v_{\lambda,\nu}$$
provided ${\tt t}_2$ is large enough. Moreover, on $\partial \Sigma_{\eps,\eps^{-1}R_0}\times ({\tt t}_1,{\tt t}_2)$, we have
$$v_{\lambda,\nu}({\tt y},{\tt t})\ge (\lambda\cosh^{-\gamma}({\tt t}_2)+\nu\cosh^\gamma({\tt t}_1)) e^{\nu(R_0-C_\Sigma)}\ge \|\phi\|_{L^\infty(\Sigma_\eps\times\R)}\ge \phi({\tt y},{\tt t}),$$
provided $R_0>0$ is large enough, where $C_\Sigma>0$ is a constant depending just on $\Sigma$. In $\Omega_{{\tt t}_1,{\tt t}_2,R_0}$ the following differential inequality holds
\begin{multline*}
\big(-(\Delta_{\Sigma_\eps}+\partial^2_{{\tt t}})-f'(v_\star({\tt t}))\big)(\phi-v_{\lambda,\nu})\le -g-\lambda(1-\frac{\gamma^2}{2})v_{\lambda,\nu}\\
\le \|g\cosh^\gamma({\tt t})\|_{L^\infty(\Sigma_\eps\times\R)}\cosh^{-\gamma}({\tt t})-\lambda(1-\frac{\gamma^2}{2})v_{\lambda,\nu} \\
\le\big(\|g\cosh^\gamma({\tt t})\|_{L^\infty(\Sigma_\eps\times\R)}-\lambda(1-\frac{\gamma^2}{2})\big)\cosh^{-\gamma}({\tt t})
<0
\end{multline*}
provided $$\lambda\ge \frac{2\|g\cosh^\gamma({\tt t})\|_{L^\infty(\Sigma_\eps\times\R)}}{2-\gamma^2}.$$
In conclusion, applying the maximum principle and letting $\nu\to 0$, we have
$$\phi\le \lambda \cosh^{-\gamma}({\tt t}) \text{ in $\Sigma_\eps\times({\tt t}_1,\infty)$}, \qquad \lambda=\|\phi\|_{L^\infty(\Sigma_\eps\times\R)}\cosh^{\gamma}({\tt t}_1)
+\frac{2\|g\cosh^\gamma({\tt t})\|_{L^\infty(\Sigma_\eps\times\R)}}{2-\gamma^2}.$$
If $\phi$ satisfies the hypothesis of the Lemma, then also $-\phi$ and $\phi({\tt y},-{\tt t})$ do, hence the proof is concluded.
\end{proof}

\begin{lemma}{[Decay along the surface]}
Let $\alpha\in(0,1)$, $\gamma\in(0,\sqrt{2})$ and $a\in(0,\bar{a})$. Let $\phi\in H^1(\Sigma_\eps\times\R)$ be a bounded solution to equation (\ref{eq_near_lin_scaled}), with $g\in\mathcal{C}^{0,\alpha}_{a,\gamma}(\Sigma_\eps\times\R)$. Then $\phi w_{a,\eps}({\tt y})\cosh^\gamma({\tt t})\in L^\infty(\Sigma_\eps\times\R)$, and there exists a constant $C(a,\gamma)>0$ depending on $a$ and $\gamma$ such that, for $\eps$ small enough,
\begin{eqnarray}
\|\phi w_{a,\eps}({\tt y})\cosh^\gamma({\tt t})\|_{L^\infty(\Sigma_\eps\times\R)}\le C\|gw_{a,\eps}({\tt y})\cosh^\gamma({\tt t})\|_{L^\infty(\Sigma_\eps\times\R)}.\label{est_dec_infty}
\end{eqnarray}
\label{lemma_dec_y}
\end{lemma}
\begin{proof}
The proof of this result parallels the one of Lemma $3.5$ of \cite{DKPW}. Thanks to Lemma \ref{lemma_dec_t}, which guarantees the exponential decay in ${\tt t}$, the auxiliary function
\begin{eqnarray}
\psi({\tt y}):=\int_\R \phi^2({\tt y},{\tt t}) d{\tt t}, \qquad\forall\, {\tt y}\in\Sigma_\eps
\end{eqnarray}
is well defined and bounded, since
\begin{equation}
\label{est_psi_bd}
\begin{aligned}
&\psi({\tt y})\le C(\gamma)\|\phi\cosh^\gamma({\tt t})\|_{L^\infty(\Sigma_\eps\times\R)}^2
\le C(\gamma)\big(\|\phi\|_{L^\infty(\Sigma_\eps\times\R)}^2+\|g\cosh^\gamma({\tt t})\|_{L^\infty(\Sigma_\eps\times\R)}^2\big)\\
&\le C(a,\gamma)\|gw_{a,\eps}({\tt y})\cosh^\gamma({\tt t})\|^2_{L^\infty(\Sigma_\eps\times\R)}.
\end{aligned}
\end{equation}
Direct differentiation yields
$$\Delta_{\Sigma_\eps}\psi({\tt y})=\int_\R \big(2|\nabla_{\Sigma_\eps}\phi|^2+2\phi\Delta_{\Sigma_\eps}\phi\big) d{\tt t}.$$
Testing the equation with $\phi$ and using the decay in ${\tt t}$, we have
$$\int_\R |\nabla_{\Sigma_\eps}\phi|^2 d\sigma({\tt y})d{\tt t}=\int_\R(\partial_{{\tt t}}\phi)^2d{\tt t}-\int_\R f'(v_\star)\phi^2 d{\tt t}+\int_\R g\phi d{\tt t},$$
thus, by the spectral properties of $-\partial_{{\tt t}}^2-f'(v_\star({\tt t}))$, 
$$\Delta_{\Sigma_\eps}\psi({\tt y})\ge 2\int_\R|\nabla_{\Sigma_\eps}\phi|^2 d{\tt t}+3\int_\R \phi^2 d{\tt t} +2\int_\R g\phi d{\tt t}.$$
Due to the H\"{o}lder inequality and the Young inequality, we can see that $\psi$ satisfies the following differential inequality
$$-\Delta_{\Sigma_\eps}\psi({\tt y})+2\psi({\tt y})\le \int_\R g^2({\tt y},{\tt t})d{\tt t}\le C(\gamma)w_{a,\eps}({\tt y})^{-2}\|gw_{a,\eps}({\tt y})\cosh^\gamma({\tt t})\|^2_{L^\infty(\Sigma_\eps\times\R)}.$$
As a consequence, using the function
$$\bar{\psi}_\nu({\tt y}):=\lambda\|g\cosh^\gamma({\tt t})w_{a,\eps}({\tt y})\|^2_{L^\infty(\Sigma_\eps\times\R)} w_{a,\eps}^{-2}({\tt y})+\nu w_{a,\eps}^2({\tt y})$$
as a barrier, with $\lambda>0$ large enough and $\nu>0$ arbitrarily small, by the maximum principle we can show that 
\begin{eqnarray}
\psi({\tt y})\le C(a,\gamma) \|g\cosh^\gamma({\tt t})w_{a,\eps}({\tt y})\|^2_{L^\infty(\Sigma_\eps\times\R)} w_{a,\eps}^{-2}({\tt y}).\label{est_psi_y}
\end{eqnarray}
In fact, (\ref{est_psi_y}) is trivially satisfied in the compact part $\mathbf{K}$ of $\Sigma$, due to (\ref{est_psi_bd}), and along the ends $w_{a,\eps}$ is explicitly expressed in terms of exponential functions. Therefore, by local elliptic estimates, for any ${\tt t}_0>0$ we have
$$|\phi({\tt y},{\tt t})|\cosh^\gamma({\tt t})w_{a,\eps}({\tt y})\le C(a,\gamma,{\tt t}_0)\|g\cosh^\gamma({\tt t})w_{a,\eps}({\tt y})\|_{L^\infty(\Sigma_\eps\times\R)},$$
for any $({\tt y},{\tt t})\in\Sigma_\eps\times(-{\tt t}_0,{\tt t}_0)$. In order to deal with the region $\{({\tt y},{\tt t})\in\Sigma_\eps\times\R:\, |{\tt t}|>{\tt t}_0\}$, we compare $\phi$ with the barrier
$$\bar{\phi}_\nu({\tt y},{\tt t}):=M\|g\cosh^\gamma({\tt t})w_{a,\eps}({\tt y})\|_{L^\infty(\Sigma_\eps\times\R)} \cosh^{-\gamma}({\tt t})w_{a,\eps}^{-1}({\tt y})+\nu\cosh^\gamma({\tt t})w_{a,\eps}({\tt y}),$$
with $\lambda>0$ large enough and $\nu>0$ arbitrarily small.
\end{proof}

\begin{lemma}
Let $\alpha\in(0,1)$, $\gamma\in(0,\sqrt{2})$ and $a\in(0,\bar{a})$. Let $g\in \mathcal{C}^{0,\alpha}_{a,\gamma}(\Sigma_\eps\times\R)$ be a function satisfying (\ref{cond_ort_scaled}). Then, for $\eps$ small enough, there exists a unique solution $\phi\in\mathcal{C}^{2,\alpha}_{a,\gamma}(\Sigma_\eps\times\R)$ to (\ref{eq_near_lin_scaled}) satisfying (\ref{cond_ort_scaled}). Moreover the solution satisfies the estimate
\begin{eqnarray}
\|\phi\|_{\mathcal{C}^{2,\alpha}_{a,\gamma}(\Sigma_\eps\times\R)}\le C\|g\|_{\mathcal{C}^{0,\alpha}_{a,\gamma}(\Sigma_\eps\times\R)}.\label{est_near_lin_scaled}
\end{eqnarray}
\label{lemma_lin_Sigma_scaled}
\end{lemma}
\begin{proof}
Existence and uniqueness follow from Lemma \ref{lemma_existence}. By Lemmas \ref{lemma_dec_t} and \ref{lemma_dec_y}, the solution is decaying both along $\Sigma_\eps$ and in the orthogonal direction, and satisfies the estimate (\ref{est_dec_infty}). In conclusion, local elliptic estimates yield that $\phi\in\mathcal{C}^{2,\alpha}_{a,\gamma}(\Sigma_\eps\times\R)$ and (\ref{est_near_lin_scaled}) is true.
\end{proof}
\begin{lemma}
Let $\alpha\in(0,1)$, $\gamma\in(0,\sqrt{2})$ and $a\in(0,\bar{a})$. Let $g\in \mathcal{E}^{0,\alpha}_{a,\gamma}(\Sigma\times\R)$ be a function satisfying (\ref{cond_ort}). Then, for $\eps$ small enough, there exists a unique solution $\phi:=\Psi^4_\eps(g)\in\mathcal{E}^{2,\alpha}_{a,\gamma}(\Sigma\times\R)$ to (\ref{eq_near_lin}) satisfying (\ref{cond_ort}) and
\begin{equation}
\|\Psi^4(g)\|_{\mathcal{E}^{2,\alpha}_{a,\gamma}(\Sigma\times\R)}\le C\eps^{1-\alpha}\|g\|_{\mathcal{E}^{0,\alpha}_{a,\gamma}(\Sigma\times\R)},
\end{equation}
for some constant $C>0$.
\end{lemma}
\begin{proof}
Once again, it is enough to use Lemma \ref{lemma_lin_Sigma_scaled} and a scaling argument. The size of the inverse is a consequence of (\ref{compare_norms_Sigma}).
\end{proof}

\subsubsection{The proof of Proposition \ref{prop_near_X}: a fixed point argument}
The proof is based on a fixed point argument. In fact, equation (\ref{eq_X}) is equivalent to the fixed point problem
$$\phi_0=\Psi^4_\eps\Pi_{\mathcal{X}}\mathcal{G}^6_\eps(h_0,\phi_0,\mathbf{d}).$$
Using the previous construction, it is possible to prove that, if $\alpha\in(0,1/2)$, the right-hand side defines a contraction on the ball
$$B_{\Lambda_4}:=\{\phi_0\in\mathcal{E}^{2,\alpha}_{a,\gamma}(\Sigma\times\R)
:\,\|\phi_0\|_{\mathcal{E}^{2,\alpha}_{a,\gamma}(\Sigma\times\R)}<\Lambda_4\eps^{3-\alpha}\},$$
thus we have existence of a solution.

\end{document}